\documentclass{article}

\usepackage[a4paper, top=2.5cm, bottom=2.5cm, left=2.5cm, right=2.5cm]{geometry}

\usepackage{amsthm}
\usepackage{tikz-cd} 
\usepackage{aliascnt} 

\newtheorem{Thm}{Theorem}[section]

\newaliascnt{Def}{Thm}
\newtheorem{Def}[Def]{Definition}
\aliascntresetthe{Def} 

\newaliascnt{Rem}{Thm}
\newtheorem{Rem}[Rem]{Remark}
\aliascntresetthe{Rem}

\newaliascnt{Lem}{Thm}
\newtheorem{Lem}[Lem]{Lemma}
\aliascntresetthe{Lem}

\newaliascnt{Cor}{Thm}
\newtheorem{Cor}[Cor]{Corollary}
\aliascntresetthe{Cor}

\newaliascnt{Prop}{Thm}
\newtheorem{Prop}[Prop]{Proposition}
\aliascntresetthe{Prop}

\usepackage{amsmath,amssymb,mathtools}
\usepackage{stmaryrd}  

\usepackage{tikz} 

\usepackage{xcolor} 

\usepackage[numbers]{natbib}
\newcommand{\thmref}[2]{\cite[Thm. #2]{#1}}  
\newcommand{\propref}[2]{\cite[Prop. #2]{#1}}  
\newcommand{\defref}[2]{\cite[Def. #2]{#1}}   
\newcommand{\chapref}[2]{\cite[Ch. #2]{#1}}   

\usepackage[colorlinks=true, citecolor=blue]{hyperref} 
\usepackage[capitalize, noabbrev]{cleveref} 

\crefname{Def}{Definition}{Definitions}
\crefname{Thm}{Theorem}{Theorems}
\crefname{Lem}{Lemma}{Lemmas}
\crefname{Rem}{Remark}{Remarks}
\crefname{Prop}{Proposition}{Propositions}
\crefname{Cor}{Corollary}{Corollaries}

\renewcommand{\eqref}[1]{\textup{\textcolor{red}{(\ref{#1})}}} 

\makeatletter
\newcommand{\pcite}[2][]{%
  \if\relax\detokenize{#1}\relax%
    \cite{#2}%
  \else%
    \cite[\protect{\hyperlink{page.#1}{#1}}]{#2}%
  \fi%
}

\newcommand{\smartcite}[2][]{%
  \if\relax\detokenize{#1}\relax%
    \cite{#2}%
  \else%
    \ifnum\pdf@strcmp{#1}{#1}=\z@
      \ifnum\pdf@strcmp{\detokenize{-}}{\detokenize{#1}}>0%
        \cite[pp.~#1]{#2}%
      \else%
        \cite[p.~#1]{#2}%
      \fi%
    \fi%
  \fi%
}
\makeatother

\usepackage{enumitem}

\newlist{Cenum}{enumerate}{1}
\setlist[Cenum,1]{
  label=(C\arabic*),
  leftmargin=*,
  labelindent=\parindent,   
  itemindent=!,             
  align=left,               
  widest=(C99),            
  start=1                   
}

\newlist{CRenum}{enumerate}{1}
\setlist[CRenum,1]{
  label=(CR\arabic*),
  leftmargin=*,
  labelindent=\parindent,   
  itemindent=!,             
  align=left,               
  widest=(CR99),            
  start=1                   
}

\newlist{CFenum}{enumerate}{1}
\setlist[CFenum,1]{
  label=(CF\arabic*),
  leftmargin=*,
  labelindent=\parindent,   
  itemindent=!,             
  align=left,               
  widest=(CF99),            
  start=1                   
}

\newlist{CMenum}{enumerate}{1}
\setlist[CMenum,1]{
  label=(CM\arabic*),
  leftmargin=*,
  labelindent=\parindent,   
  itemindent=!,             
  align=left,               
  widest=(CM99),            
  start=1                   
}

\newlist{tPenum}{enumerate}{1}
\setlist[tPenum,1]{
  label=(tP\arabic*),
  leftmargin=*,
  labelindent=\parindent,   
  itemindent=!,             
  align=left,               
  widest=(tP99),            
  start=0                   
}

\newlist{tPRenum}{enumerate}{1}
\setlist[tPRenum,1]{
  label=(tPR\arabic*),
  leftmargin=*,
  labelindent=\parindent,   
  itemindent=!,             
  align=left,               
  widest=(tPR99),            
  start=0                   
}

\newlist{tPFenum}{enumerate}{1}
\setlist[tPFenum,1]{
  label=(tPF\arabic*),
  leftmargin=*,
  labelindent=\parindent,   
  itemindent=!,             
  align=left,               
  widest=(tPF99),            
  start=0                   
}

\newlist{tPMenum}{enumerate}{1}
\setlist[tPMenum,1]{
  label=(tPM\arabic*),
  leftmargin=*,
  labelindent=\parindent,   
  itemindent=!,             
  align=left,               
  widest=(tPM99),            
  start=1                   
}

\newlist{tPDenum}{enumerate}{1}
\setlist[tPDenum,1]{
  label=(tPD\arabic*),
  leftmargin=*,
  labelindent=\parindent,   
  itemindent=!,             
  align=left,               
  widest=(tPD99),            
  start=1                   
}

\newlist{DCenum}{enumerate}{1}
\setlist[DCenum,1]{
  label=(DC\arabic*),
  leftmargin=*,
  labelindent=\parindent,   
  itemindent=!,             
  align=left,               
  widest=(DC99),            
  start=0                   
}

\newlist{DCRenum}{enumerate}{1}
\setlist[DCRenum,1]{
  label=(DCR\arabic*),
  leftmargin=*,
  labelindent=\parindent,
  itemindent=!,
  align=left,
  widest=(DCR99),
  start=1
}

\newlist{DCFenum}{enumerate}{1}
\setlist[DCFenum,1]{
  label=(DCF\arabic*),
  leftmargin=*,
  labelindent=\parindent,
  itemindent=!,
  align=left,
  widest=(DCF99),
  start=1
}

\newlist{DCMenum}{enumerate}{1}
\setlist[DCMenum,1]{
  label=(DCM\arabic*),
  leftmargin=*,
  labelindent=\parindent,
  itemindent=!,
  align=left,
  widest=(DCM99),
  start=1
}

\newlist{PLPenum}{enumerate}{1}
\setlist[PLPenum,1]{
  label=(PLP\arabic*),
  leftmargin=*,
  labelindent=\parindent,   
  itemindent=!,             
  align=left,               
  widest=(PLP99),            
  start=1                   
}

\newlist{PLPRenum}{enumerate}{1}
\setlist[PLPRenum,1]{
  label=(PLPR\arabic*),
  leftmargin=*,
  labelindent=\parindent,
  itemindent=!,
  align=left,
  widest=(PLPR99),
  start=1
}

\newlist{PLPMenum}{enumerate}{1} 
\setlist[PLPMenum,1]{
  label=(PLPM\arabic*),
  leftmargin=*,
  labelindent=\parindent,
  itemindent=!,
  align=left,
  widest=(PLPM99),
  start=1
}

\usepackage{mathtools}

\usepackage{graphicx}
\graphicspath{{images/}}

\newcommand{\bbK}{\mathbb{K}}

\newcommand{\bbN}{\mathbb{N}}

\newcommand{\calF}{\mathcal{F}}
\newcommand{\calA}{\mathcal{A}}
\newcommand{\calV}{\mathcal{V}}
\newcommand{\calB}{\mathcal{B}}
\newcommand{\calE}{\mathcal{E}}
\newcommand{\calS}{\mathcal{S}}
\newcommand{\calP}{\mathcal{P}}
\newcommand{\calC}{\mathcal{C}}

\newcommand{\calL}{\mathcal{L}}

\newcommand{\rmId}{\mathrm{Id}}

\newcommand{\Extd}{\mathrm{\mathbf{Extd}}}

\newcommand{\Fact}{\mathrm{\mathbf{Fact}}}
\newcommand{\Comp}{\mathrm{\mathbf{Comp}}}

\newcommand{\lh}{\leftharpoonup} 
\newcommand{\rh}{\rightharpoonup}
\newcommand{\lt}{\triangleleft}
\newcommand{\rt}{\triangleright}
\newcommand{\blt}{\blacktriangleleft}
\newcommand{\brt}{\blacktriangleright}

\newcommand{\olOmega}{\overline{\Omega}}

\newcommand{\CE}{\mathcal{CE}}

\newcommand{\TE}{\mathcal{TE}}
\newcommand{\DCE}{\mathcal{DCE}}
\newcommand{\DCH}{\mathcal{DCH}}
\newcommand{\DCF}{\mathcal{DCF}}
\newcommand{\DDM}{\mathcal{DDM}}

\newcommand{\CH}{\mathcal{CH}}
\newcommand{\TTH}{\mathcal{TH}}

\newcommand{\CF}{\mathcal{CF}}
\newcommand{\TF}{\mathcal{TF}}

\newcommand{\M}{\mathcal{M}}
\newcommand{\CDM}{\mathcal{CDM}}
\newcommand{\TDM}{\mathcal{TDM}}

\newcommand{\MP}{\mathcal{MP}}

\newcommand{\address}[1]{\def\theaddress{#1}}
\newcommand{\email}[1]{\def\theemail{#1}}
\newcommand{\mail}[1]{\href{mailto:#1}{\textcolor{black}{#1}}}

\title{Extending structures for the linear framework of Universal Algebra with applications to transposed Poisson  algebras}
\author{Changhuan Liu, Tingfeng Wan and Guodong Zhou}

\address{School of Mathematical Sciences, Key Laboratory of MEA (Ministry of Education), Shanghai Key Laboratory of PMMP,
  East China Normal University,
 Shanghai 200241,
   China}

\email{
  \mail{(Changhuan Liu) 51265500028@stu.ecnu.edu.cn}\\
  \mail{(Tingfeng Wan) 3186988636@qq.com}\\
  \mail{(Guodong Zhou) gdzhou@math.ecnu.edu.cn}
}

\date{\today}

\begin{document}

\maketitle

\begin{center}
\small
\theaddress\\
\texttt{\theemail}
\end{center}

	\renewcommand{\thefootnote}{\alph{footnote}}
	\setcounter{footnote}{-1} 
	\renewcommand{\thefootnote}{\alph{footnote}}
	\setcounter{footnote}{-1} 

\begin{abstract}
  We present a general study of extending structures in the linear framework of Universal Algebra,  which unifies many existing results and contains  applications to various  Poisson type  algebras including  transposed Poisson algebras. 
\end{abstract}

\tableofcontents

\section{Introduction}

Poisson brackets, differential operators and compatible nonassociative
products connect algebraic structures with Hamiltonian dynamics and
infinite-dimensional symmetry algebras. For an algebra of classical
observables, the Leibniz identity
\[
\{f,gh\}=\{f,g\}h+g\{f,h\}
\]
expresses the compatibility between multiplication and Hamiltonian
evolution. In the theory of Hamiltonian operators and Poisson brackets of
hydrodynamic type, compatibility conditions also lead to the products now
known as Novikov products
\cite{GD79,BN85}. These examples motivate the study of algebras carrying
several operations, together with methods for enlarging such algebras
while preserving their defining identities.

Transposed Poisson algebras, introduced by C. Bai, R. Bai, L. Guo and
Y. Wu \cite{BBGW23}, provide a natural setting for a related compatibility
between a commutative associative product and a Lie bracket. Their
defining relation is
\[
2c\cdot[a,b]=[c\cdot a,b]+[a,c\cdot b].
\]
This identity arises when one takes the commutator of the Novikov product
in a Novikov--Poisson algebra, and more generally of the pre-Lie product
in a pre-Lie Poisson algebra \cite{BBGW23}. Over a field of characteristic
zero, it says that every multiplication operator $M_c(a)=c\cdot a$ is a
$\tfrac12$-derivation of the underlying Lie algebra:
\[
M_c([a,b])=\tfrac12\bigl([M_c(a),b]+[a,M_c(b)]\bigr).
\]
Thus the multiplication constrains how the Lie algebra can be enlarged.
The transposed identity is distinct from the ordinary Poisson Leibniz
identity; its connection with Hamiltonian algebra comes through these
compatible products.

A concrete link with differential operators is supplied by a commutative
associative algebra $C$ with a derivation $d$. The operations
\[
f\circ g=f\cdot d(g),\qquad
[f,g]_d=f\cdot d(g)-g\cdot d(f)
\]
give a Novikov--Poisson algebra and a transposed Poisson algebra,
respectively \cite{BBGW23}. For example, take $C=\bbK[t,t^{-1}]$ over a
field of characteristic zero and $d=d/dt$. Direct computation gives
\[
[f\partial_t,g\partial_t]
   =(fg'-gf')\partial_t=[f,g]_d\partial_t,
\]
so the resulting Lie algebra is the algebra of Laurent polynomial vector
fields in one variable. This example places the four classes studied
below---commutative, differential commutative, pre-Lie Poisson and
transposed Poisson algebras---in a common setting: the additional product
and bracket encode information obtained from a derivation, while each
construction retains the original commutative multiplication.

Factorization gives a second motivation from mathematical physics.
Matched pairs describe mutual actions between two algebraic components,
and bicrossproduct constructions have been used to construct quantum
groups \cite{S90}. At the classical level, the work of Lu and Weinstein
\cite{LW90} relates double Lie groups, dressing transformations and
Poisson--Lie geometry. These constructions suggest studying compatible
mixed operations on a direct sum, together with the dependence of those
operations on a chosen decomposition. Our results address this problem
for multilinear algebraic operations; additional geometric or
coalgebraic structures must be specified separately in such applications.

The questions considered here can be formulated without choosing a
particular family of products or brackets. Given an algebra $A$ and a
vector space $V$, which algebra structures on $E=A\oplus V$ preserve
$A$ as a subalgebra? If both summands carry prescribed algebra structures,
which mixed operations make $E$ an algebra containing both as
subalgebras? Finally, once $E$ and $A$ are fixed, how can one describe and
classify all subalgebras complementary to $A$? These are, respectively,
the \textbf{Extending Structures Problem (ESP)}, the
\textbf{Factorization Problem (FP)} and the
\textbf{Classifying Complements Problem (CCP)}. For algebras with several
operations, the mixed terms must satisfy all the compatibility identities
simultaneously. In the differential and Poisson type examples, this
requires keeping track of the product, the derivation or bracket, and
the relations among them throughout the construction.

Agore and Militaru introduced the \textbf{ESP} as a common framework for
extension and factorization problems and developed unified products in
several algebraic settings; see
\cite{AM14a,AM11,AM14,AM14b,AM16,AM15,AM15g,AM19}.
Subsequent studies include Leibniz algebras \cite{AM13}, Lie and
associative conformal algebras \cite{HS17,H19a}, left-symmetric algebras
\cite{H19b}, Lie bialgebras \cite{H23}, $3$-Lie algebras \cite{Z22},
Rota--Baxter Lie algebras \cite{PZ24}, Gel'fand--Dorfman bialgebras
\cite{WH24}, Novikov--Poisson algebras \cite{BH22}, dendriform algebras
\cite{ZW25} and pre-Poisson algebras \cite{ZLS25}. Across these settings,
a recurring procedure is to resolve the operations into components,
translate the defining identities into compatibility equations, and
identify the changes of components induced by algebra isomorphisms.
The aim of the present paper is to formulate this procedure once for
multilinear operations of arbitrary finite arity and to apply it
explicitly to the Poisson type algebras described above.

We work in the linear framework of Universal Algebra. The underlying
$\bbK$-vector space is fixed, the additional operations of positive
arity are multilinear, and the defining identities are multilinear in
their variables. Distinguished constants are allowed and remain fixed
when a given subalgebra is extended. We use only the elementary language
of types, identities and varieties; a general reference is
\cite{BS81}. A related approach was developed by Wires
\cite{W23a,W23b} for extensions of multilinear module expansions.
There the starting point is a classical extension
$0\to I\to E\to Q\to0$, with a prescribed ideal and quotient.
Here the prescribed algebra is a subalgebra, and its chosen vector space
complement need not be an ideal or a subalgebra. The resulting
classification therefore also includes structures beyond classical
ideal extensions.

Our general results provide a common description and classification
for the three problems. For the \textbf{ESP}, fix
$\calA=(A,G)\in\calV$ and $E=A\oplus V$. Decomposing each operation by
its input and output summands produces an extending datum
$\Omega(\calA,V)$. We give the compatibility conditions under which
this datum defines a unified product $\calA\ltimes V$, and show that
every algebra structure extending $\calA$ is obtained in this way
(\cref{thm:universal-description}). We then determine the transformation
relations induced by isomorphisms fixing $A$ pointwise. Allowing arbitrary
induced automorphisms of $E/A\cong V$, or requiring the induced map to
be the identity, gives the two nonabelian second cohomology sets
$\mathcal H^2_{\sim}(\calA,V)$ and
$\mathcal H^2_{\approx}(\calA,V)$
(\cref{thm:universal-classification for sim,thm:universal-classification-equiv}).
These are classification sets; a group structure is not part of the
construction.

For the \textbf{FP}, requiring both summands to be prescribed subalgebras
leads to matched pairs and bicrossed products $\calA\bowtie\calB$
(\cref{thm:universal-factorization}). For the \textbf{CCP}, start from a
fixed factorization $E=A\oplus B$. Every vector space complement of $A$
is the graph of a unique linear map $r:B\to A$. We determine precisely
when this graph is a subalgebra and when two such subalgebras are
isomorphic. This gives the classification by deformation maps and
identifies the factorization index with the cardinality of the resulting
quotient set
(\cref{Thm:ccp-deformation-correspondence,Thm:ccp-classification}).
In this formulation, the mixed operations and the choice of complementary
subalgebra are both accessible through explicit equations.

The main concrete application is to transposed Poisson algebras. We
obtain the eight bilinear component maps of a unified product, their
compatibility and transformation relations, flag extending structures,
matched pairs, and the classification of complements. The commutative
case supplies the common multiplication data used in all four
applications. For differential commutative algebras, two additional linear
maps describe the extension of the derivation; for pre-Lie Poisson
algebras, six additional bilinear maps describe the pre-Lie product.
Moreover, the constructions $f\circ g=f\cdot d(g)$ and
$[f,g]=f\circ g-g\circ f$ lift explicitly to extending data
(\cref{Rem: PLP from DCA,Rem: tP from PLP specialization}). Thus the
construction of an enlarged algebra is compatible with passage from
differential data to the associated products and brackets. This
compatibility is the specific link between our general classification
method and the differential algebra examples motivating the paper.

The following diagram records these constructions. The commutator arrow
is taken in characteristic zero, and both paths preserve the underlying
commutative extending structure:
\begin{center}
\begin{tikzcd}[column sep=5em, row sep=3em]
\text{\shortstack{differential commutative\\extending structures}}
  \arrow[r, "{f\circ g=f\cdot d(g)}"]
  \arrow[d, "{\text{forget }d}"']
& \text{\shortstack{pre-Lie Poisson\\extending structures}}
  \arrow[d, "{[f,g]=f\circ g-g\circ f}"] \\
\text{\shortstack{commutative\\extending structures}}
& \text{\shortstack{transposed Poisson\\extending structures}}
  \arrow[l, "{\text{forget the bracket}}"]
\end{tikzcd}
\end{center}

The paper is organized as follows.
\Cref{sec: framework} sets up the linear framework and the three
classification problems.
\Cref{sec: Universal Algebra ESP} develops extending data, unified
products and their classification.
\Cref{sec: Universal Algebra FP} treats matched pairs and factorization,
and \cref{sec: Universal Algebra CCP} classifies complements by
deformation maps.
\Crefrange{sec: commutative algebras}{sec: pre-Lie Poisson alg} give the
applications to commutative, transposed Poisson, differential commutative
and pre-Lie Poisson algebras, respectively.

Throughout this paper, $\bbK$ denotes a field.
All vector spaces, linear maps, bilinear maps and related structures are
taken over $\bbK$. We write $\bbK^{\times}$ for its set of nonzero elements
and
$\overline{\{0,1\}^n}\coloneqq\{0,1\}^n\setminus\{(0,\dots,0)\}$.
Additional characteristic assumptions are stated where required.

\section{The linear framework of Universal Algebra}\label{sec: framework}

Let
$$
\calF=\bigsqcup_{n\in\bbN}\calF_n
$$
be a type, where elements of $\calF_n$ are understood as operation symbols of arity $n$, say, operations with $n$ inputs. 
An algebra of type $\calF$ is a pair
$$
\calA=(A,G),
$$
where $A$ is a set and
$$
G=\{\omega^\calA\mid \omega\in\calF\}
$$
is a family of operations on $A$.
For each $\omega\in\calF_n$ with $n\in\bbN_+$, the realization of $\omega$ is given by a map
$$
\omega^\calA\colon A^n\longrightarrow A,
$$
which is interpreted as an $n$-ary operation, 
whereas for $\omega\in\calF_0$, the realization $\omega^\calA$ is interpreted as a distinguished element of $A$.

Let $\calV$ be a variety of algebras of type $\calF$. 
By Birkhoff's theorem, $\calV$ is determined by a set $\Sigma$ of identities. 
Thus, if $\calA=(A,G)$ is an algebra of type $\calF$, then $\calA\in\calV$ means that, for every identity
$$
p(x_1,\dots,x_m)\approx q(x_1,\dots,x_m)
$$
in $\Sigma$, we have 
$$
p^\calA(a_1,\dots,a_m)=q^\calA(a_1,\dots,a_m)
$$
for all $a_1,\dots,a_m\in A$.

The \textbf{Extending Structures Problem (ESP)} can be stated in the following general form. 
Given an algebra $\calA=(A,G)\in \calV$ and a set $E\supseteq A$, describe and classify all families $H$ of operations on $E$ such that $ \calE=(E,H)\in \calV $ and $\calA$ is a subalgebra of $\calE$.

At this level of set-theoretical generality, however, the \textbf{ESP} is mainly a conceptual framework rather than an effective construction method.
Indeed, in an arbitrary variety, the operation symbols may have different arities, and requiring all identities in $\Sigma$ to hold on $E$ usually leads to a large and unwieldy system of compatibility conditions. 
Therefore, the above formulation does not, in general, yield a concrete procedure for constructing and classifying extending structures.

The problem will therefore be considered in a more tractable setting. 
From now on, the discussion takes place in a linear framework of Universal Algebra.
The underlying $\bbK$-vector space structure is regarded as fixed and is not included among the additional operation symbols in $\calF$. 
Thus $\calF$ records only the extra algebraic operations, such as products, brackets, derivations, or distinguished elements.

More precisely, if $n\in\bbN$, every $\omega\in\calF_n$ is interpreted on a $\bbK$-vector space $A$ as an $n$-linear map
$$
\omega^\calA:A^n\longrightarrow A,
$$
where by convention $A^0=\bbK$. 
Under this convention, a distinguished element $a_0\in A$ is identified with the linear map $\bbK\longrightarrow A$, $\lambda\longmapsto\lambda a_0$.
The identities in $\Sigma$ are assumed to be multilinear in their variables and may involve both operation symbols of positive arity and constant symbols.

In this restricted linear setting, the \textbf{ESP} can be reformulated as follows. 
Given an algebra
$$
\calA=(A,G)\in\calV,\quad G=\{\omega^\calA\mid \omega\in\calF\},
$$
and a vector space $E$ containing $A$ as a subspace, describe and classify all families $H$ of multilinear operations on $E$ such that
$$ \calE=(E,H)\in\calV,\quad H=\{\omega^{\calE}\mid \omega\in\calF\}, $$
and $\calA$ is a subalgebra of $\calE$.

To make the description explicit, a vector space complement $V$ of $A$ in $E$ is fixed, so that
$$ E=A\oplus V. $$
Equivalently, every element $e\in E$ can be written uniquely as
$$ e=a+x,\qquad a\in A,\ x\in V. $$
This choice does not change the ambient vector space $E$; it only provides coordinates in which the values of the operations can be decomposed into their $A$- and $V$-components. 
With respect to this decomposition, denote by 
$$ \Extd(E,\calA)_V $$ 
the set of all families 
$H$ of multilinear operations on $E=A\oplus V$ such that $\calE=(A\oplus V,H)\in\calV$ and $\calA$ is a subalgebra of $\calE$.

In order to classify the families $H$ of operations on $E=A\oplus V$, it remains to compare different elements of $\Extd(E,\calA)_V$ by means of suitable equivalence relations.
We first introduce the following terminology in an abstract form, so that it can be used for any $\calF$-algebras in $\calV$ whose underlying vector space is equipped with a fixed decomposition.

\begin{Def}\label{Def:stabilize-costabilize}
Let $E=A\oplus V$ be a direct sum of vector spaces.
Denote by $i:A\to E$ the canonical inclusion and by $\pi:E\to V$ the canonical projection associated to this decomposition. 
A linear map $\varphi:E\to E$ is said to \textbf{stabilize} $A$ if the left square
commutes, and to \textbf{costabilize} $V$ if the right square commutes in the following diagram:
$$
\begin{tikzcd}
A \arrow[r,"i"] \arrow[d,"\operatorname{Id}_A"'] &
E \arrow[r,"\pi"] \arrow[d,"\varphi"] &
V \arrow[d,"\operatorname{Id}_V"] \\
A \arrow[r,"i"'] &
E \arrow[r,"\pi"'] &
V .
\end{tikzcd}
$$
Equivalently, $\varphi$ stabilizes $A$ if $\varphi(a)=a$ for all $a\in A$, and costabilizes $V$ if $\pi\circ\varphi=\pi$.
\end{Def}

\begin{Def}\label{Def:equivalent-algebras-over-decomposition}
Let $E=A\oplus V$ be a direct sum of vector spaces, $\calA=(A,G)\in\calV$,  and $H_1,H_2\in \Extd(E,\calA)_V$. 
Denote
$$
\calE_1\coloneqq (E,H_1),\qquad \calE_2\coloneqq(E,H_2).
$$
We say that $H_1$ and $H_2$ are \textbf{equivalent}, denoted by
$$ H_1\sim H_2, $$
if there exists an isomorphism of $\calF$-algebras $ \varphi:\calE_1\to \calE_2 $ stabilizing $A$. 
If, in addition, $\varphi$ costabilizes $V$, then $H_1$ and $H_2$ are called \textbf{cohomologous}, denoted by 
$$ H_1\approx H_2. $$
\end{Def}

It is clear that the relations $\sim$ and $\approx$ are equivalence relations on 
$\Extd(E,\calA)_V$. 
The corresponding quotient sets are denoted by
$$
\Extd'(E,\calA)_V\coloneqq\Extd(E,\calA)_V/\sim,
\qquad
\Extd''(E,\calA)_V\coloneqq\Extd(E,\calA)_V/\approx.
$$

We now state two other closely related problems.

\textbf{The Factorization Problem (FP)}: 
Given two $\calF$-algebras $\calA=(A,G)$, $\calB=(B,Q)$ in $\calV$ and a vector space $E=A\oplus B$, describe and classify all families $H$ of operations on $E$ such that $\calE=(E,H)\in \calV$ and $\calA$, $\calB$ are both subalgebras of $\calE$.

\textbf{The Classifying Complements Problem (CCP)}: 
Given a subalgebra $\calA=(A,G)$ of a $\calF$-algebra $\calE=(E,H)$ in $\calV$, describe and classify all subalgebras $\calB=(B,Q)$ of $\calE$ such that $E=A\oplus B$, 
and determine the number of their isomorphism classes.

\section{Extending Structures Problem (ESP)}\label{sec: Universal Algebra ESP}

We now turn to the description part of \textbf{ESP} in the fixed decomposition $E=A\oplus V$. The guiding idea is to decompose the value of each operation according to the input pattern and the output component. 
In this way, an algebra structure on $E$ extending the given algebra $\calA$ can be encoded by a system of component maps between $A$ and $V$.

We  first add an observation concerning  nullary operation symbols.

In fact, when dealing with extending structures, if $\omega\in\calF_0$, the condition that $\calA$ is a subalgebra of $\calE$ forces
$$
\omega^\calE(1_\bbK)=\omega^\calA(1_\bbK).
$$
Thus nullary operations give no additional  information  in the extending datum. 
Hence, we can dispose of them.

\subsection{Extending structures and unified products}

Let $\calA=(A,G)$ be an $\calF$-algebra in $\calV$ and let $V$ be a $\bbK$-linear space. 

Put 
$$M_0\coloneqq A,\qquad M_1\coloneqq V.$$ 
For $n\geq 1$ and  each $ \varepsilon=(\varepsilon_1,\dots,\varepsilon_n)\in\overline{\{0,1\}^n}$, 
denote 
$$
M_\varepsilon\coloneqq M_{\varepsilon_1}\times \dots \times M_{\varepsilon_n}; 
$$
choose arbitrary $n$-linear maps for  $ \omega\in\calF_n $,
$$
\omega_{0,\varepsilon}:
M_\varepsilon=M_{\varepsilon_1}\times\cdots\times M_{\varepsilon_n}\longrightarrow A,
\qquad
\omega_{1,\varepsilon}:
M_\varepsilon=M_{\varepsilon_1}\times\cdots\times M_{\varepsilon_n}\longrightarrow V.
$$
This collection of multilinear maps $\omega_{0,\varepsilon}, \omega_{1,\varepsilon}$ with $ \omega\in\calF_n,\ \varepsilon\in\overline{\{0,1\}^n},\ n\in\bbN_+$ is called an \textbf{extending datum} of $\calA$ through $V$, denoted by
$$
\Omega(\calA,V).
$$

Starting from an extending datum $\Omega(\calA,V)$, we define the
corresponding operations on $A\times V$ as follows: for each $\omega\in\calF_0$, let
 $$
 \omega^{\calA\ltimes V}(1_\bbK)=(\omega^\calA(1_\bbK),0);
 $$
for each $\omega\in\calF_n$, $n\in\bbN_+$, we define an $n$-linear map on $A\times V$ by the following formula:
\begin{equation*}
\begin{aligned}
\omega^{\calA\ltimes V}\bigl((a_1,x_1),\dots,(a_n,x_n)\bigr)
=
\Bigg(
\omega^{\calA}(a_1,\dots,a_n)
+
\sum_{\varepsilon\in\overline{\{0,1\}^n}}
\omega_{0,\varepsilon}(z_{\varepsilon_1},\dots,z_{\varepsilon_n}),
\sum_{\varepsilon\in\overline{\{0,1\}^n}}
\omega_{1,\varepsilon}(z_{\varepsilon_1},\dots,z_{\varepsilon_n})
\Bigg),
\end{aligned}
\end{equation*}
where $a_1, \dots, a_n\in A, x_1, \dots, x_n\in V,$ and for each $i=1,\dots,n$,
$$
z_{\varepsilon_i}=
\begin{cases}
a_i, & \varepsilon_i=0,\\
x_i, & \varepsilon_i=1.
\end{cases}
$$
For a more compact notation, let
$$
p_\xi:A\times V\to M_\xi,\qquad \iota_\xi:M_\xi\to A\times V,\qquad \xi\in \{0,1\},
$$
be the canonical projections and inclusions. 
For $n\geq 1$, each $ \varepsilon=(\varepsilon_1,\dots,\varepsilon_n)\in\{0,1\}^n$ and $ \omega\in\calF_n $,
put  
$$
p_\varepsilon\coloneqq p_{\varepsilon_1} \times \cdots \times p_{\varepsilon_n} : (A\times V)^n\longrightarrow M_\varepsilon=M_{\varepsilon_1}\times \dots \times M_{\varepsilon_n}.
$$
For $\xi\in\{0,1\}$ and $\varepsilon\in\overline{\{0,1\}^n}$, set
$$
\widehat{\omega}_{\xi,\varepsilon}
\coloneqq
\iota_\xi\circ\omega_{\xi,\varepsilon}\circ p_\varepsilon.
$$
Also for $ \mathbf 0^n\coloneqq (0,\dots,0)\in\{0,1\}^n $,
set
$$
\widehat{\omega}^{\calA}
\coloneqq
\iota_0\circ\omega^\calA\circ p_{\mathbf 0^n}.
$$
Then, as an $n$-linear map from $(A\times V)^n$ to $A\times V$, the
operation defined above can be written as
$$
\omega^{\calA\ltimes V}
=
\widehat{\omega}^{\calA}
+
\sum_{\varepsilon\in\overline{\{0,1\}^n}}
\sum_{\xi\in\{0,1\}}
\widehat{\omega}_{\xi,\varepsilon}.
$$
In what follows, whenever no confusion can arise, we suppress the canonical
projections and inclusions and write 
$$
\omega^{\calA\ltimes V}
=
\omega^\calA
+
\sum_{\varepsilon\in\overline{\{0,1\}^n}}
\sum_{\xi\in\{0,1\}}
\omega_{\xi,\varepsilon}.
$$

Put
$$
K=K_\Omega(\calA,V)\coloneqq\{\omega^{\calA\ltimes V}\mid \omega\in\calF_n,\ n\in \bbN\}.
$$
Thus the extending datum $\Omega(\calA,V)$ gives a family $K$ of multilinear operations on the vector space $A\times V$. 
Note that for any $\omega\in\calF_n,\ n \geq 1$,  $\omega^{\calA\ltimes V}|_{A^n}=\omega^\calA$, 
where $A^n$ is considered as a canonical subspace of $(A\times V)^n$.  

Conversely, given  a family $K=\{\omega^{\calA\ltimes V}\mid \omega\in\calF_n,\ n\in \bbN\}$ of multilinear operations on the vector space $A\times V$ such that for any $\omega\in\calF_n,\ n\in \bbN$,  $\omega^{\calA\ltimes V}|_{A^n}=\omega^\calA$, one obtains an extending datum $\Omega_K(\calA,V)=\{\omega_{\xi, \varepsilon} \mid  \varepsilon\in\overline{\{0,1\}^n}, \xi \in \{0, 1\},  \omega\in\calF_n,\ n\in \bbN_+\}$.
Obviously, these two sets are in bijection via the above correspondences. 

However, an extending datum $\Omega(\calA,V)$ does not necessarily define an $\calF$-algebra structure on $A\times V$ belonging to $\calV$. Substituting the operations in $K$ into the identities in $\Sigma$ and decomposing the resulting identities into their $A$- and $V$-components yields a system of conditions on the multilinear maps contained in $\Omega(\calA,V)$. We denote this system by 
$$\Gamma(\calA,V)$$
 and call its members the \textbf{compatibility conditions}.

The extending datum $\Omega(\calA,V)$ is called an \textbf{extending structure} of $\calA$ through $V$ if it satisfies the compatibility conditions $\Gamma(\calA,V)$, or equivalently, if
$$
(A\times V,K)\in\calV.
$$
In this case, the algebra
$$
\calA\ltimes V=\calA\ltimes_{\Omega(\calA,V)}V\coloneqq(A\times V,K)
$$
is called the \textbf{unified product} of $\calA$ and $V$ associated to $\Omega(\calA,V)$.

Denote by 
$$\calS(\calA,V)$$
the set of all extending structures of $\calA$ through $V$ and by 
$$\calP(\calA, V)$$  
the set of all unified products of $\calA$ and $V$.

\begin{Thm}\label{thm:universal-description}
Let $\calA=(A,G)\in\calV$ and let $E$ be a vector space containing $A$ as a subspace. Fix a vector space complement $V$ of $A$ in $E$, so that $E=A\oplus V$. Then there exist    bijections
among $ 
\Extd(E,\calA)_V$,  $\calS(\calA,V)$  and $\calP(\calA, V)$. 
\end{Thm}

\begin{proof} The bijection between  $\calS(\calA,V)$  and $\calP(\calA, V)$ is already explained above.

The bijection between  $\Extd(E,\calA)_V$  and $\calP(\calA, V)$ follows easily from the correspondence between the internal direct sum $E=A\oplus V$ and the external direct sum $A\times V$. 
 
\end{proof}

Therefore, the description part of the \textbf{ESP} is reduced to describing
$\calS(\calA,V)$, or equivalently, to determining the explicit form of the compatibility conditions $\Gamma(\calA,V)$ on the multilinear maps contained in $\Omega(\calA,V)$. The precise form of $\Omega(\calA,V)$ and $\Gamma(\calA,V)$ depends on the particular type $\calF$ and the identities $\Sigma$, but the above procedure provides a common framework for the description part of the \textbf{ESP} in many familiar algebraic settings, as illustrated by the examples in this paper.


\subsection{The classification part of \textbf{ESP}}\label{subsec: UA-ESP-classification}

We now turn to the classification part of the \textbf{ESP} by means of transferring the two equivalence relations on $\Extd(E,\calA)_V$ to $\calS(\calA,V)$. 

By the bijections in \cref{thm:universal-description}, the equivalence relations $\sim$ and $\approx$ on $\Extd(E,\calA)_V$ induce equivalence relations on $\calS(\calA,V)$ as follows.

For $\Omega^1(\calA,V),\Omega^2(\calA,V)\in \calS(\calA,V)$, write $H_{\Omega^i}\in \Extd(E,\calA)_V,\ i=1,2$ the corresponding elements.  Write
$$
\calA\ltimes_i V\coloneqq \calA\ltimes_{\Omega^i(\calA,V)}V,
\quad
\calE_{\Omega^i}\coloneqq(E,H_{\Omega^i}),
\quad i=1,2.
$$

\begin{Def}
    We say that $\Omega^1(\calA,V)$ and $\Omega^2(\calA,V)$ are \textbf{equivalent}, and write
$$
\Omega^1(\calA,V)\sim\Omega^2(\calA,V)
$$
if and only if
$$
H_{\Omega^1}\sim H_{\Omega^2}.
$$
Similarly, we say that
$\Omega^1(\calA,V)$ and $\Omega^2(\calA,V)$ are \textbf{cohomologous}, and write
$$
\Omega^1(\calA,V)\approx\Omega^2(\calA,V)
$$
if and only if
$$
H_{\Omega^1}\approx H_{\Omega^2}.
$$
\end{Def}

By Definition~\ref{Def:equivalent-algebras-over-decomposition}, 
$H_{\Omega^1} \sim H_{\Omega^2}$ if and only if there is an isomorphism of $\calF$-algebras from $\calE_{\Omega^1}$ to $\calE_{\Omega^2}$ stabilizing $A$, 
which, by Theorem~\ref{thm:universal-description}, 
is equivalent to saying that there is an isomorphism of $\calF$-algebras from $\calA\ltimes_1 V$ to $\calA\ltimes_2 V$ stabilizing $A$.
Hence, to describe $\Omega^1(\calA,V)\sim\Omega^2(\calA,V)$, it suffices to describe all isomorphisms of $\calF$-algebras
$\psi:A\ltimes_1 V\to A\ltimes_2 V$  which stabilize $A$.

Similarly, $\Omega^1(\calA,V)\approx\Omega^2(\calA,V)$ if and only if 
there is an isomorphism of $\calF$-algebras from $\calA\ltimes_1 V$ to $\calA\ltimes_2 V$ stabilizing $A$ and costabilizing $V$. 

To describe these equivalence relations explicitly, we proceed in three steps: 
We firstly describe   linear maps 
$$\psi:A\times V\to A\times V $$ componentwisely,   then 
homomorphisms of $\calF$-algebras  $$\psi:A\ltimes_1 V\to A\ltimes_2 V $$ componentwisely,   and finally 
isomorphisms of $\calF$-algebras  $$\psi:A\ltimes_1 V\to A\ltimes_2 V $$ componentwisely.

For the first step, recall that  we   have already put
$$
M_0\coloneqq A,\qquad M_1\coloneqq V,
$$
and let
$$
p_\xi:A\times V\longrightarrow M_\xi,
\qquad
\iota_\xi:M_\xi\longrightarrow A\times V 
$$
be the canonical projections and inclusions, respectively for  
$\xi\in\{0,1\}$. For a linear map
$$
\psi:A\times V\longrightarrow A\times V,
$$
denote its block components by
$$
\psi_{\zeta,\xi}\coloneqq
p_\zeta\circ\psi\circ\iota_\xi
\in \operatorname{Hom}_{\bbK}(M_\xi,M_\zeta),
\qquad \xi,\zeta\in\{0,1\}.
$$

\begin{Lem}\label{Lem:linear-maps-correspondence}
There is a bijection between the set of all linear maps
$$
\psi:A\times V\longrightarrow A\times V
$$
which stabilize $A$ and the set of all systems
$$
(\psi_{0,0},\psi_{1,0},\psi_{0,1},\psi_{1,1}),
$$
where
$$
\psi_{0,0}=\rmId_A:A\longrightarrow A,\qquad
\psi_{1,0}=0:A\longrightarrow V,
$$
and
$$
\psi_{0,1}:V\longrightarrow A,\qquad
\psi_{1,1}:V\longrightarrow V
$$
are linear maps. 
Moreover, $\psi$ is a linear isomorphism if and only if $\psi_{1,1}$ is a linear
isomorphism, and $\psi$ costabilizes $V$ if and only if
$\psi_{1,1}=\rmId_V$.
\end{Lem}

Thus the second step reduces to determining when the pair $(\psi_{0,1},\psi_{1,1})$ is compatible with the algebra structures of $\calA\ltimes_1 V$ and $\calA\ltimes_2 V$.

For a linear map $\psi:A\times V\to A\times V$ stabilizing A, imposing that $\psi$ be a homomorphism of $\calF$-algebras
$$
\psi:\calA\ltimes_1 V\longrightarrow \calA\ltimes_2 V
$$
is equivalent to requiring that, for any $\omega \in \calF_n$, $n\in \bbN$,
\begin{align}\label{Eq: psi is a map of F algebras}
\psi\circ\omega^{\calA\ltimes_1 V}
=
\omega^{\calA\ltimes_2 V}\circ\psi^n: (\calA\ltimes_1 V)^n\longrightarrow \calA\ltimes_2 V,
\end{align}
with the convention that $\psi^0=\rmId_\bbK$.

Next we describe Equation~\eqref{Eq: psi is a map of F algebras} componentwisely.

To this end, we need an explicit construction of the extending structures $
\Omega^i(\calA,V), i=1,2$ via Theorem~\ref{thm:universal-description}.

For each $\omega\in\calF_0$, this condition is automatic, since both nullary
operations are fixed as $(\omega^\calA,0)$ and $\psi$ stabilizes $A$.

For $n\in\bbN_+$, each $\delta=(\delta_1,\ldots,\delta_n)\in\{0,1\}^n$ and $\omega\in\calF_n$,
put
$$
M_\delta\coloneqq M_{\delta_1}\times\cdots\times M_{\delta_n},
\qquad
\iota_\delta
\coloneqq
\iota_{\delta_1}\times\cdots\times\iota_{\delta_n}
:
M_\delta\longrightarrow (A\times V)^n.
$$
and define
$$
\omega^i_{\xi,\delta}
\coloneqq
p_\xi\circ \omega^{\calA\ltimes_i V} \circ\iota_\delta
:
M_\delta\longrightarrow M_\xi,
\qquad \xi\in\{0,1\},
$$
Hence, 
$$
\Omega^i(\calA,V)=\{\omega^i_{\xi, \delta} \mid  \delta\in\overline{\{0,1\}^n}, \xi \in \{0, 1\},  \omega\in\calF_n,\ n\in \bbN_+\}. 
$$
Note that for the pure $A$-input word $\mathbf 0^n=(0,\ldots,0)$, we have
$$
\omega^i_{0,\mathbf 0^n}=\omega^\calA:A^n\to A,
\qquad
\omega^i_{1,\mathbf 0^n}=0:A^n\to V.
$$

Now we express Equation~\eqref{Eq: psi is a map of F algebras} by using the maps $\omega^i_{\xi, \delta}$'s.

Components of $\psi^n$ can be made explicit. 
For $\varepsilon,\delta\in\{0,1\}^n$, define
$$
\psi_{\varepsilon,\delta}
\coloneqq
\psi_{\varepsilon_1,\delta_1}
\times\cdots\times
\psi_{\varepsilon_n,\delta_n}
:
M_\delta\longrightarrow M_\varepsilon.
$$
Then for $\varepsilon,\delta\in\{0,1\}^n$,  $$\psi_{\varepsilon,\delta}=  p_\varepsilon\circ\psi^n \circ  \iota_\delta.$$  
 
Since $\psi_{1,0}=0$, only the terms  $\psi_{\varepsilon,\delta}$ with
$$
\varepsilon\leq\delta
$$
can contribute, where  $\varepsilon\leq\delta$ means coordinatewise inequalities:
$$
\varepsilon_j\leq\delta_j,\qquad j=1,\ldots,n.
$$

Now for $n\in\bbN_+$ and   $\omega\in\calF_n$, the  condition \eqref{Eq: psi is a map of F algebras} 
$$\psi\circ\omega^{\calA\ltimes_1 V}
=
\omega^{\calA\ltimes_2 V}\circ\psi^n$$
is equivalent to  
$$p_\zeta \circ \psi\circ\omega^{\calA\ltimes_1 V}\circ \iota_\delta
=
p_\zeta \circ\omega^{\calA\ltimes_2 V}\circ\psi^n \circ \iota_\delta, \quad  \delta\in\{0,1\}^n, \zeta\in\{0,1\}$$
which can be written as
$$\sum_{\xi\geq \zeta} p_\zeta \circ \psi\circ \iota_{\xi}\circ p_\xi\circ \omega^{\calA\ltimes_1 V}\circ \iota_\delta
=\sum_{\varepsilon\leq\delta}
p_\zeta \circ\omega^{\calA\ltimes_2 V}\circ \iota_{\varepsilon}\circ p_\varepsilon\circ\psi^n \circ  \iota_\delta, \quad  \delta\in\{0,1\}^n, \zeta\in\{0,1\}$$
Finally the  condition \eqref{Eq: psi is a map of F algebras}   is equivalent to 
the family of identities: 
\begin{equation}\label{Eq: F algebra map componentwisely}
\sum_{\xi\geq \zeta}
\psi_{\zeta,\xi}\circ\omega^1_{\xi,\delta}
=
\sum_{\varepsilon\leq\delta}
\omega^2_{\zeta,\varepsilon}\circ\psi_{\varepsilon,\delta}: M_\delta\longrightarrow M_\zeta.
 \end{equation}
These  relations \eqref{Eq: F algebra map componentwisely} are called \textbf{transformation relations}. 
This finishes the second step.

For the third step,  it suffices to ask, furthermore, that $\psi_{1,1}$ is a linear isomorphism.

We have shown that  $$
\Omega^1(\calA,V)\sim\Omega^2(\calA,V)
$$ if and only if  there exist linear maps $$
\psi_{0,1}:V\to A,\qquad
\psi_{1,1}:V\to V
$$ such that $\psi_{1,1}$ is a linear isomorphism and that 
for  $n\in\bbN_+$, $\omega\in\calF_n$, $\delta\in\{0,1\}^n$ and $\zeta\in\{0,1\}$, 
the  transformation relation   \eqref{Eq: F algebra map componentwisely}  
holds (where $\psi_{0,0}$ is understood as $\rmId_A$). 
Denote the corresponding quotient set  by
$$
\mathcal{H}^2_{\sim}(\calA,V)\coloneqq \calS(\calA,V)/\sim.
$$

We have shown the following result: 
\begin{Thm}\label{thm:universal-classification for sim}
Let $\calA=(A,G)\in\calV$ and let $E$ be a vector space containing $A$ as a subspace. 
Fix a vector space complement $V$ of $A$ in $E$, so that $E=A\oplus V$. 
Then there exists a bijection
$$
\mathcal{H}^2_{\sim}(\calA,V)\longleftrightarrow \Extd'(E,\calA)_V.
$$
 
\end{Thm}

For the cohomologous case, we further require that $\psi$ costabilizes $V$.
By \cref{Lem:linear-maps-correspondence}, this is equivalent to
$$
\psi_{1,1}=\rmId_V.
$$

We have shown that  $$
\Omega^1(\calA,V)\approx\Omega^2(\calA,V)
$$ if and only if  there exists a linear map $$
\psi_{0,1}:V\to A $$ such that 
for  $n\in\bbN_+$, $\omega\in\calF_n$, $\delta\in\{0,1\}^n$ and $\zeta\in\{0,1\}$, 
the identity \eqref{Eq: F algebra map componentwisely}  
holds (where $\psi_{0,0}$ is understood as $\rmId_A$ and $\psi_{1,1}=\rmId_V$). 
Denote the corresponding quotient set by
$$\mathcal{H}^2_{\approx}(\calA,V)\coloneqq\calS(\calA,V)/\approx.
$$

\begin{Thm}\label{thm:universal-classification-equiv}
Let $\calA=(A,G)\in\calV$ and let $E$ be a vector space containing $A$ as a subspace. 
Fix a vector space complement $V$ of $A$ in $E$, so that $E=A\oplus V$. 
Then there exists a bijection 
$$
\mathcal{H}^2_{\approx}(\calA,V)\longleftrightarrow \Extd''(E,\calA)_V.
$$
\end{Thm}

Consequently, $\mathcal{H}^2_{\sim}(\calA,V)$ provides the theoretical classification of all families of operations
$$ H=\{\omega^\calE\mid \omega\in\calF\} $$
on $E=A\oplus V$ such that $\calE=(E,H)\in\calV$ and $\calA$ is a subalgebra of $\calE$, up to isomorphisms stabilizing $A$. The finer object $\mathcal{H}^2_{\approx}(\calA,V)$ gives the corresponding classification under isomorphisms which stabilize $A$ and costabilize the fixed complement $V$.

In concrete algebraic settings, the abstract compatibility conditions $\Gamma(\calA,V)$ on extending data and the transformation relations become explicit systems of equations on the corresponding component maps.

These two quotient sets $\mathcal{H}^2_{\sim}(\calA,V)$ and $\mathcal{H}^2_{\approx}(\calA,V)$ define two kinds of   nonabelian  second  cohomology in the linear setup of Universal Algebra,  
which may be investigated further. 

\section{Factorization Problem (FP)}\label{sec: Universal Algebra FP}

The Factorization Problem  (\textbf{FP}) is a special case of \textbf{ESP} in which the
complement is not merely a vector space but a prescribed algebra. Thus, for
$E=A\oplus B$, the pure $A$-part and the pure $B$-part of every
operation are already fixed. The only remaining data are the mixed components
between $A$ and $B$.

Let $\calA=(A,G)$ and $\calB=(B,Q)$ be two $\calF$-algebras in $\calV$, and let
$E$ be a vector space containing $A$ and $B$ as subspaces such that
$E=A\oplus B$.

An $\calF$-algebra $\calE=(E,H)$ in $\calV$ is said to \textbf{factorize} through $\calA$ and $\calB$ if both $\calA$ and $\calB$ are subalgebras of $\calE$. 
We denote by 
$$\Fact(\calA,\calB)$$
the set of all such families $H$ of multilinear operations on $E=A\oplus B$.

Recall that we already set 
$$ M_0=A,\qquad M_1=B, $$ and for $n\geq 1$ and 
$\varepsilon=(\varepsilon_1,\ldots,\varepsilon_n)\in\{0,1\}^n$, put
$$
M_\varepsilon=M_{\varepsilon_1}\times\cdots\times M_{\varepsilon_n}.
$$
We also denote
$$
\mathbf 0^n=(0,\ldots,0),\qquad
\mathbf 1^n=(1,\ldots,1).
$$
 
We  first add an observation concerning the restriction imposed by nullary operations.

Let $\omega\in\calF_0$. If an algebra
$\calE$ factorizes through $\calA$ and $\calB$, then
the value of $\omega^{\calE}$ belongs to both $A$ and $B$, since
both $\calA$ and $\calB$ are subalgebras of $\calE$.
Since $E=A\oplus B$, we have $A\cap B=\{0\}$, and hence
$$
\omega^{\calE}=0.
$$
Consequently,
$$
\omega^{\calA}=0=\omega^{\calB}.
$$
Thus, if some nullary operation is nonzero in either $\calA$ or
$\calB$, then
$$
\Fact(\calA,\calB)=\varnothing.
$$
In what follows, assume that, for every $\omega\in\calF_0$, one has
$$
\omega^\calA=0=\omega^\calB.
$$

\begin{Def}
A \textbf{matched pair} of  two $\calF$-algebras $\calA$ and $\calB$ in $\calV$ is a system
$$
(\calA,\calB,\olOmega(\calA,\calB)),
$$
consisting of the two prescribed $\calF$-algebras $\calA=(A,G)$ and $\calB=(B,Q)$  in $\calV$ together
with a family of mixed multilinear maps
$$
\olOmega(\calA,\calB)
=
\left\{
\omega_{\xi,\varepsilon}:M_\varepsilon\to M_\xi
\ \middle|\
\omega\in\calF_n,\ 
\varepsilon\in \overline{\{0,1\}^n}\setminus\{\mathbf 1^n\},\ 
\xi\in\{0,1\},\ 
n\geq 1
\right\}
$$
such that, after adjoining the prescribed pure parts
$$
\omega_{0,\mathbf 1^n}=0,\qquad
\omega_{1,\mathbf 1^n}=\omega^\calB,
$$
determined by the algebra structure of $\calB$,
the resulting extending datum $\Omega(\calA,B)$ is an extending structure of $\calA$ through $B$. Equivalently,
$$
\calA\ltimes_{\Omega(\calA,B)}B\in\calV.
$$
Then the associated unified product is denoted by
$$
\calA\bowtie_{\olOmega(\calA,\calB)}\calB
\coloneqq
\calA\ltimes_{\Omega(\calA,B)}B
$$
and is called the \textbf{bicrossed product} associated to the matched pair $(\calA,\calB,\olOmega(\calA,\calB))$.
The set of all matched pairs of $\calA$ and $\calB$ is denoted by
$$
\MP(\calA,\calB).
$$
\end{Def}

\begin{Thm}\label{thm:universal-factorization}
Let $\calA=(A,G)$, $\calB=(B,Q)$ be two $\calF$-algebras in $\calV$ such that $\omega^\calA=\omega^\calB=0$, $\forall \omega\in \calF_0$ and let $E$ be a vector space containing $A$ and $B$ as subspaces such that $E=A\oplus B$. Then
there exists a bijection
$$
\Fact(\calA,\calB)\longleftrightarrow \MP(\calA,\calB).
$$

\end{Thm}

\begin{proof}
Regard $B$ as the fixed vector space complement of $A$ in
$E=A\oplus B$. By \cref{thm:universal-description}, the algebra structures on $E$
extending $\calA$ are in bijection with the extending structures
\[
\Omega(\calA,B)\in\calS(\calA,B).
\]
Under this bijection, the additional requirement that $\calB$ be a
subalgebra of $E$ with its prescribed operations is equivalent, for every
$\omega\in\calF_n$ with $n\geq 1$, to
\[
\omega_{0,\mathbf 1^n}=0,
\qquad
\omega_{1,\mathbf 1^n}=\omega^{\calB}.
\]
Hence, after the pure $A$- and $B$-components have been fixed, the
remaining components are precisely the mixed components defining a matched
pair of $\calA$ and $\calB$. Therefore, the bijection in
\cref{thm:universal-description} restricts to the desired bijection
$$
\Fact(\calA,\calB)
\longleftrightarrow
\MP(\calA,\calB).
$$
\end{proof}

\begin{Rem}
Unlike \textbf{ESP}, the \textbf{FP} does not lead to any new quotient
classification.

In fact, in \textbf{FP}, given an $\calF$-algebra structure  $\calE=(E,H)$ in $\calV$ on the fixed vector space $E=A\oplus B$ with both factors $\calA$ and $\calB$ being  prescribed, an automorphism    of  $\calE=(E,H)$ stabilizing both factors $\calA$ and $\calB$  is necessarily the identity $\rmId_E$. 
Indeed, every $e\in E$ has a unique decomposition $e=a+x$ with $a\in A$ and $x\in B$, 
and such an automorphism $\varphi$ must satisfy
$$
\varphi(e)=\varphi(a+x)=a+x=e.
$$
Therefore, under this convention, the classification part of \textbf{FP} coincides with the description part. 
\end{Rem}

\section{Classifying Complements Problem (CCP)}\label{sec: Universal Algebra CCP}

In \textbf{ESP} and \textbf{FP}, the algebra structure on the ambient vector space is itself the object to be determined.
In the Classifying Complements Problem (\textbf{CCP}), the ambient algebra $\calE=(E,H)$ is fixed, and the unknown objects are the subalgebras of $\calE$ complementary to $\calA$.

\begin{Def}

Let $\calE=(E,H)$ be an $\calF$-algebra in $\calV$ and $\calA=(A,G)$ a subalgebra of $\calE$.
A subalgebra $\calB=(B,Q)$ of $\calE$ is called an $\calF$-\textbf{complement} of $\calA$ in $\calE$ if $E=A\oplus B$ as vector spaces.
We denote by
$$
\Comp(\calA,\calE)
$$
the set of all $\calF$-complements of $\calA$ in $\calE$.

\end{Def}

The number of isomorphism classes of $\calF$-complements of $\calA$ in
$\calE$ is called the \textbf{factorization index} of $\calA$ in $\calE$,
and is denoted by
$$
[\calE:\calA]^f
\coloneqq
\left|\Comp(\calA,\calE)/\cong\right|.
$$

The Classifying Complements Problem (\textbf{CCP}) asks to describe and
classify all $\calF$-complements of $\calA$ in $\calE$ and to compute the
factorization index $[\calE:\calA]^f$.

Note that once the algebra structure on $E$ is fixed,
a subalgebra structure on a subspace  is uniquely determined by its underlying vector space,  that is,  whether the subspace is closed under all operations from $\calE$.

Assume from now on that $\calA$ admits at least one $\calF$-complement in $\calE$.
Fix one such complement
$$
\calB=(B,Q)\in\Comp(\calA,\calE).
$$
Since both $\calA$ and $\calB$ are subalgebras of $\calE$, the algebra $\calE$ factorizes
through $\calA$ and $\calB$. Hence, by \cref{thm:universal-factorization}, this factorization determines a unique matched pair, denoted by
$$
(\calA,\calB,\olOmega^c(\calA,\calB)).
$$
We call it the \textbf{canonical matched pair} associated to the fixed $\calF$-complement $\calB$ of $\calA$ in $\calE$.

Fix the complement $B$ as a reference complement. Every vector space
complement $B'$ of $A$ in $E=A\oplus B$ is the graph of a unique linear
map $r:B\to A$, namely
$$
B'=\{r(x)+x\mid x\in B\}.
$$
Therefore, the problem is to determine for which linear maps $r:B\to A$ this graph is closed under all operations of $\calE$.  
Linear maps of this kind $r:B\to A$ will be called \textbf{deformation maps} and will be shown to correspond bijectively to $\calF$-complements of $\calA$ in $\calE$. 

By the observation concerning nullary operations in the preceding section, 
in the present factorization setting, nullary operation symbols are either absent or forced to be interpreted as the zero element in $\calE$.
Therefore, in what follows, we only consider operation symbols of positive arity.

We still need some notation.

Similarly as  before, put
$$ M_0\coloneqq A,\qquad M_1\coloneqq B. $$
Let
$$ \pi_\xi:E=A\oplus B\longrightarrow M_\xi, \qquad \xi\in\{0,1\}, $$
be the canonical projections associated to the decomposition $E=A\oplus B$.

For a linear map $r:B\to A$, write
$$ r_{0,1}\coloneqq r:M_1=B\longrightarrow M_0=A $$
for the same map. Define the graph embedding associated to $r$ by
$$ \lambda_r:B\longrightarrow E=A\oplus B, \qquad \lambda_r(x)\coloneqq r(x)+x. $$

\begin{Def}\label{Def:ccp-deformation-map}
Let $\calE=(E,H)$ be an $\calF$-algebra in $\calV$, $\calA=(A,G)$ a subalgebra of $\calE$,
and $\calB=(B,Q)$ a fixed $\calF$-complement of $\calA$ in $\calE$.
Let $(\calA,\calB,\olOmega^c(\calA,\calB))$ be the canonical matched pair associated to
$E=A\oplus B$.

For every $\omega\in\calF_n$ with $n\in\bbN_+$ and every $\xi\in\{0,1\}$,
define
$$ \omega^r_{\xi,\mathbf 1^n}
\coloneqq \pi_\xi\circ\omega^{\calE}\circ\lambda_r^n
: M_{\mathbf 1^n}\to E^n \to E \to M_\xi. $$
Thus $\omega^r_{0,\mathbf 1^n}$ and $\omega^r_{1,\mathbf 1^n}$ are respectively the $A$-component and the $B$-component of the value of $\omega^\calE$ on elements of the graph of
$r$.

A linear map $r:B\to A$ is called a \textbf{deformation map} of the canonical matched pair $(\calA,\calB,\olOmega^c(\calA,\calB))$ 
if, for every $\omega\in\calF_n$ with $n\in\bbN_+$, one has
\begin{equation*}
\omega^r_{0,\mathbf 1^n} = r_{0,1}\circ \omega^r_{1,\mathbf 1^n}.
\end{equation*}
The system of these equations is denoted by
$$ R(\olOmega^c(\calA,\calB);r), $$ and its members are called the \textbf{deformation equations}.
We denote by
$$ \M(\calA,\calB \mid \olOmega^c(\calA,\calB)) $$
the set of all deformation maps of the canonical matched pair.
\end{Def}

\begin{Thm}\label{Thm:ccp-deformation-correspondence}
Let $\calE=(E,H)$ be an $\calF$-algebra in $\calV$, $\calA=(A,G)$ a subalgebra of $\calE$, $\calB=(B,Q)$ a fixed $\calF$-complement of $\calA$ in $\calE$
and $(\calA,\calB,\olOmega^c(\calA,\calB))$ the canonical matched pair associated to $E=A\oplus B$.
\begin{enumerate}[label=(\alph*)]
\item If $r\in\M(\calA,\calB\mid\olOmega^c(\calA,\calB))$, then
$$
B_r\coloneqq\{r(x)+x\mid x\in B\}
$$
is closed under all operations of $\calE$.
Hence the induced algebra $\calB_r$ on $B_r$ is an $\calF$-complement of $\calA$ in $\calE$.
\item Conversely, every $\calF$-complement $\calB'=(B',Q')$ of $\calA$ in $\calE$ is of
the form $\calB_r$ for a unique deformation map
$r\in\M(\calA,\calB\mid\olOmega^c(\calA,\calB))$.
\end{enumerate}
\end{Thm}

\begin{proof}
The deformation equations $R(\olOmega^c(\calA,\calB);r)$ are precisely the condition that the graph $B_r$ is closed under all operations.

Let $r\in\M(\calA,\calB\mid\olOmega^c(\calA,\calB))$. 
Since $E=A\oplus B$, the subspace $B_r=\{\,r(x)+x\mid x\in B\,\}$ is a vector space complement of $A$ in $E$.

Let $\omega\in\calF_n$ with $n\in\bbN_+$ and let $x_1,\ldots,x_n\in B$.
By the notations above, we have
$$
\omega^\calE(r(x_1)+x_1,\ldots,r(x_n)+x_n)
=
\omega^r_{0,\mathbf 1^n}(x_1,\ldots,x_n)
+
\omega^r_{1,\mathbf 1^n}(x_1,\ldots,x_n).
$$
Since $r$ satisfies $R(\olOmega^c(\calA,\calB);r)$, this becomes
$$
\omega^\calE(r(x_1)+x_1,\ldots,r(x_n)+x_n)
=
r_{0,1}\bigl(\omega^r_{1,\mathbf 1^n}(x_1,\ldots,x_n)\bigr)
+
\omega^r_{1,\mathbf 1^n}(x_1,\ldots,x_n)
$$
which falls into $B_r$. 
Hence $\calB_r$ is an $\calF$-complement of $\calA$ in $\calE$.

Conversely, let $\calB'=(B',Q')$ be an $\calF$-complement of $\calA$ in $\calE$.
Since $B'$ is a vector space complement of $A$ in $E$, 
there exists a unique linear map $r:B\to A$ such that
$$
B'=B_r=\{r(x)+x\mid x\in B\}.
$$

Since $\calB'$ is a subalgebra of $\calE$, the subspace $B_r$ is closed under all
$\omega^\calE$ for every $\omega\in\calF_n$ with $n\in\bbN_+$. For $x_1,\ldots,x_n\in B$,
  we have
$$
\omega^\calE(r(x_1)+x_1,\ldots,r(x_n)+x_n)
=
\omega^r_{0,\mathbf 1^n}(x_1,\ldots,x_n)
+
\omega^r_{1,\mathbf 1^n}(x_1,\ldots,x_n)\in B_r.
$$ This shows that 
$$
\omega^r_{0,\mathbf 1^n}
=
r_{0,1}\circ\omega^r_{1,\mathbf 1^n}
$$
for every $\omega\in\calF_n$ with $n\in\bbN_+$. Hence $r$ is a deformation map.
\end{proof}

Thus every $\calF$-complement of $\calA$ in $\calE$ is obtained as a deformation of the fixed complement $\calB$ via a deformation map.

To complete the classification,
it remains to identify when two deformation maps give rise to isomorphic complements. 
The subalgebra $\calB_r$ lives inside $E$, while the deformation maps all have the same domain $B$. 
For comparing different deformation maps, it is convenient to transport the algebra structure of $B_r$ back to the fixed vector space $B$.

In the rest of this section, write
$$
\olOmega^c \coloneqq \olOmega^c(\calA,\calB),\qquad \M^c(\calA,\calB) \coloneqq \M(\calA,\calB \mid \olOmega^c(\calA,\calB)).
$$

For a deformation map $r\in\M^c(\calA,\calB)$, denote by $\calB^r$ the algebra whose algebra structure is transported from the subalgebra $\calB_r\subseteq \calE$ to the fixed vector space $B$ via the graph embedding $\lambda_r$ defined above. Since its image is $B_r$, $\lambda_r$ is naturally viewed as the linear isomorphism
$$
\lambda_r:B\longrightarrow B_r.
$$
By the definition of the transported structure and the deformation equations, for every $\omega\in\calF_n$ with $n\in\bbN_+$, the
corresponding operation of $\calB^r$ is
$$
\omega^{\calB^r}=\omega^r_{1,\mathbf 1^n}:B^n\longrightarrow B.
$$

We now compare different deformation maps.

\begin{Def}\label{Def:equiv-deform-maps}
Let $r_1,r_2\in\M^c(\calA,\calB)$. We say that $r_1$ and $r_2$ are \textbf{equivalent},
and write
$$
r_1\sim r_2,
$$
if there exists a linear isomorphism
$$
\sigma:B\longrightarrow B
$$
such that, for every $\omega\in\calF_n$ with $n\in\bbN_+$, there holds
$$
\sigma\circ \omega^{r_1}_{1,\mathbf 1^n}
=
\omega^{r_2}_{1,\mathbf 1^n}\circ \sigma^{n}.
$$
Equivalently, $r_1\sim r_2$ if and only if $\calB^{r_1}$ and $\calB^{r_2}$ are
isomorphic as $\calF$-algebras.

We denote
$$
\mathcal H^2_{\mathrm{ccp}}(\calA,\calB\mid\olOmega^c)
\coloneqq
\M^c(\calA,\calB)/\sim .
$$
\end{Def}

\begin{Thm}\label{Thm:ccp-classification}
Let $\calE=(E,H)$ be an $\calF$-algebra in $\calV$, let $\calA=(A,G)$ be a subalgebra of $\calE$,
and assume that $\calA$ admits an $\calF$-complement $\calB=(B,Q)$ in $\calE$. Let
$
(\calA,\calB,\olOmega^c(\calA,\calB))
$
be the canonical matched pair associated to the factorization
$
E=A\oplus B.
$
Then there exists a bijection
$$
\mathcal H^2_{\mathrm{ccp}}(\calA,\calB\mid\olOmega^c)
\longrightarrow
\Comp(\calA,\calE)/\cong,
\qquad
[r]\longmapsto [\calB_r].
$$
Consequently,
$$
[\calE:\calA]^f
=
\left|
\mathcal H^2_{\mathrm{ccp}}(\calA,\calB\mid\olOmega^c)
\right|.
$$
\end{Thm}

\begin{proof}
By \cref{Thm:ccp-deformation-correspondence}, the assignment
$$
r\longmapsto \calB_r
$$
defines a bijection
$$
\M^c(\calA,\calB)
\longrightarrow
\Comp(\calA,\calE).
$$
It remains to pass to isomorphism classes.

To identify the equivalence relation on deformation maps, let
$$
r_1,r_2\in\M^c(\calA,\calB).
$$
The maps
$$
\lambda_{r_i}:B\longrightarrow B_{r_i},
\qquad
\lambda_{r_i}(x)=r_i(x)+x,
\qquad
i=1,2,
$$
are linear isomorphisms transporting the algebra structures of $\calB_{r_i}$ to
$\calB^{r_i}$. Hence $\calB_{r_1}$ and $\calB_{r_2}$ are isomorphic as $\calF$-algebras
if and only if $\calB^{r_1}$ and $\calB^{r_2}$ are isomorphic as $\calF$-algebras.
By \cref{Def:equiv-deform-maps}, this is equivalent to $r_1\sim r_2$.
So we have
$$
r_1\sim r_2
\quad\Longleftrightarrow\quad
\calB^{r_1}\cong\calB^{r_2}
\quad\Longleftrightarrow\quad
\calB_{r_1}\cong\calB_{r_2}.
$$

Therefore the assignment $r\mapsto \calB_r$ is compatible with the equivalence
relations $\sim$ and $\cong$, and hence descends to a bijection
$$
\mathcal H^2_{\mathrm{ccp}}(\calA,\calB\mid\olOmega^c)
\longrightarrow
\Comp(\calA,\calE)/\cong,
\qquad
[r]\longmapsto [\calB_r].
$$
Taking cardinalities gives
$$
[\calE:\calA]^f
=
\left|
\mathcal H^2_{\mathrm{ccp}}(\calA,\calB\mid\olOmega^c)
\right|.
$$
\end{proof}

Therefore, \textbf{CCP} in this linear framework of Universal Algebra is reduced to the study of deformation maps associated to the canonical matched pair obtained from a fixed factorization of $\calE$.
In concrete algebraic settings,
the deformation equations become explicit systems of equations on the corresponding component maps.


\section{Extending structures for commutative algebras}\label{sec: commutative algebras}

\crefrange{sec: commutative algebras}{sec: pre-Lie Poisson alg} apply the general theory developed in \crefrange{sec: Universal Algebra ESP}{sec: Universal Algebra CCP} to four concrete varieties in the present linear framework. 
In these applications, we use the general results of Sections~\ref{sec: Universal Algebra ESP}--\ref{sec: Universal Algebra CCP} without repeating their arguments, and focus instead on the explicit component data and the corresponding compatibility, transformation, and deformation formulas. 
Computations are included only when they clarify the concrete form of these conditions.

We begin the applications with commutative algebras, which provide the first and simplest concrete specialization of the general theory and the underlying commutative structure common to the later examples. Extending structures for associative algebras were studied in \cite{AM16} and \chapref{AM19}{4}, where the commutative case is also treated explicitly at the level of the extending datum and its compatibility conditions. Here we place the commutative case in the framework of \crefrange{sec: Universal Algebra ESP}{sec: Universal Algebra CCP} and record the resulting classification, flag extending structures, factorization, and classifying complements without repeating the general arguments.

Here
$$ \calF=\{-\cdot-\}, $$
where $-\cdot-$ is a binary operation symbol, and the defining
identities are
$$
(x\cdot y)\cdot z \approx x\cdot(y\cdot z),
\qquad
x\cdot y\approx y\cdot x.
$$
Throughout this section, by a commutative algebra we mean a (nonunital) commutative associative algebra.

\subsection{Extending structures and their classification}

We begin with the description part of the \textbf{ESP} for commutative algebras.

Let $\calC=(C,-\cdot-)$ be a commutative algebra, $E$ a vector space containing $C$ as a subspace, and $V$ a fixed complement of $C$ in $E$, so that $E=C\oplus V$. 
With this fixed decomposition, we retain from \cref{sec: framework} the notation $\Extd(E,\calC)_V$ for the set of all commutative algebra structures on $E$ containing $\calC$ as a subalgebra.

Denote the unique binary operation symbol by $\mu$. 
For this operation, the non-pure $C$-input patterns in the general extending
datum of \cref{sec: Universal Algebra ESP} are
$(1,0)$, $(0,1)$ and $(1,1)$.
The components corresponding to $(1,0)$ are written as
$$ \mu_{0,(1,0)}(x,a)=x\rt a, \qquad \mu_{1,(1,0)}(x,a)=x\lt a. $$
Commutativity determines the other mixed pattern:
$$ \mu_{0,(0,1)}(a,x)=x\rt a, \qquad \mu_{1,(0,1)}(a,x)=x\lt a. $$
For the pure $V$-input pattern, put
$$ \mu_{0,(1,1)}=f, \qquad \mu_{1,(1,1)}=\ell. $$
Thus commutativity reduces the six binary components of a general extending
datum from \cref{sec: Universal Algebra ESP} to the four independent maps in the following
definition; in particular, the pure $V$-input components $f$ and $\ell$ must
be symmetric.

\begin{Def}\label{Def: commutative extending structure}
A \textbf{commutative extending datum} of $\calC$ through $V$ is a quadruple
$$
\Omega(\calC,V)=(\lt,\rt,f,\ell)
$$
of bilinear maps
\begin{align*}
  \lt:V\times C\to V, \qquad \rt:V\times C\to C, \qquad
  f:V\times V\to C, \qquad \ell:V\times V\to V,  
\end{align*}
where $f$ and $\ell$ are symmetric. Given such a datum, define a bilinear
multiplication on $C\times V$ by
\begin{equation}\label{eq: commutative unified product}
(a,x)*(b,y)
=
\bigl(
a\cdot b+x\rt b+y\rt a+f(x,y),
\ x\lt b+y\lt a+\ell(x,y)
\bigr).
\end{equation}
If this multiplication makes $C\times V$ a commutative algebra, then
$\Omega(\calC,V)$ is called a \textbf{commutative extending structure}. In
this case, the resulting commutative algebra
$$
\calC\ltimes V=\calC\ltimes_{\Omega(\calC,V)}V\coloneqq (C\times V,-*-)
$$
is called the associated \textbf{unified product}.
The set of all commutative extending structures is denoted by $$\CE(\calC,V).$$ 
\end{Def}

Since commutativity is already incorporated into
\cref{Def: commutative extending structure}, it remains only to impose
associativity. Decomposing the associativity condition for
\eqref{eq: commutative unified product} into its $C$- and $V$-components
gives precisely the specialization of $\Gamma(\calC,V)$ recorded in
\ref{C1}--\ref{C6} below.

\begin{Thm}\label{Thm: commutative equivalent definition}
Let $\Omega(\calC,V)=(\lt,\rt,f,\ell)$ be a commutative
extending datum. Then \eqref{eq: commutative unified product} defines a
commutative algebra structure on $C\times V$ if and only if, for all
$a,b\in C$ and $x,y,z\in V$, the following conditions hold:
\begin{Cenum}
\item \label{C1}
$
x\rt(a\cdot b)
=(x\rt a)\cdot b+(x\lt a)\rt b;
$
\item \label{C2} $(V,\lt)$ is a nonunital right module over
$\calC$, that is,
$
x\lt(a\cdot b)=(x\lt a)\lt b;
$
\item \label{C3}
$
\ell(x,y)\rt a
=x\rt(y\rt a)
+f(x,y\lt a)-f(x,y)\cdot a;
$
\item \label{C4}
$
\ell(x,y)\lt a
=\ell(x,y\lt a)+x\lt(y\rt a);
$
\item \label{C5}
$
f(\ell(x,y),z)-f(x,\ell(y,z))
=x\rt f(y,z)-z\rt f(x,y);
$
\item \label{C6}
$
\ell(\ell(x,y),z)-\ell(x,\ell(y,z))
=x\lt f(y,z)-z\lt f(x,y).
$
\end{Cenum}
\end{Thm}

By \cref{thm:universal-description}, we obtain the following correspondence.

\begin{Thm}\label{Thm: first commutative correspondence theorem}
There is a bijection
$$
\Extd(E,\calC)_V\longleftrightarrow\CE(\calC,V).
$$
Under this bijection, $\Omega(\calC,V)=(\lt,\rt,f,\ell)$
corresponds to the multiplication
$$
(a+x)\cdot_E(b+y)
=\bigl(a\cdot b+x\rt b+y\rt a+f(x,y)\bigr)
+\bigl(x\lt b+y\lt a+\ell(x,y)\bigr),
$$
for all $a,b\in C$ and $x,y\in V$. 
Conversely, $-\cdot_E-$ corresponds to the extending structure $(\lt',\rt',f',\ell')$ whose four components are given by
\begin{equation}\label{eq: first commutative correspondence}
\begin{aligned}
x\rt' a&=p(x\cdot_E a),
&\qquad
x\lt' a&=x\cdot_E a-p(x\cdot_E a),\\
f'(x,y)&=p(x\cdot_E y),
&
\ell'(x,y)&=x\cdot_E y-p(x\cdot_E y),
\end{aligned}
\end{equation}
for all $a\in C$ and $x,y\in V$, where $p:E\to C,\ p(a+x)=a,$ is the canonical projection determined by the fixed decomposition $E=C\oplus V$.
\end{Thm}

Thus \cref{Thm: commutative equivalent definition,Thm: first commutative correspondence theorem}
identify the elements of $\Extd(E,\calC)_V$ with the commutative extending
data satisfying \ref{C1}--\ref{C6}.

We now turn to the classification part.
For this purpose, fix two commutative extending structures of $\calC$ through $V$,
$$ \Omega(\calC,V) = (\lt,\rt,f,\ell), \qquad
\Omega'(\calC,V) = (\lt',\rt',f',\ell'), $$
and denote their corresponding unified products, respectively, by
$$ \calC\ltimes V, \qquad 
\calC\ltimes'V. $$
By \cref{Lem:linear-maps-correspondence}, every linear map
$C\times V\to C\times V$ stabilizing $C$ is uniquely of the form
$$ \psi_{(r,v)}(a,x) = (a+r(x),v(x)), \qquad \forall\,a\in C,\ x\in V, $$
where $r:V\to C$ and $v:V\to V$ are linear maps. 
Specializing \eqref{Eq: F algebra map componentwisely} to the unique binary operation
yields the following.

\begin{Lem}\label{Lem: commutative transformation relations}
For linear maps $r:V\to C$ and $v:V\to V$, the map
$$
\psi_{(r,v)}:\calC\ltimes V\longrightarrow\calC\ltimes'V
$$
is a homomorphism of commutative algebras if and only if, for all
$a\in C$ and $x,y\in V$, one has
\begin{CRenum}
  \item\label{CR1}  $v(x\lt a)=v(x)\lt' a$;
  \item\label{CR2}  $r(x\lt a)=v(x)\rt' a+r(x)\cdot a-x\rt a$;
  \item\label{CR3}  $v(\ell(x,y))=\ell'(v(x),v(y))+v(x)\lt' r(y)+v(y)\lt' r(x)$;
  \item\label{CR4} $r(\ell(x,y))=f'(v(x),v(y))+v(x)\rt' r(y)+v(y)\rt' r(x)+r(x)\cdot r(y)-f(x,y)$.
\end{CRenum}
Moreover, $\psi_{(r,v)}$ is an isomorphism if and only if
$v:V\to V$ is a linear isomorphism, and it costabilizes $V$ if and only if
$v=\operatorname{id}_V$.
\end{Lem}

Together with the last assertion of \cref{Lem: commutative transformation relations}, conditions \ref{CR1}--\ref{CR4} give the explicit forms of the relations $\sim$ and $\approx$ as follows.

\begin{Def}\label{Def: commutative equivalence relations}
The commutative extending structures
$\Omega(\calC,V)$ and $\Omega'(\calC,V)$ are called
\textbf{equivalent}, denoted by
$$
\Omega(\calC,V)\sim\Omega'(\calC,V),
$$
if there exist a linear map $r:V\to C$ and a linear isomorphism
$v:V\to V$ satisfying \ref{CR1}--\ref{CR4}. They are called
\textbf{cohomologous}, denoted by
$$
\Omega(\calC,V)\approx\Omega'(\calC,V),
$$
if the same conditions hold with $v=\operatorname{id}_V$. Denote
$$
\CH^2_{\sim}(\calC,V)\coloneqq \CE(\calC,V)/\sim,
\qquad
\CH^2_{\approx}(\calC,V)\coloneqq \CE(\calC,V)/\approx.
$$
\end{Def}

Together with
\cref{thm:universal-classification for sim,thm:universal-classification-equiv},
these relations yield the following classification.

\begin{Thm}\label{Thm: second commutative correspondence theorem}
There are bijections
$$
\CH^2_{\sim}(\calC,V)\longleftrightarrow\Extd'(E,\calC)_V,
\qquad
\CH^2_{\approx}(\calC,V)\longleftrightarrow\Extd''(E,\calC)_V.
$$
\end{Thm}

\subsection{Flag extending structures}\label{subsec: commutative flag extending structures}
  
We now turn to flag extending structures, for which the extension can be studied through successive codimension-one steps.

\begin{Def}\label{Def: commutative flag extending structure}
  Let $\calC$ be a commutative algebra and $V$ a finite-dimensional vector space, 
  and denote by $E=C\oplus V$ their direct sum as vector spaces.
  A commutative algebra structure $-\cdot_E-$ on $E$ extending $\calC$ is called a \textbf{flag extending structure} of $\calC$ to $E$ if there exists a flag 
  $$0=V_0\subset V_1\subset \cdots \subset V_d=V$$
  of $V$ of length $d=\dim V$ 
  (i.e. the  $V_i$'s are subspaces of $V$ with $\dim(V_i)=i$ for $0\leq i\leq d$) 
  such that for each $1\leq i\leq d$, 
  $E_i=C\oplus V_i$ is closed under $-\cdot_E-$. 

\end{Def}

Throughout the remainder of this subsection, 
let $\calC=(C,-\cdot-)$ be a commutative algebra and let $V$ be a one-dimensional vector space. Fix a nonzero vector $x\in V$, so that $V=\bbK x$.
In this case, every commutative algebra structure on $E=C\oplus V$ extending $\calC$ is automatically a flag extending structure.
Since $V$ is one-dimensional, the four components of a commutative extending datum $ \Omega(\calC,V)=(\lt,\rt,f,\ell) $ are uniquely of the form
$$
x\lt a=\Lambda(a)x, \qquad x\rt a=\Delta(a), 
\qquad f(x,x)=a_0, \qquad \ell(x,x)=k_0x,
$$
where $\Lambda:C\to\bbK$ and $\Delta:C\to C$ are linear maps,
$a_0\in C$, and $k_0\in\bbK$. Under these substitutions, the compatibility
conditions \ref{C1}--\ref{C6} reduce precisely to
\ref{CF1}--\ref{CF4}.

\begin{Def}\label{Def: commutative flag datum}
  A \textbf{commutative flag datum} of $\calC$ is a $4$-tuple
  $$ \Omega(\calC)=(\Lambda,\Delta,a_0,k_0), $$
  consisting of two linear maps
  $$ \Lambda:C\to\bbK \qquad\text{and}\qquad \Delta:C\to C, $$
  together with two elements $a_0\in C$ and $k_0\in\bbK$, 
  satisfying the following compatibility conditions for all $a,b\in C$:
  \begin{CFenum}
  \item\label{CF1} $\Lambda(a\cdot b)=\Lambda(a)\Lambda(b)$;
  \item\label{CF2} $\Lambda(\Delta(a))=0$;
  \item\label{CF3} $\Delta(a\cdot b)=\Delta(a)\cdot b+\Lambda(a)\Delta(b)$;
  \item\label{CF4} $\Delta^2(a)=k_0\Delta(a)+a_0\cdot a-\Lambda(a)a_0$.
  \end{CFenum}
  We denote the set of all commutative flag data of $\calC$ by
  $$ \CF(\calC). $$

\end{Def}

Combining the preceding identification with \cref{Thm: first commutative correspondence theorem}, we obtain the following correspondence.

\begin{Prop}\label{Prop: commutative flag correspondence property}
Let $E$ be a vector space containing $C$ as a subspace and suppose that $E=C\oplus V$.
Then there are bijections
$$
\Extd(E,\calC)_V
\longleftrightarrow
\CE(\calC,V)
\longleftrightarrow
\CF(\calC).
$$
\end{Prop}

Under the right-hand bijection in \cref{Prop: commutative flag correspondence property}, 
the unified product corresponding to a commutative flag datum $\Omega(\calC)=(\Lambda,\Delta,a_0,k_0)$ is given by
$$
\begin{aligned}
(a,k_1x)*(b,k_2x)
=
\bigl(
a\cdot b+k_1\Delta(b)+k_2\Delta(a)+k_1k_2a_0,\,
(k_1\Lambda(b)+k_2\Lambda(a)+k_1k_2k_0)x
\bigr).
\end{aligned}
$$

We now turn to the classification of commutative flag data. Through the
same bijection, the relations $\sim$ and $\approx$ on $\CE(\calC,V)$ induce
corresponding relations on $\CF(\calC)$.

\begin{Def}\label{Def: commutative flag equivalence relations}
  Let
  $ \Omega(\calC),\ \Omega'(\calC) \in \CF(\calC) $
  be two commutative flag data and let $\Omega(\calC,V),\ \Omega'(\calC,V)$ be their corresponding commutative extending structures in $\CE(\calC,V)$. 
  We define
  \begin{align*}
    \Omega(\calC)\sim\Omega'(\calC) \Longleftrightarrow \Omega(\calC,V)\sim\Omega'(\calC,V),\qquad
    \Omega(\calC)\approx\Omega'(\calC) \Longleftrightarrow \Omega(\calC,V)\approx\Omega'(\calC,V).
  \end{align*}

\end{Def}

For the equivalence relation $\sim$, the pair $(r,v)$ in \cref{Lem: commutative transformation relations} is required to have $v:V\to V$ a linear isomorphism.
Since $V=\bbK x$, there exist $c_0\in C$ and $p_0\in\bbK^\times$ such that
$ r(x)=c_0$ and $v(x)=p_0x $.
Hence \ref{CR1}--\ref{CR4} reduce to the following conditions.

\begin{Prop}\label{Prop: commutative flag equivalence criterion}
Given two commutative flag data
$$
\Omega(\calC)=(\Lambda,\Delta,a_0,k_0),
\qquad
\Omega'(\calC)=(\Lambda',\Delta',a_0',k_0'),
$$
we have $\Omega(\calC)\sim\Omega'(\calC)$ if and
only if $\Lambda'=\Lambda$ and there exists a pair
$(p_0,c_0)\in\bbK^{\times}\times C$ such that, for all $a\in C$,
$$
\begin{aligned}
p_0\Delta'(a)
&=\Delta(a)+\Lambda(a)c_0-a\cdot c_0,\\
p_0^2a_0'
&=a_0+k_0c_0+c_0\cdot c_0
-2\Delta(c_0)-2\Lambda(c_0)c_0,\\
p_0k_0'
&=k_0-2\Lambda(c_0).
\end{aligned}
$$
Moreover, $\Omega(\calC)\approx\Omega'(\calC)$ if and only if the above
conditions hold with $p_0=1$; equivalently, $\Lambda'=\Lambda$ and there
exists $c_0\in C$ such that, for all $a\in C$,
$$
\begin{aligned}
\Delta'(a)
&=\Delta(a)+\Lambda(a)c_0-a\cdot c_0,\\
a_0'
&=a_0+k_0c_0+c_0\cdot c_0
-2\Delta(c_0)-2\Lambda(c_0)c_0,\\
k_0'
&=k_0-2\Lambda(c_0).
\end{aligned}
$$
\end{Prop}

Combining
\cref{Prop: commutative flag correspondence property,Def: commutative flag equivalence relations,Thm: second commutative correspondence theorem}
yields the following classification.

\begin{Thm}\label{Thm: commutative flag classification}
  Let $E$ be a vector space containing $C$ as a subspace and suppose that $E=C\oplus V$.
  There are bijections
  \begin{align*}
    \CF(\calC)/\sim\ \longleftrightarrow\ \CH^2_{\sim}(\calC,V)
    \ \longleftrightarrow\ \Extd'(E,\calC)_V, \\
    \CF(\calC)/\approx\ \longleftrightarrow\ \CH^2_{\approx}(\calC,V)
    \ \longleftrightarrow\ \Extd''(E,\calC)_V.
  \end{align*}

\end{Thm}

Thus
\cref{Prop: commutative flag correspondence property,Thm: commutative flag classification}
give the complete description and classification of commutative extending
structures with a one-dimensional complement; this is the basic step for
studying general flag extending structures.

\subsection{Factorization problem}

We now specialize the Factorization Problem developed in \cref{sec: Universal Algebra FP} to commutative algebras.

Let $\calC=(C,-\cdot-)$ and $\calB=(B,\ell)$ be two prescribed commutative algebras, 
and let $E$ be a vector space containing $C$ and $B$ as subspaces such that $E=C\oplus B$.
Following the general notation of \cref{sec: Universal Algebra FP}, denote by
$
\Fact(\calC,\calB)
$
the commutative algebra structures on $E$ factorizing through $\calC$ and
$\calB$.

Requiring $\calB$ to be a subalgebra with its prescribed multiplication
fixes the pure $B$-components as
$$
\mu_{0,(1,1)}=0,
\qquad
\mu_{1,(1,1)}=\ell.
$$
Thus, in a general commutative extending datum
$
(\lt,\rt,f,\ell),
$
one has $f=0$, while $\ell$ is the prescribed multiplication of $\calB$. 
Hence only the two mixed maps
$$
\lt:B\times C\to B,
\qquad
\rt:B\times C\to C
$$
remain to be determined.

\begin{Def}\label{Def: commutative matched pair}
Given two bilinear maps
$$
\lt:B\times C\to B,
\qquad
\rt:B\times C\to C
$$
the system
$$
\bigl(\calC,\calB,\olOmega(\calC,\calB)\bigr),\qquad 
\olOmega(\calC,\calB) = (\lt,\rt)
$$
is called a \textbf{matched pair of commutative algebras} if, for all
$a,b\in C$ and $x,y\in B$, the following conditions hold:
\begin{CMenum}
  \item \label{CM1} 
    $ x\rt(a\cdot b)= (x\rt a)\cdot b+(x\lt a)\rt b;$
  \item \label{CM2}
    $(B,\lt)$ is a nonunital right $\calC$-module;
  \item \label{CM3}
    $\ell(x,y)\rt a=x\rt(y\rt a);$
  \item \label{CM4}
    $\ell(x,y)\lt a=\ell(x,y\lt a)+x\lt(y\rt a).$
\end{CMenum}
The set of all such matched pairs is denoted by
$$ \MP(\calC,\calB). $$

For a matched pair $\bigl(\calC,\calB,\olOmega(\calC,\calB)\bigr)$,
let $ \Omega(\calC,B)=(\lt,\rt,0,\ell) $ be the corresponding commutative
extending structure.
The unified product of $\calC$ and $B$ associated to $\Omega(\calC,B)$ is denoted by
$$
\calC\bowtie_{\olOmega(\calC,\calB)}\calB 
\coloneqq\calC\ltimes_{\Omega(\calC,B)}B
=(C\times B,-*-)
$$
and is called the \textbf{bicrossed product} associated to the matched pair $\bigl(\calC,\calB,\olOmega(\calC,\calB)\bigr)$. 
Its multiplication is given by
$$
(a,x)*(b,y)
=
\bigl(
a\cdot b+x\rt b+y\rt a,\,
x\lt b+y\lt a+\ell(x,y)
\bigr)
$$
for all $a,b\in C$ and $x,y\in B$.
\end{Def}

Indeed, by \cref{Thm: commutative equivalent definition}, for the commutative extending datum of the form $(\lt,\rt,0,\ell)$, 
conditions \ref{C1}--\ref{C4} reduce precisely to \ref{CM1}--\ref{CM4}, respectively.
Condition \ref{C5} is automatic because $f=0$, while \ref{C6} reduces to
the associativity of $\ell$, which already holds because $\calB$ is a
prescribed commutative algebra.

Therefore, specializing \cref{thm:universal-factorization} to commutative algebras yields the following result.

\begin{Thm}\label{Thm:commutative-matched-pair-description}
There is a bijection
$$
\Fact(\calC,\calB)
\longleftrightarrow
\MP(\calC,\calB).
$$

\end{Thm}

For later use in the \textbf{CCP}, we record explicitly the matched pair
determined by a fixed factorization.

\begin{Rem}\label{Rem: comm canonical matched pair}
  Let $\calE=(E,-\cdot_E-)$ be a commutative algebra factorizing through the prescribed commutative algebras $\calC=(C,-\cdot-)$ and $\calB=(B,\ell)$.
  Then
  $ -\cdot_E- \in \Fact(\calC,\calB), $
  and hence, by \cref{Thm:commutative-matched-pair-description}, 
  it corresponds to a unique matched pair
  $$ \bigl(\calC,\calB,\olOmega^c(\calC,\calB)\bigr),\qquad 
  \olOmega^c(\calC,\calB) = (\lt,\rt).$$
  We call this the \textbf{canonical matched pair} associated to the
  factorization of $\calE$ through $\calC$ and $\calB$. 

  More explicitly, the decomposition $E=C\oplus B$ determines the canonical projection
  $$ p:E=C\oplus B\longrightarrow C, \qquad p(a+x)=a, $$
  and the mixed maps are given, for $a\in C$ and $x\in B$, by
  $$ x\rt a\coloneqq p(x\cdot_Ea), \qquad x\lt a\coloneqq x\cdot_Ea-p(x\cdot_Ea). $$

\end{Rem}

\subsection{Classifying complements problem}

Here the commutative algebra $\calE=(E,-\cdot_E-)$ and its subalgebra $\calC=(C,-\cdot-)$ are fixed, and we describe and classify the commutative complements of $\calC$ in $\calE$.
Following \cref{sec: Universal Algebra CCP}, write $\Comp(\calC,\calE)$ for the set of all commutative complements of $\calC$ in $\calE$, and $[\calE:\calC]^f$ for the factorization index.

Assume that $\calC$ admits at least one commutative complement in
$\calE$, and fix one such reference complement
$$
\calB=(B,\ell)\in\Comp(\calC,\calE).
$$
Then $E=C\oplus B$. By \cref{Rem: comm canonical matched pair}, this
factorization determines the canonical matched pair
$$ \bigl(\calC,\calB,\olOmega^c(\calC,\calB)\bigr),\qquad 
\olOmega^c(\calC,\calB) = (\lt,\rt).$$

Recall from \cref{sec: Universal Algebra CCP} that, relative to the fixed decomposition $E=C\oplus B$, every vector space complement $B'$ of $C$ in $E$ is uniquely the graph of a linear map $r:B\to C$, namely
$$
B'=B_r\coloneqq\{r(x)+x\mid x\in B\}.
$$
Thus, the description of commutative complements reduces to determining those linear maps $r:B\to C$ for which $B_r$ is closed under $-\cdot_E-$.

Using the decomposition $E=C\oplus B$ 
and the canonical matched pair $\bigl(\calC,\calB,\olOmega^c(\calC,\calB)\bigr)$,
we have, for $x,y\in B$, 
$$
\begin{aligned}
(r(x)+x)\cdot_E(r(y)+y)
=\bigl(
r(x)\cdot r(y)+x\rt r(y)+y\rt r(x)
\bigr)
+\bigl(
\ell(x,y)+x\lt r(y)+y\lt r(x)
\bigr),
\end{aligned}
$$
where the first and second parenthesized terms are respectively the $C$- and $B$-components. 
Thus $B_r$ is closed under $-\cdot_E-$ if and only if its $C$-component is obtained by applying $r$ to its $B$-component.
This yields the following explicit criterion.

\begin{Prop}\label{Prop: commutative graph deformation criterion}
  Let $r:B\to C$ be a linear map.
  Then the graph $B_r$ is closed under $-\cdot_E-$ if and only if
  \begin{equation}\label{eq: commutative deformation equation}
  r \bigl( \ell(x,y)+x\lt r(y)+y\lt r(x) \bigr)
  =r(x)\cdot r(y)+x\rt r(y)+y\rt r(x)
  \end{equation}
  for all $x,y\in B$.

  When this condition holds,
  $$\calB_r \coloneqq \bigl(B_r,-\cdot_E-|_{B_r\times B_r} \bigr)$$
  is a commutative complement of $\calC$ in $\calE$,
  and hence $\calB_r\in \Comp(\calC,\calE)$.
\end{Prop}

Equation \eqref{eq: commutative deformation equation} is precisely the
specialization to commutative algebras of the general deformation equations
in \cref{Def:ccp-deformation-map}.

\begin{Def}\label{Def:commutative-deformation map}
A linear map $r:B\to C$ satisfying
\eqref{eq: commutative deformation equation} is called a
\textbf{deformation map} of the canonical matched pair
$(\calC,\calB,\olOmega^c(\calC,\calB))$. The set of all such
maps is denoted by
$$
\CDM(\calC,\calB\mid\olOmega^c(\calC,\calB)).
$$
\end{Def}

Specializing \cref{Thm:ccp-deformation-correspondence} to commutative algebras yields the following correspondence.

\begin{Thm}\label{Thm: comm CCP correspondence}
The map
$$
\CDM(\calC,\calB\mid\olOmega^c(\calC,\calB))
\longrightarrow
\Comp(\calC,\calE),
\qquad
r\longmapsto\calB_r,
$$
is a bijection.
\end{Thm}

We now turn to the classification of complements up
to isomorphism of commutative algebras.
For $r\in\CDM(\calC,\calB\mid\olOmega^c(\calC,\calB))$, let
$$
\lambda_r:B\longrightarrow B_r,\qquad \lambda_r(x)=r(x)+x,
$$
be the canonical linear isomorphism.
Transport the commutative algebra structure of $\calB_r$ to the fixed vector space $B$ through $\lambda_r$,
and denote the resulting commutative algebra by
$$ \calB^r=(B,\ell^r), \qquad
\ell^r(x,y) = \ell(x,y)+x\lt r(y)+y\lt r(x),
$$
for all $x,y\in B$. 
By construction,
$ \lambda_r:\calB^r\longrightarrow\calB_r $
is an isomorphism of commutative algebras.

\begin{Def}\label{Def: commutative deformation equivalence}
Two deformation maps
$r_1,r_2\in
\CDM(\calC,\calB\mid\olOmega^c(\calC,\calB))$
are called \textbf{equivalent}, denoted by
$$
r_1\sim r_2,
$$
if and only if
$
\calB^{r_1}\cong\calB^{r_2}
$
as commutative algebras.
Equivalently, there exists a linear isomorphism
$
\sigma:B\to B
$
such that
$$
\sigma\bigl(\ell^{r_1}(x,y)\bigr)
=
\ell^{r_2}(\sigma(x),\sigma(y))
$$
for all $x,y\in B$.
After expanding the transported multiplication, this condition is equivalently
$$
\sigma(\ell(x,y))-\ell(\sigma(x),\sigma(y))
=\sigma(x)\lt r_2(\sigma(y))+\sigma(y)\lt r_2(\sigma(x))
-\sigma(x\lt r_1(y))-\sigma(y\lt r_1(x)).
$$

\end{Def}

The relation $\sim$ above is an equivalence relation on
$\CDM(\calC,\calB\mid\olOmega^c(\calC,\calB))$.
We denote the corresponding quotient set by
$$
\mathcal H^2_{\mathrm{ccp}}
(\calC,\calB\mid\olOmega^c(\calC,\calB))
\coloneqq
\CDM(\calC,\calB\mid\olOmega^c(\calC,\calB))/\sim.
$$
Since $\lambda_r:\calB^r\to\calB_r$ is an isomorphism, 
specialization of \cref{Thm:ccp-classification} yields the following classification.

\begin{Thm}\label{Thm: commutative CCP classification}
There is a bijection
$$
\mathcal H^2_{\mathrm{ccp}}
(\calC,\calB\mid\olOmega^c(\calC,\calB))
\longrightarrow
\Comp(\calC,\calE)/\cong,
\qquad
[r]\longmapsto[\calB_r].
$$
In particular, the factorization index is
$$
[\calE:\calC]^f 
=
\left|
\mathcal H^2_{\mathrm{ccp}}
(\calC,\calB\mid\olOmega^c(\calC,\calB))
\right|
=
\left|
\Comp(\calC,\calE)/\cong
\right|
.
$$
\end{Thm}

\section{Extending structures for transposed Poisson algebras}\label{sec: transposed Poisson algebras}

Throughout this section, $\bbK$ denotes a field of characteristic zero.  

In 2023, C. Bai,   R. Bai,  L. Guo,  and Y. Wu introduced the notion of transposed Poisson algebras:

\begin{Def}\defref{BBGW23}{1.1}\label{Def: transposed Poisson algebra}
  Let $L$ be a vector space equipped with two bilinear operations  
  $$-\cdot -: L \otimes L \to L \quad \mathrm{and}\quad   [-,-] : L \otimes L \to L.$$
  The triple $\calL=(L, -\cdot -, [-,-])$ is called a \textbf{transposed Poisson algebra} if $(L,-\cdot -)$ is a (nonunital) commutative algebra, 
  $(L, [-,-])$ is a Lie algebra, 
  and the two operations satisfy the following compatibility condition for any $a, b, c\in L$:
  \begin{align}\label{eq: transposed Poisson alg compatibility condition}
    2c\cdot [a, b] = [c\cdot a, b] + [a, c\cdot b]. 
  \end{align}

\end{Def}

Extending structures for related Poisson type structures were studied in several concrete cases: Jacobi and Poisson algebras \cite{AM15}, noncommutative Poisson algebras \cite{AM15g}, Lie bialgebras \cite{H23} and Gel'fand--Dorfman bialgebras \cite{WH24}.

For this variety,
$
\calF=\{-\cdot-,[-,-]\}
$
with both operation symbols of arity $2$.
Its defining identities are
\begin{gather*}
(x\cdot y)\cdot z\approx x\cdot(y\cdot z),\qquad x\cdot y\approx y\cdot x,\\
[x,y]\approx -[y,x],\qquad [x,[y,z]]+[y,[z,x]]+[z,[x,y]]\approx 0,\\
2z\cdot[x,y]\approx [z\cdot x,y]+[x,z\cdot y].
\end{gather*}

\subsection{Extending structures and their classification}\label{subsec: transposed Poisson extending structures and classification}

We begin with the description part of the \textbf{ESP} for transposed Poisson algebras.

Let $\calL=(L,-\cdot-,[-,-])$ be a transposed Poisson algebra, $E$ a vector
space containing $L$ as a subspace, and $V$ a complement of $L$ in $E$, so
that $E=L\oplus V$.
We retain from \cref{sec: framework} the notation $\Extd(E,\calL)_V$ for
the set of all transposed Poisson algebra structures on $E$ containing
$\calL$ as a subalgebra.

For the commutative product, the general binary extending datum specializes
to the four maps $(\lt,\rt,f,\ell)$ of
\cref{sec: commutative algebras}. For the Lie bracket, antisymmetry likewise
yields four maps $(\lh,\rh,\theta,\rho)$ as in \thmref{AM14}{2.2}, with
$\theta$ and $\rho$ antisymmetric. Thus the extending datum consists of the
following eight maps.

\begin{Def}\label{Def: transposed Poisson extending structure}
  A \textbf{transposed Poisson extending datum} of $\calL$ through $V$ is a system
  $$ \Omega(\calL, V) = (\lt, \rt, f, \ell, \lh, \rh, \theta, \rho)$$
  consisting of eight bilinear maps
  \begin{gather*}
    \lt : V \times L \to V, \qquad \rt : V \times L \to L, \qquad f : V \times V \to L, \qquad \ell : V \times V \to V, \\
    \lh : V \times L \to V, \qquad \rh : V \times L \to L, \qquad \theta : V \times V \to L, \qquad \rho : V \times V \to V
  \end{gather*}
  where $f$ and $\ell$ are symmetric, while $\theta$ and $\rho$ are
  antisymmetric.
  Given such a datum, define a multiplication and a bracket on $L\times V$
   by
  \begin{align*}
    (a,x)*(b,y) &\coloneqq (a\cdot b+x\rt b+y\rt a+f(x,y),\
    x\lt b+ y\lt a+ \ell (x,y)),\\
    \{(a,x),(b,y)\} &\coloneqq ([a,b]+x\rh b-y\rh a+\theta (x,y),\
    x\lh b- y\lh a+ \rho (x,y))
  \end{align*}
  for all $a,b\in L$ and $x,y\in V$. If these operations make $L\times V$
  a transposed Poisson algebra, then $\Omega(\calL,V)$ is called a
  \textbf{transposed Poisson extending structure} of $\calL$ through $V$. In
  this case, write
  $$
  \calL\ltimes V
  = \calL\ltimes_{\Omega(\calL,V)}V
  \coloneqq (L\times V,-*-,\{-,-\})
  $$
  for the resulting algebra and call it the \textbf{unified product} of
  $\calL$ and $V$ associated to $\Omega(\calL,V)$. We denote by
  $$\TE(\calL,V)$$ the set of all transposed Poisson extending structures of
  $\calL$ through $V$.
\end{Def}

Thus, to determine $\TE(\calL,V)$ explicitly, it remains to specialize the
general compatibility conditions $\Gamma(\calL,V)$ to transposed Poisson
algebras.

\begin{Thm}\label{Thm: transposed Poisson equivalent definition}
  Let $\Omega(\calL, V) = (\lt, \rt, f, \ell, \lh, \rh, \theta, \rho)$ be
  a transposed Poisson extending datum of $\calL$ through $V$. Then
  $\calL\ltimes_{\Omega(\calL,V)}V$ is a unified product if and only if the
  following compatibility conditions hold for all $a,b \in L$ and
  $x,y,z \in V$:
\begin{tPenum}
  \item \label{tP0} $(\lt, \rt, f, \ell)$ is a commutative extending structure and 
  $(\lh, \rh, \theta, \rho)$ is a Lie extending structure 
  (see \cref{Thm: commutative equivalent definition} and \thmref{AM14}{2.2});
    
  \item \label{tP1} $x\rh (a\cdot b)=2(x\lh a)\rt b+2(x\rh a)\cdot b-[x\rt b,a]-(x\lt b)\rh a$;
      
  \item \label{tP2} $x\lh (a\cdot b)=2(x\lh a)\lt b-(x\lt b)\lh a$;
      
  \item \label{tP3} $2\rho (x,y)\rt a=-2\theta (x,y)\cdot a+\theta (x\lt a,y)+\theta (x,y\lt a)+x\rh (y\rt a)-y\rh (x\rt a)$;
      
  \item \label{tP4} $2\rho (x,y)\lt a=\rho (x\lt a,y)+\rho (x,y\lt a)+x\lh(y\rt a)-y\lh (x\rt a)$;
      
  \item \label{tP5} $2x\rt [a,b]=[x\rt a,b]+[a,x\rt b]+(x\lt a)\rh b-(x\lt b)\rh a$;

  \item \label{tP6} $2x\lt [a,b]=(x\lt a)\lh b-(x\lt b)\lh a$;
        
  \item \label{tP7} $\ell (x,y)\rh a=2x\rt (y\rh a)+2f(x,y\lh a)+\theta (x\lt a,y)-y\rh (x\rt a)-[f(y,x),a]$;
      
  \item \label{tP8} $\ell (x,y)\lh a=2x\lt (y\rh a)+2\ell (x,y\lh a)+\rho (x\lt a,y)-y\lh (x\rt a)$;
      
  \item \label{tP9} $2f(x,\rho(y,z))=\theta (\ell (x,y),z)+\theta (y,\ell (x,z))+y\rh f(x,z)-z\rh f(x,y)-2x\rt \theta (y,z)$;
      
  \item \label{tP10} $2\ell (x,\rho (y,z))=\rho (\ell(x,y),z)+\rho (y,\ell (x,z))+y\lh f(x,z)-z\lh f(x,y)-2x\lt \theta (y,z)$.
\end{tPenum}

\end{Thm} 

\begin{proof}

  We already know that condition \ref{tP0} is equivalent to requiring
  $(L\times V,-*-)$ to be a commutative algebra and
  $(L\times V,\{-,-\})$ to be a Lie algebra. Hence, under \ref{tP0},
  $\calL\ltimes_{\Omega(\calL,V)}V$ is a transposed Poisson algebra if and
  only if
  $$
    2(a,x)*\{(b,y),(c,z)\}
    =
    \{(a,x)*(b,y),(c,z)\}
    +
    \{(b,y),(a,x)*(c,z)\}
  $$
  for all $a,b,c\in L$ and $x,y,z\in V$.

  Since this identity is trilinear and $L\times V$ is spanned by elements of
  the form $(a,0)$ and $(0,x)$, it suffices to check it on these generators.
  Moreover, interchanging the second and third inputs of this identity 
  changes both sides by
  a sign, so it is enough to consider the following six representative
  substitutions, listing first the $L$-component and then the $V$-component.

  \begin{enumerate}[label=(\arabic*), wide=0pt, leftmargin=*]
    \item For $((a,0),(b,0),(c,0))$, the identity holds because $\calL$ is a transposed Poisson algebra.

    \item For $((a,0),(b,0),(0,x))$, the identity is equivalent to
    \begin{align*}
      2(-a\cdot (x\rh b)-(x\lh b)\rt a,\ -(x\lh b)\lt a )
      =(&-x\rh (a\cdot b)-(x\lt a)\rh b + [b,x\rt a],\\
      &-x\lh (a\cdot b)-(x\lt a)\lh b ),
    \end{align*}
    Since $a$ and $b$ are arbitrary, exchanging their names gives precisely
    \ref{tP1} and \ref{tP2}.

    \item For $((a,0),(0,x),(0,y))$, the identity is equivalent to
    \begin{align*}
      \begin{split}
        2(a\cdot \theta (x,y)+\rho(x,y)\rt a, \rho (x,y)\lt a)
        =&(-y\rh (x\rt a)+\theta (x\lt a, y)+x\rh (y\rt a) + \theta (x,y\lt a),\\ 
        &-y\lh (x\rt a)+\rho (x\lt a, y)+ x\lh (y\rt a)+\rho(x, y\lt a)),
      \end{split}
    \end{align*}
    i.e., if and only if \ref{tP3} and \ref{tP4} hold.

    \item For $((0,x),(a,0),(b,0))$, the identity is equivalent to
    \begin{align*}
      \begin{split}
        2(x\rt[a,b], x\lt [a,b])=([x\rt a,b]+(x\lt a)\rh b+[a,x\rt b]-(x\lt b)\rh a,
        (x\lt a)\lh b-(x\lt b)\lh a),
      \end{split}
    \end{align*}
    i.e., if and only if \ref{tP5} and \ref{tP6} hold.

    \item For $((0,x),(a,0),(0,y))$, the identity is equivalent to
    \begin{align*}
      \begin{split}
        &2(-x\rt (y\rh a)-f(x,y\lh a), -x\lt (y\rh a)-\ell (x,y\lh a))\\
        =&(-y\rh (x\rt a)+\theta (x\lt a,y)+[a,f(x,y)]-\ell (x,y)\rh a,\\
        &-y\lh (x\rt a)+\rho (x\lt a,y)-\ell (x,y)\lh a),
      \end{split}
    \end{align*}
    i.e., if and only if \ref{tP7} and \ref{tP8} hold.

    \item For $((0,x),(0,y),(0,z))$, the identity is equivalent to
    \begin{align*}
      \begin{split}
        &2(x\rt \theta (y,z)+f(x,\rho (y,z)),x\lt \theta (y,z)+\ell (x,\rho (y,z)))\\
        =&(-z\rh f(x,y)+\theta (\ell (x,y),z)+y\rh f(x,z)+\theta (y,\ell (x,z)),\\
        &-z\lh f(x,y)+\rho (\ell (x,y),z)+y\lh f(x,z)+\rho (y,\ell(x,z))),
      \end{split}
    \end{align*}
    i.e., if and only if \ref{tP9} and \ref{tP10} hold.

  \end{enumerate}

\end{proof}

Combining \cref{Thm: transposed Poisson equivalent definition}
with \cref{thm:universal-description} yields the following correspondence.

\begin{Thm}\label{Thm: first transposed Poisson correspondence theorem}
There is a bijection
$$ \Extd(E,\calL)_V\longleftrightarrow\TE(\calL,V). $$
\end{Thm}

\begin{proof}
Associated to the fixed decomposition $E=L\oplus V$ are the canonical linear
isomorphism
$$
\phi:L\times V\to E,
\qquad
\phi(a,x)=a+x,
$$
and the canonical projection
$$
p:E\to L,
\qquad
p(a+x)=a.
$$
Given
$
\Omega(\calL,V)=(\lt,\rt,f,\ell,
\lh,\rh,\theta,\rho)
\in\TE(\calL,V),
$
the associated unified product on $L\times V$ is
$$
\calL\ltimes V=(L\times V,-*-,\{-,-\}).
$$
Transporting its operations to $E$ through $\phi$ gives an element
$(-\cdot_E-,[-,-]_E)\in\Extd(E,\calL)_V$ determined by
\begin{align*}
(a+x)\cdot_E(b+y)
&=\bigl(a\cdot b+x\rt b+y\rt a+f(x,y)\bigr)
+\bigl(x\lt b+y\lt a+\ell(x,y)\bigr),\\
[a+x,b+y]_E
&=\bigl([a,b]+x\rh b-y\rh a+\theta(x,y)\bigr)
+\bigl(x\lh b-y\lh a+\rho(x,y)\bigr),
\end{align*}
for all $a,b\in L$ and $x,y\in V$.

Conversely, given
$(-\cdot_E-,[-,-]_E)\in\Extd(E,\calL)_V$, the corresponding extending
structure $\Omega(\calL,V)\in\TE(\calL,V)$ is recovered through $p$ by the
eight component formulas
  \begin{align*}
    \begin{aligned}
      x\rt a &= p(x\cdot_{E} a), & \quad
      x\lt a &= x\cdot_{E} a-p(x\cdot_{E} a),\\
      f(x,y) &= p(x\cdot_{E} y), & \quad
      \ell(x,y) &= x\cdot_{E} y-p(x\cdot_{E} y),\\
      x\rh a &= p([x,a]_{E}), & \quad
      x\lh a &= [x,a]_{E}-p([x,a]_{E}),\\
      \theta(x,y) &= p([x,y]_{E}), & \quad
      \rho(x,y) &= [x,y]_{E}-p([x,y]_{E}).
    \end{aligned}
  \end{align*}

  These two constructions are inverse to each other.
\end{proof}

Thus,
\cref{Thm: transposed Poisson equivalent definition,Thm: first transposed Poisson correspondence theorem}
complete the description part of the \textbf{ESP} for transposed Poisson
algebras.

We now turn to the classification part.
For the remainder of this part, fix two transposed Poisson extending structures of $\calL$ through $V$,
$$
\Omega(\calL,V) = (\lt,\rt,f,\ell,\lh,\rh,\theta,\rho), \qquad
\Omega'(\calL,V) = (\lt',\rt',f',\ell',\lh',\rh',\theta',\rho')
$$
and denote their corresponding unified products, respectively, by
$$
\calL\ltimes V, \qquad
\calL\ltimes'V.
$$

By \cref{Thm: first transposed Poisson correspondence theorem}, the
relations $\sim$ and $\approx$ on $\Extd(E,\calL)_V$ induce corresponding
relations on $\TE(\calL,V)$. By
\cref{Lem:linear-maps-correspondence}, every underlying linear map
stabilizing $L$ is uniquely of the form
$$
\psi_{(r,v)}(a,x)=(a+r(x),v(x)),
\qquad
\forall\ a\in L,\ x\in V,
$$
where $r:V\to L$ and $v:V\to V$ are linear maps. 
Specializing the general transformation relations
\eqref{Eq: F algebra map componentwisely} therefore gives the following.

\begin{Lem}\label{Lem: transposed Poisson transformation relations}
  For linear maps $r:V\to L$ and $v:V\to V$, the map
  $$
  \psi_{(r,v)}:\calL\ltimes V\longrightarrow\calL\ltimes'V,
  $$
  is a homomorphism of transposed Poisson algebras if and only if the following conditions hold for all $a\in L$ and $x,y\in V$:
  \begin{tPRenum}
    \item \label{tPR0} $v(x\lt a)=v(x)\lt' a$;
    \item \label{tPR1} $r(x\lt a)=v(x)\rt' a+r(x)\cdot a-x\rt a$;
    \item \label{tPR2} $v(\ell(x,y))=\ell'(v(x),v(y))+v(x)\lt' r(y)+v(y)\lt' r(x)$;
    \item \label{tPR3} $r(\ell(x,y))=f'(v(x),v(y))+v(x)\rt' r(y)+v(y)\rt' r(x)+r(x)\cdot r(y)-f(x,y)$;
    \item \label{tPR4} $v(x\lh a)=v(x)\lh' a$;
    \item \label{tPR5} $r(x\lh a)=v(x)\rh' a+[r(x),a]-x\rh a$;
    \item \label{tPR6} $v(\rho(x,y))=\rho'(v(x),v(y))+v(x)\lh' r(y)-v(y)\lh' r(x)$;
    \item \label{tPR7} $r(\rho(x,y))=\theta'(v(x),v(y))+v(x)\rh' r(y)-v(y)\rh' r(x)+[r(x),r(y)]-\theta(x,y)$.
  \end{tPRenum}
  Moreover, $\psi_{(r,v)}$ is an isomorphism if and only if $v:V\to V$ is a linear isomorphism, 
  and it costabilizes $V$ if and only if $v=\rmId_V$.

\end{Lem}

\begin{proof}
  By \eqref{Eq: F algebra map componentwisely},
  $\psi_{(r,v)}$ preserves the multiplication if and only if
  \ref{tPR0}--\ref{tPR3} hold, and preserves the bracket if and only if
  \ref{tPR4}--\ref{tPR7} hold. The assertions concerning isomorphisms and
  costabilization follow from \cref{Lem:linear-maps-correspondence}.
\end{proof}

By \cref{Lem: transposed Poisson transformation relations}, the equivalence
relation $\sim$ is obtained by requiring $v:V\to V$ to be a linear
isomorphism. Solving \ref{tPR0}--\ref{tPR7} for the primed components using
$v^{-1}$ yields the following explicit description.

\begin{Def}
The transposed Poisson extending structures $\Omega(\calL,V)$ and $\Omega'(\calL,V)$ are called \textbf{equivalent}, denoted by
$$ \Omega(\calL,V)\sim\Omega'(\calL,V), $$
if there exist a linear map $r:V\to L$ and a linear isomorphism $v:V\to V$ 
such that $\Omega'(\calL,V)$ is obtained from $\Omega(\calL,V)$ via $(r,v)$ as follows:
  \begin{enumerate}[label=(\alph*)]
    \item $x\lt' a=v(v^{-1}(x)\lt a)$;
    \item $x\rt' a=v^{-1}(x)\rt a+r(v^{-1}(x)\lt a)-r(v^{-1}(x))\cdot a$;
    \item $ \begin{array}{rcl} 
    f'(x,y)&=&f(v^{-1}(x),v^{-1}(y))+r(\ell(v^{-1}(x),v^{-1}(y)))+r(v^{-1}(x))\cdot r(v^{-1}(y))\\
    &&-v^{-1}(x)\rt r(v^{-1}(y))-r(v^{-1}(x)\lt r(v^{-1}(y)))\\
    &&-v^{-1}(y)\rt r(v^{-1}(x))-r(v^{-1}(y)\lt r(v^{-1}(x)))
    \end{array} $;
    \item $\ell'(x,y)=v(\ell(v^{-1}(x),v^{-1}(y)))-v(v^{-1}(x)\lt r(v^{-1}(y)))-v(v^{-1}(y)\lt r(v^{-1}(x)))$;
    \item $x\lh' a=v(v^{-1}(x)\lh a)$;
    \item $x\rh' a=r(v^{-1}(x)\lh a)+v^{-1}(x)\rh a-[r(v^{-1}(x)),a]$;
    \item $\begin{array}{rcl} 
      \theta'(x,y)&=&r(\rho(v^{-1}(x),v^{-1}(y)))+[r(v^{-1}(x)),r(v^{-1}(y))]+\theta(v^{-1}(x),v^{-1}(y))\\
      &&-r(v^{-1}(x)\lh r(v^{-1}(y)))-v^{-1}(x)\rh r(v^{-1}(y))\\
      &&+r(v^{-1}(y)\lh r(v^{-1}(x)))+v^{-1}(y)\rh r(v^{-1}(x)) \end{array} $;
    \item $\rho'(x,y)=v(\rho(v^{-1}(x),v^{-1}(y)))-v(v^{-1}(x)\lh r(v^{-1}(y)))+v(v^{-1}(y)\lh r(v^{-1}(x)))$.
  \end{enumerate}

\end{Def}

For the cohomologous relation $\approx$, the isomorphism $\psi_{(r,v)}$ is further required to costabilize $V$. 
By \cref{Lem: transposed Poisson transformation relations}, this is equivalent to $v=\rmId_V$.
Hence the preceding formulas specialize as follows.

\begin{Def}
The transposed Poisson extending structures $\Omega(\calL,V)$ and $\Omega'(\calL,V)$ are called \textbf{cohomologous}, denoted by
$$ \Omega(\calL,V)\approx\Omega'(\calL,V), $$
if $\lt=\lt'$, $\lh=\lh'$, and there exists a linear map $r:V\to L$ such that:
  \begin{enumerate}[label=(\alph*)]
    \item  $x\rt'a=x\rt a+r(x\lt a)-r(x)\cdot a$;
    \item  $f'(x,y)=f(x,y)+r(\ell(x,y))+r(x)\cdot r(y)-x\rt r(y)-r(x\lt r(y))-y\rt r(x)-r(y\lt r(x))$;
    \item  $\ell'(x,y)=\ell(x,y)-x\lt r(y)-y\lt r(x)$;
    \item  $x\rh' a=x\rh a+r(x\lh a)-[r(x),a]$;
    \item  $\begin{array}{rcl} \theta'(x,y)&=&r(\rho(x,y))+[r(x),r(y)]+\theta(x,y)-x\rh r(y)-r(x\lh r(y))\\ &&+y\rh r(x)+r(y\lh r(x)) \end{array}$;
    \item  $\rho'(x,y)=\rho(x,y)-x\lh r(y)+y\lh r(x)$.
  \end{enumerate}

\end{Def}

Both $\sim$ and $\approx$ are equivalence relations on $\TE(\calL,V)$.
We denote the corresponding quotient sets by
$$ \TTH^{2}_{\sim}(\calL,V) \coloneqq \TE(\calL,V)/\sim, \qquad 
\TTH^{2}_{\approx}(\calL,V) \coloneqq \TE(\calL,V)/\approx. $$

Together with
\cref{thm:universal-classification for sim,thm:universal-classification-equiv},
these relations yield the following classification.

\begin{Thm}\label{Thm: second transposed Poisson correspondence theorem}
There are bijections
$$
\TTH^{2}_{\sim}(\calL,V)\longleftrightarrow \Extd'(E,\calL)_V,\qquad
\TTH^{2}_{\approx}(\calL,V)\longleftrightarrow \Extd''(E,\calL)_V.
$$
\end{Thm}

\cref{Thm: second transposed Poisson correspondence theorem}
gives the classification of transposed Poisson extending structures.
The finer relation $\approx$ admits a fixed-module refinement, since it
preserves the pair $(\lt,\lh)$. For any transposed Poisson extending
structure, \ref{tP0} implies that $(V,\lt)$ is a right module over the
commutative algebra $(L,-\cdot-)$ and $(V,\lh)$ is a right module over the
Lie algebra $(L,[-,-])$, while \ref{tP6} and \ref{tP2} give the two
compatibility identities below. This leads to the following definition.

\begin{Def}\label{Def: right transposed Poisson module}
A  \textbf{right transposed Poisson $\calL$-module} or just a right $\calL$-module is a vector space $V$ equipped with two bilinear maps
$$
\lt : V\times L \to V, \qquad \lh : V\times L \to V,
$$
such that $(V,\lt)$ is a (nonunital) right  module over the commutative algebra $(L, -\cdot -)$, $(V,\lh)$ is a right module over the Lie algebra $(L, [-,-])$, and the following compatibility conditions hold for all $a,b \in L$ and $x \in V$:
\begin{align*}
  2\, x \lt [a,b] &= (x\lt a)\lh b-(x\lt b)\lh a,\\
  x \lh (a \cdot b) &= 2\, (x\lh a)\lt b- (x\lt b)\lh a.
\end{align*}
\end{Def}

\begin{Def}\label{Def: fixed right moduled transposed Poisson extending structure}
Let $(V, \lt, \lh)$ be a right transposed Poisson $\calL$-module. A \textbf{$(\lt, \lh)$-transposed Poisson extending structure} of $\calL$ through $V$ is a 6-tuple 
$$(\rt,f, \ell, \rh, \theta, \rho)$$ 
such that
$(\lt, \rt, f, \ell, \lh, \rh, \theta, \rho)$ 
is a transposed Poisson extending structure of $\calL$ through $V$. Denote by 
$$\TE_{(\lt, \lh)}(\calL,V)$$ 
the set of all $(\lt, \lh)$-transposed Poisson extending structures of $\calL$ through $V$. 
Two elements 
$$(\rt, f, \ell, \rh, \theta, \rho), \; (\rt', f', \ell', \rh', \theta', \rho') \in \TE_{(\lt, \lh)}(\calL,V)$$ 
are called \textbf{$(\lt, \lh)$-cohomologous}, denoted by
$$
(\rt, f, \ell, \rh, \theta, \rho) \approx_{(\lt,\lh)} (\rt', f', \ell', \rh', \theta', \rho'),
$$ 
if $(\lt, \rt, f, \ell, \lh, \rh, \theta, \rho) \approx (\lt, \rt', f', \ell', \lh, \rh', \theta', \rho')$.
\end{Def}

The relation $\approx_{(\lt,\lh)}$ is an equivalence relation on
$\TE_{(\lt,\lh)}(\calL,V)$. We denote the corresponding quotient set
by
$$
\TTH^{2}_\approx(\calL,(V,\lt,\lh))
\coloneqq
\TE_{(\lt,\lh)}(\calL,V)/\approx_{(\lt,\lh)}.
$$

Allowing the fixed right transposed Poisson $\calL$-module structure on $V$ to vary yields the following decomposition.

\begin{Cor}\label{Cor: transposed Poisson right-module decomposition}
  There exists a bijection
  \begin{align*}
    \TTH^{2}_\approx(\calL,V)\longrightarrow \bigsqcup_{(\lt, \lh)}\TTH^{2}_\approx(\calL,(V, \lt, \lh)),\qquad
    [\Omega(\calL,V)]_{\approx}
    \longmapsto
    [(\rt,f,\ell,\rh,\theta,\rho)]_{\approx_{(\lt,\lh)}},
  \end{align*}
  where $\Omega(\calL,V)=(\lt,\rt,f,\ell,\lh,\rh,\theta,\rho)$, 
  and the disjoint union is taken over all right transposed Poisson $\calL$-module structures $(\lt,\lh)$ on $V$.
\end{Cor}

\begin{proof}
By \ref{tPR0} and \ref{tPR4}, the relation $\approx$ preserves the pair
$(\lt,\lh)$. Hence the relation $\approx$
restricts, for each fixed right transposed Poisson $\calL$-module structure
$(V,\lt,\lh)$, to the relation $\approx_{(\lt,\lh)}$ on
$\TE_{(\lt,\lh)}(\calL,V)$. The stated bijection follows.
\end{proof}

\subsection{Flag extending structures}

We now consider flag extending structures for transposed Poisson algebras.

\begin{Def}\label{Def: transposed Poisson flag extending structure}
Let $\calL$ be a transposed Poisson algebra and $V$ a finite-dimensional vector space,
and denote by $E=L\oplus V$ their direct sum as vector spaces.
A transposed Poisson algebra structure $(-\cdot_E-,[-,-]_E)$ on $E$ extending $\calL$ is called a \textbf{flag extending structure} of $\calL$ to $E$ if there exists a flag
$$
0=V_0\subset V_1\subset\cdots\subset V_d=V
$$
of $V$ of length $d=\dim V$
(i.e. the $V_i$'s are subspaces of $V$ with $\dim(V_i)=i$ for $0\leq i\leq d$)
such that, for each $1\leq i\leq d$, the subspace
$$
E_i\coloneqq L\oplus V_i
$$
is closed under both operations $-\cdot_E-$ and $[-,-]_E$.
\end{Def}

Throughout the remainder of this subsection,
let $\calL=(L,-\cdot-,[-,-])$ be a transposed Poisson algebra and $V$ a one-dimensional vector space.
Fix a nonzero vector $x\in V$, so that $V=\bbK x$.
In this case, every transposed Poisson algebra structure on $E=L\oplus V$ extending $\calL$ is automatically a flag extending structure.

Since $V$ is one-dimensional, the eight components of a transposed Poisson extending datum
$$
\Omega(\calL,V)=(\lt,\rt,f,\ell,\lh,\rh,\theta,\rho)
$$
are uniquely of the form
\begin{align*}
  \begin{aligned}
    x\lt a &=\Lambda(a)x, & \ x\rt a &=\Delta(a), & \
    f(x,x) &=a_0, & \ \ell(x,x) &=k_0x,\\
    x\lh a &=\lambda(a)x, & \ x\rh a &=D(a),  & \
    \theta &=0, & \ \rho &=0,
  \end{aligned}
\end{align*}
for all $a\in L$, where $ \Lambda,\lambda:L\to\bbK $ and $ \Delta,D:L\to L $ are linear maps, $a_0\in L$, and $k_0\in\bbK$.
Here $\theta$ and $\rho$ vanish because they are antisymmetric and $V$ is one-dimensional.

Substituting these expressions into \ref{tP0}--\ref{tP10} yields
\ref{tPF0}--\ref{tPF6} below: \ref{tP0} gives \ref{tPF0};
\ref{tP1}, \ref{tP2}, \ref{tP5}--\ref{tP8} give
\ref{tPF1}--\ref{tPF6}, respectively, while
\ref{tP3}, \ref{tP4}, \ref{tP9}, and \ref{tP10} hold automatically.

\begin{Def}\label{transposed Poisson flag datum}
A \textbf{transposed Poisson flag datum} of $\calL$ is a 6-tuple
  $$
\Omega(\calL)=(\Lambda,\Delta,a_0,k_0,\lambda,D),
  $$
consisting of linear maps
  $$
\Lambda,\lambda:L\to\bbK,
\qquad
\Delta,D:L\to L,
  $$
together with $a_0\in L$ and $k_0\in\bbK$, satisfying the following
compatibility conditions for all $a,b\in L$:
  \begin{tPFenum}
  \item \label{tPF0} $(\Lambda,\Delta,a_0,k_0)$ is a commutative flag
  datum of $(L,-\cdot-)$ in the sense of
  \cref{Def: commutative flag datum}, and $(\lambda,D)$ is a twisted derivation
  (see \defref{AM14}{4.2}) of the Lie algebra $(L,[-,-])$;
    \item \label{tPF1} $D(a\cdot b)=2\lambda(a)\Delta(b)+2D(a)\cdot b-[\Delta(b),a]-\Lambda(b)D(a)$;
    \item \label{tPF2} $\lambda(a\cdot b)=\lambda(a)\Lambda(b)$;
    \item \label{tPF3} $2\Delta([a,b])=[\Delta(a),b]+[a,\Delta(b)]+\Lambda(a)D(b)-\Lambda(b)D(a)$;
    \item \label{tPF4} $2\Lambda([a,b])=\Lambda(a)\lambda(b)-\Lambda(b)\lambda(a)$;
  \item \label{tPF5} $2\Delta(D(a))=D(\Delta(a))+k_0D(a)+[a_0,a]-2\lambda(a)a_0$;
  \item \label{tPF6} $2\Lambda(D(a))=\lambda(\Delta(a))-k_0\lambda(a)$.
  \end{tPFenum}
  We denote the set of all transposed Poisson flag data of $\calL$ by
  $$ \TF(\calL). $$
\end{Def}

Combining the preceding identification with
\cref{Thm: first transposed Poisson correspondence theorem}, we obtain the
following correspondence.

\begin{Prop}\label{Prop: transposed Poisson flag correspondence property}
Let $E$ be a vector space containing $L$ as a subspace and suppose that
$E=L\oplus V$. Then there are bijections
$$
\Extd(E,\calL)_V
\longleftrightarrow
\TE(\calL,V)
\longleftrightarrow
\TF(\calL).
$$
\end{Prop}

Under the right-hand bijection in
\cref{Prop: transposed Poisson flag correspondence property}, the unified
product corresponding to a transposed Poisson flag datum
$ \Omega(\calL)=(\Lambda,\Delta,a_0,k_0,\lambda,D) $
is
$$
\calL\ltimes_{\Omega(\calL,V)}V
=(L\times V,-*-,\{-,-\}),
$$
where, for all $a,b\in L$ and $k_1,k_2\in\bbK$,
\begin{align*}
(a,k_1x)*(b,k_2x)
&=\Bigl(
a\cdot b+k_1\Delta(b)+k_2\Delta(a)+k_1k_2a_0,
(k_1\Lambda(b)+k_2\Lambda(a)+k_1k_2k_0)x
\Bigr),\\
\{(a,k_1x),(b,k_2x)\}
&=\Bigl(
[a,b]+k_1D(b)-k_2D(a),
(k_1\lambda(b)-k_2\lambda(a))x
\Bigr).
\end{align*}

Through the right-hand bijection in
\cref{Prop: transposed Poisson flag correspondence property}, the relations
$\sim$ and $\approx$ on $\TE(\calL,V)$ induce corresponding relations on
$\TF(\calL)$ as follows.

\begin{Def}\label{Def: transposed Poisson flag equivalence relations}
  Let $ \Omega(\calL), \Omega'(\calL)\in\TF(\calL) $ 
  be two transposed Poisson flag data, 
  and let $ \Omega(\calL,V),\\ \Omega'(\calL,V)\in\TE(\calL,V) $
  be their corresponding transposed Poisson extending structures.
  We define
  \begin{align*}
    \Omega(\calL)\sim\Omega'(\calL) \Longleftrightarrow \Omega(\calL,V)\sim\Omega'(\calL,V),\qquad
    \Omega(\calL)\approx\Omega'(\calL) \Longleftrightarrow \Omega(\calL,V)\approx\Omega'(\calL,V).
  \end{align*}

\end{Def}

These are equivalence relations on $\TF(\calL)$. Since $V=\bbK x$, the
linear maps occurring in
\cref{Lem: transposed Poisson transformation relations} are uniquely of the
form $r(x)=c_0$ and $v(x)=p_0x$ for some
$c_0\in L$ and $p_0\in\bbK^\times$. Hence
\ref{tPR0}--\ref{tPR7} reduce to the following explicit description of
$\sim$ and $\approx$ on $\TF(\calL)$.

\begin{Prop}\label{Prop: transposed Poisson flag equivalence criterion}
Given two transposed Poisson flag data
$$
\Omega(\calL)=(\Lambda,\Delta,a_0,k_0,\lambda,D),
\qquad
\Omega'(\calL)=(\Lambda',\Delta',a'_0,k'_0,\lambda',D'),
$$
we have $\Omega(\calL)\sim\Omega'(\calL)$ if and only if
$
\Lambda'=\Lambda,\
\lambda'=\lambda,
$
and there exists $(p_0,c_0)\in\bbK^\times\times L$ such that, for every
$a\in L$,
  \begin{align*}
p_0\Delta'(a)
&=\Delta(a)+\Lambda(a)c_0-a\cdot c_0,\\
p_0^2a'_0
&=a_0+k_0c_0+c_0\cdot c_0
-2\Delta(c_0)-2\Lambda(c_0)c_0,\\
p_0k'_0
&=k_0-2\Lambda(c_0),\\
p_0D'(a)
&=D(a)+\lambda(a)c_0+[a,c_0].
  \end{align*}
Moreover, $\Omega(\calL)\approx\Omega'(\calL)$ if and only if the above
conditions hold with $p_0=1$; equivalently,
$
\Lambda'=\Lambda,
\
\lambda'=\lambda,
$
and there exists $c_0\in L$ such that, for every $a\in L$,
\begin{align*}
\Delta'(a)
&=\Delta(a)+\Lambda(a)c_0-a\cdot c_0,\\
a'_0
&=a_0+k_0c_0+c_0\cdot c_0
-2\Delta(c_0)-2\Lambda(c_0)c_0,\\
k'_0
&=k_0-2\Lambda(c_0),\\
D'(a)
&=D(a)+\lambda(a)c_0+[a,c_0].
\end{align*}
\end{Prop}

\begin{proof}
Substituting $r(x)=c_0$ and $v(x)=p_0x$ into
\ref{tPR0} and \ref{tPR4} gives
$\Lambda'=\Lambda$ and $\lambda'=\lambda$.
Relations \ref{tPR1}--\ref{tPR3} and \ref{tPR5} give the displayed formulas,
while \ref{tPR6} and \ref{tPR7} are automatic since
$\theta,\rho,\theta',\rho'$ vanish in the one-dimensional case.
For $\approx$, we further have $v=\rmId_V$, hence $p_0=1$.
\end{proof}

Combining
\cref{Prop: transposed Poisson flag correspondence property,Def: transposed Poisson flag equivalence relations,Thm: second transposed Poisson correspondence theorem}
yields the following classification.

\begin{Thm}\label{Thm: transposed Poisson flag classification}
  Let $E$ be a vector space containing $L$ as a subspace and suppose that
  $E=L\oplus V$. Then there are bijections
  \begin{align*}
  \TF(\calL)/\sim
  &\longleftrightarrow
  \TTH^2_{\sim}(\calL,V)
  \longleftrightarrow
  \Extd'(E,\calL)_V,\\
  \TF(\calL)/\approx
  &\longleftrightarrow
  \TTH^2_{\approx}(\calL,V)
  \longleftrightarrow
  \Extd''(E,\calL)_V.
\end{align*}
\end{Thm}

Finally, we specialize the fixed-module refinement of
\cref{subsec: transposed Poisson extending structures and classification}
to the one-dimensional case. 

\begin{Prop}\label{Prop: one-dimensional right transposed Poisson modules}
  Right transposed Poisson $\calL$-module structures
  $(\lt,\lh)$ on $V$ are in bijection with pairs of linear
  maps $\Lambda,\lambda:L\to\bbK$ satisfying, for all $a,b\in L$,
  \begin{align*}
    \begin{aligned}
      \Lambda(a\cdot b)&=\Lambda(a)\Lambda(b),&
      \lambda(a\cdot b)&=\lambda(a)\Lambda(b),\\
      2\Lambda([a,b])&=\Lambda(a)\lambda(b)-\Lambda(b)\lambda(a),&
      \lambda([a,b])&=0.
    \end{aligned}
  \end{align*}
  The correspondence is given by
  $$
  x\lt a=\Lambda(a)x,
  \qquad
  x\lh a=\lambda(a)x
  $$
  for all $a\in L$.
\end{Prop}

\begin{proof}
Writing $x\lt a=\Lambda(a)x$ and $x\lh a=\lambda(a)x$,
the commutative module condition, the Lie module condition, and the two compatibility conditions in \cref{Def: right transposed Poisson module}
reduce precisely to the four displayed identities.
\end{proof}

By \cref{Prop: one-dimensional right transposed Poisson modules},
fixing a right transposed Poisson $\calL$-module structure on $V$ is equivalent to fixing a pair $(\Lambda,\lambda)$ satisfying the four conditions above.
We now express the corresponding fixed-module refinement in terms of transposed Poisson flag data.

\begin{Def}\label{Def: fixed one-dimensional transposed Poisson flag data}
  Fix a pair $(\Lambda,\lambda)$ satisfying the four conditions in \cref{Prop: one-dimensional right transposed Poisson modules}. 
  We denote by
  $$ \TF_{(\Lambda,\lambda)}(\calL) $$
  the set of all quadruples $(\Delta,a_0,k_0,D)$ such that
  $ (\Lambda,\Delta,a_0,k_0,\lambda,D) $
  is a transposed Poisson flag datum of $\calL$. 

  For
  $
  (\Delta,a_0,k_0,D),\
  (\Delta',a'_0,k'_0,D')
  \in\TF_{(\Lambda,\lambda)}(\calL), 
  $
  we define
    \begin{align*}
  (\Delta,a_0,k_0,D)
  \approx_{(\Lambda,\lambda)}
  (\Delta',a'_0,k'_0,D')
    \end{align*}
  if and only if the corresponding transposed Poisson flag data
  $
  (\Lambda,\Delta,a_0,k_0,\lambda,D)$ and
  $(\Lambda,\Delta',a'_0,k'_0,\lambda,D')
  $
  are cohomologous under the relation $\approx$ of
  \cref{Def: transposed Poisson flag equivalence relations}.
\end{Def}

By \cref{Prop: transposed Poisson flag equivalence criterion}, $\approx_{(\Lambda,\lambda)}$ is an
equivalence relation on $\TF_{(\Lambda,\lambda)}(\calL)$.

Notice that $k_0$ remains part of the datum in
$ \TF_{(\Lambda,\lambda)}(\calL) $ and is not fixed. 
Indeed,
\cref{Prop: transposed Poisson flag equivalence criterion} gives
$ k'_0=k_0-2\Lambda(c_0), $
so $k_0$ is not preserved in general.

By
\cref{Prop: one-dimensional right transposed Poisson modules}
and the right-hand bijection in
\cref{Prop: transposed Poisson flag correspondence property},
fixing $(\Lambda,\lambda)$ is equivalent to fixing the corresponding right
transposed Poisson $\calL$-module structure $(V,\lt,\lh)$.
Moreover, by
\cref{Def: fixed right moduled transposed Poisson extending structure,Def: fixed one-dimensional transposed Poisson flag data},
the corresponding cohomologous relations agree. Hence
$$
\TTH^2_{\approx}\bigl(\calL,(V,\lt,\lh)\bigr)
\longleftrightarrow
\TF_{(\Lambda,\lambda)}(\calL)/\approx_{(\Lambda,\lambda)}.
$$

Allowing $(\Lambda,\lambda)$ to vary and applying
\cref{Cor: transposed Poisson right-module decomposition}
gives the following decomposition.

\begin{Cor}\label{Cor: transposed Poisson flag fixed module decomposition}
  There is a bijection
  $$
  \TTH^2_{\approx}(\calL,V)
  \longleftrightarrow
  \bigsqcup_{(\Lambda,\lambda)}
  \TF_{(\Lambda,\lambda)}(\calL)
  /\approx_{(\Lambda,\lambda)},
  $$
  where the disjoint union is taken over all pairs $(\Lambda,\lambda)$
  satisfying the four conditions in
  \cref{Prop: one-dimensional right transposed Poisson modules}.
\end{Cor}

\begin{proof}
This follows from
\cref{Cor: transposed Poisson right-module decomposition,Prop: one-dimensional right transposed Poisson modules}
together with the preceding fixed-module identification.
\end{proof}

\subsection{Factorization problem}

We now turn to the Factorization Problem for transposed Poisson algebras.

Let $\calL=(L,-\cdot-,[-,-])$ and $\calB=(B,\ell,\rho)$ be two prescribed transposed Poisson algebras, 
and let $E=L\oplus B$ be the direct sum of their underlying vector spaces.
We retain from \cref{sec: Universal Algebra FP} the notation $\Fact(\calL,\calB)$ for the transposed Poisson algebra structures on $E$ that factorize through $\calL$ and $\calB$.

Since $\calB$ is required to remain a subalgebra with its prescribed multiplication $\ell$ and bracket $\rho$, in a general transposed Poisson extending datum
$
(\lt,\rt,f,\ell,\lh,\rh,\theta,\rho),
$
we must have
$f=0$ and $\theta=0$,
while $\ell$ and $\rho$ are fixed as the prescribed operations of $\calB$. Hence only the four mixed maps
$$
\lt:B\times L\to B,\qquad
\rt:B\times L\to L,\qquad
\lh:B\times L\to B,\qquad
\rh:B\times L\to L
$$
remain to be determined.
Specializing \ref{tP0}--\ref{tP10} to this situation leads to the following definition.

\begin{Def}\label{Def: transposed Poisson matched pair}
Given four bilinear maps
$$
\lt:B\times L\to B,\qquad \rt:B\times L\to L,\qquad
\lh:B\times L\to B,\qquad \rh:B\times L\to L,
$$
the system
$$
\bigl(\calL,\calB,\olOmega(\calL,\calB)\bigr),\qquad \olOmega(\calL,\calB)=(\lt,\rt,\lh,\rh)
$$
is called a \textbf{matched pair of transposed Poisson algebras} if
$$ ((L,-\cdot-),(B,\ell),\lt,\rt) $$
is a matched pair of commutative algebras in the sense of
\cref{Def: commutative matched pair},
$$ ((L,[-,-]),(B,\rho), \lh,\rh) $$
is a matched pair of Lie algebras (see \thmref{LW90}{3.9} and \thmref{S90}{4.1}), 
and the following conditions hold for all $a,b\in L$ and $x,y\in B$:
\begin{tPMenum}
  \item\label{tPM1}
  $x\rh(a\cdot b)=2(x\lh a)\rt b+2(x\rh a)\cdot b-[x\rt b,a]-(x\lt b)\rh a.$
  
  \item\label{tPM2}
  $x\lh(a\cdot b) =2(x\lh a)\lt b -(x\lt b)\lh a.$

  \item\label{tPM3}
  $2\rho(x,y)\rt a =x\rh(y\rt a) -y\rh(x\rt a).$

  \item\label{tPM4}
  $2\rho(x,y)\lt a =\rho(x\lt a,y) +\rho(x,y\lt a) +x\lh(y\rt a) -y\lh(x\rt a).$

  \item\label{tPM5}
  $2x\rt[a,b] =[x\rt a,b] +[a,x\rt b] +(x\lt a)\rh b -(x\lt b)\rh a.$

  \item\label{tPM6}
  $2x\lt[a,b] =(x\lt a)\lh b -(x\lt b)\lh a.$

  \item\label{tPM7}
  $\ell(x,y)\rh a =2x\rt(y\rh a) -y\rh(x\rt a).$

  \item\label{tPM8}
  $\ell(x,y)\lh a =2x\lt(y\rh a) +2\ell(x,y\lh a) +\rho(x\lt a,y) -y\lh(x\rt a).$
\end{tPMenum}
The set of all such matched pairs is denoted by
$$ \MP(\calL,\calB). $$

For a matched pair $\bigl(\calL,\calB,\olOmega(\calL,\calB)\bigr)$,
let $ \Omega(\calL,B)=(\lt,\rt,0,\ell,\lh,\rh,0,\rho) $ be the corresponding transposed Poisson extending structure.
The unified product of $\calL$ and $B$ associated to $\Omega(\calL,B)$ is denoted by
$$
\calL\bowtie_{\olOmega(\calL,\calB)}\calB 
\coloneqq\calL\ltimes_{\Omega(\calL,B)}B
=(L\times B,-*-,\{-,-\})
$$
and is called the \textbf{bicrossed product} associated to the matched pair $\bigl(\calL,\calB,\olOmega(\calL,\calB)\bigr)$. 
Its multiplication and bracket are given by
\begin{align*}
  (a,x)*(b,y) &= \bigl( a\cdot b+x\rt b+y\rt a, x\lt b+y\lt a+\ell(x,y) \bigr),\\
  \{(a,x),(b,y)\} &= \bigl( [a,b]+x\rh b-y\rh a, x\lh b-y\lh a+\rho(x,y) \bigr)
\end{align*}
for all $a,b\in L$ and $x,y\in B$.
\end{Def}

Applying \cref{Thm: transposed Poisson equivalent definition} to
$ \Omega(\calL,B)=(\lt,\rt,0,\ell,\lh,\rh,0,\rho) $
gives precisely the conditions in the preceding definition. Condition
\ref{tP0} gives the requirements that
$ ((L,-\cdot-),(B,\ell),\lt,\rt) $
be a matched pair of commutative algebras and that
$ ((L,[-,-]),(B,\rho),\lh,\rh) $
be a matched pair of Lie algebras. Conditions
\ref{tP1}--\ref{tP8} reduce respectively to
\ref{tPM1}--\ref{tPM8}. Condition \ref{tP9} becomes automatic, whereas
\ref{tP10} reduces to
$$ 2\ell(x,\rho(y,z)) = \rho(\ell(x,y),z)+\rho(y,\ell(x,z)), $$
which is precisely the compatibility condition
\eqref{eq: transposed Poisson alg compatibility condition} for $\calB$.
Thus
\cref{Thm: first transposed Poisson correspondence theorem}
restricts to the following factorization correspondence.

\begin{Thm}\label{Thm: transposed Poisson factorization}
There is a bijection
$$
\Fact(\calL,\calB)\longleftrightarrow\MP(\calL,\calB).
$$
\end{Thm}

For later use in the \textbf{CCP}, we record explicitly the matched pair determined by a fixed factorization.

\begin{Rem}\label{Rem: transposed Poisson canonical matched pair}
Let
$
\calE=(E,-\cdot_E-,[-,-]_E)
$
be a transposed Poisson algebra factorizing through the prescribed
transposed Poisson algebras $\calL$ and $\calB$, with $E=L\oplus B$.
By \cref{Thm: transposed Poisson factorization}, this factorization determines
a unique matched pair
$$
\bigl(\calL,\calB,\olOmega^c(\calL,\calB)\bigr),
\qquad
\olOmega^c(\calL,\calB)=(\lt,\rt,\lh,\rh),
$$
called the \textbf{canonical matched pair} associated to the factorization.

More explicitly, let
$ p:E=L\oplus B\to L $
be the canonical projection. Then, for all $a\in L$ and $x\in B$,
\begin{align*}
  \begin{aligned}
    x\lt a &= x\cdot_E a-p(x\cdot_E a), &
    \qquad x\rt a &= p(x\cdot_E a), \\
    x\lh a &= [x,a]_E-p([x,a]_E), &
    \qquad x\rh a &= p([x,a]_E).
  \end{aligned}
\end{align*}
\end{Rem}

\subsection{Classifying complements problem}

Here the transposed Poisson algebra $\calE=(E,-\cdot_E-,[-,-]_E)$ 
and its subalgebra $\calL=(L,-\cdot-,[-,-])$ are fixed, 
and we describe and classify the transposed Poisson complements of $\calL$ in $\calE$.
Following \cref{sec: Universal Algebra CCP}, 
write $\Comp(\calL,\calE)$ for the set of all transposed Poisson complements of $\calL$ in $\calE$, and $[\calE:\calL]^f$ for the factorization index.

Assume that $\Comp(\calL,\calE)\neq \varnothing$ and fix one reference complement
$$
\calB=(B,\ell,\rho)\in\Comp(\calL,\calE).
$$
Then $E=L\oplus B$. By
\cref{Rem: transposed Poisson canonical matched pair}, this factorization
determines the canonical matched pair
$$
\bigl(\calL,\calB,\olOmega^c(\calL,\calB)\bigr),
\qquad
\olOmega^c(\calL,\calB)=(\lt,\rt,\lh,\rh).
$$

Recall from \cref{sec: Universal Algebra CCP} that, relative to the fixed
decomposition $E=L\oplus B$, every vector space complement $B'$ of $L$ in
$E$ is uniquely of the form
$$
B'=B_r\coloneqq\{r(x)+x\mid x\in B\}
$$
for a linear map $r:B\to L$.

For $x,y\in B$, the canonical matched pair $\bigl(\calL,\calB,\olOmega^c(\calL,\calB)\bigr)$ gives
\begin{align*}
(r(x)+x)\cdot_E(r(y)+y)
&=\bigl(r(x)\cdot r(y)+x\rt r(y)+y\rt r(x)\bigr)
 +\bigl(\ell(x,y)+x\lt r(y)+y\lt r(x)\bigr),\\
[r(x)+x,r(y)+y]_E
&=\bigl([r(x),r(y)]+x\rh r(y)-y\rh r(x)\bigr)
 +\bigl(\rho(x,y)+x\lh r(y)-y\lh r(x)\bigr).
\end{align*}
Since the two parenthesized terms are respectively the $L$- and
$B$-components, $B_r$ is closed under both operations precisely when the
$L$-component in each line is obtained by applying $r$ to the corresponding
$B$-component.

\begin{Prop}\label{Prop: transposed Poisson graph deformation criterion}
  Let $r:B\to L$ be a linear map.
  Then the graph $B_r$ is closed under both $-\cdot_E-$ and $[-,-]_E$ if and only if the following two concrete deformation equations hold for all $x,y\in B$:
  \begin{tPDenum}
    \item\label{eq: transposed Poisson deformation multiplication}
    $r\bigl(\ell(x,y)+x\lt r(y)+y\lt r(x)\bigr)
    =r(x)\cdot r(y)+x\rt r(y)+y\rt r(x).$

    \item\label{eq: transposed Poisson deformation bracket}
    $r\bigl(\rho(x,y)+x\lh r(y)-y\lh r(x)\bigr)
    =[r(x),r(y)]+x\rh r(y)-y\rh r(x).$
  \end{tPDenum}
  When these conditions hold,
  $$ \calB_r\coloneqq \bigl(B_r,-\cdot_E-|_{B_r\times B_r},[-,-]_E|_{B_r\times B_r}\bigr) $$
  is a transposed Poisson complement of $\calL$ in $\calE$,
  and hence $\calB_r\in\Comp(\calL,\calE)$.

\end{Prop}

The two concrete deformation equations above are precisely the specialization to transposed Poisson algebras of the general deformation equations in \cref{Def:ccp-deformation-map}.

\begin{Def}\label{Def: transposed Poisson deformation map}
  A linear map $r:B\to L$ satisfying
  \ref{eq: transposed Poisson deformation multiplication} and
  \ref{eq: transposed Poisson deformation bracket} is called a
  \textbf{deformation map} of the canonical matched pair
  $(\calL,\calB,\olOmega^c(\calL,\calB))$. 
  The set of all such maps is denoted by
  $$ \TDM(\calL,\calB\mid\olOmega^c(\calL,\calB)). $$
\end{Def}

Together with the preceding description of vector space complements, \cref{Prop: transposed Poisson graph deformation criterion} yields the following bijection.

\begin{Thm}\label{Thm: transposed Poisson complements correspondence}
  The assignment
  $$
  \TDM\bigl(\calL,\calB\mid\olOmega^c(\calL,\calB)\bigr)
  \longrightarrow
  \Comp(\calL,\calE),
  \qquad
  r\longmapsto \calB_r
  $$
  is a bijection.

\end{Thm}

\begin{proof}
  Every vector space complement of $L$ in $E=L\oplus B$ is uniquely of the
  form $B_r$ for a linear map $r:B\to L$. By
  \cref{Prop: transposed Poisson graph deformation criterion}, $B_r$ defines
  a transposed Poisson complement if and only if $r$ satisfies
  \ref{eq: transposed Poisson deformation multiplication} and
  \ref{eq: transposed Poisson deformation bracket}, equivalently,
  $r\in\TDM(\calL,\calB\mid\olOmega^c(\calL,\calB))$.

\end{proof}

Thus,
\cref{Thm: transposed Poisson complements correspondence}
completes the description part of the \textbf{CCP} for transposed Poisson
algebras. For the classification up to isomorphism of transposed Poisson
algebras, we transport the
complements associated to different deformation maps to the fixed vector
space $B$.

For $r\in\TDM(\calL,\calB\mid\olOmega^c(\calL,\calB))$, let
$$ \lambda_r:B\longrightarrow B_r, \qquad \lambda_r(x)=r(x)+x, $$
be the canonical linear isomorphism. Transport the transposed Poisson algebra
structure of $\calB_r$ back to the fixed vector space $B$ via $\lambda_r$, and
denote the resulting transposed Poisson algebra by
$$ \calB^r=(B,\ell^r,\rho^r), $$
where, for all $x,y\in B$,
\begin{align*}
\ell^r(x,y)
&=
\ell(x,y)
+x\lt r(y)
+y\lt r(x),\\
\rho^r(x,y)
&=
\rho(x,y)
+x\lh r(y)
-y\lh r(x).
\end{align*}
By construction,
$
\lambda_r:
\calB^r
\to
\calB_r
$
is an isomorphism of transposed Poisson algebras.

\begin{Def}\label{Def: transposed Poisson deformation equivalence}
  Two deformation maps
  $ r_1,r_2\in\TDM(\calL,\calB\mid\olOmega^c(\calL,\calB)) $
  are called \textbf{equivalent}, and we write 
  $$ r_1\sim r_2, $$ 
  if and only if
  $ \calB^{r_1}\cong\calB^{r_2} $ as transposed Poisson algebras. 
  Equivalently, there exists a linear isomorphism $\sigma:B\to B$ such that, for all $x,y\in B$,
    \begin{align*}
  \sigma(\ell^{r_1}(x,y))
  =\ell^{r_2}(\sigma(x),\sigma(y)),\qquad
  \sigma(\rho^{r_1}(x,y))
  =\rho^{r_2}(\sigma(x),\sigma(y)).
    \end{align*}
  After expanding the transported operations, these conditions are
  equivalently the following two identities:
  \begin{align*}
    \sigma(\ell(x,y))-\ell(\sigma(x),\sigma(y))
    &=\sigma(x)\lt r_2(\sigma(y))+\sigma(y)\lt r_2(\sigma(x))
    -\sigma(x\lt r_1(y))-\sigma(y\lt r_1(x)),\\
    \sigma(\rho(x,y))-\rho(\sigma(x),\sigma(y))
    &=\sigma(x)\lh r_2(\sigma(y))-\sigma(y)\lh r_2(\sigma(x))
    -\sigma(x\lh r_1(y))+\sigma(y\lh r_1(x)).
  \end{align*}

\end{Def}

The relation $\sim$ above is an equivalence relation on
$ \TDM(\calL,\calB\mid\olOmega^c(\calL,\calB)). $
We denote the corresponding quotient set by
$$
\mathcal H^2_{\mathrm{ccp}}
(\calL,\calB\mid\olOmega^c(\calL,\calB))
\coloneqq
\TDM(\calL,\calB\mid\olOmega^c(\calL,\calB))/\sim.
$$
Together with
\cref{Thm: transposed Poisson complements correspondence},
this yields the following classification.

\begin{Thm}\label{Thm: transposed Poisson CCP classification}
The assignment
$$
\mathcal H^2_{\mathrm{ccp}}
(\calL,\calB\mid\olOmega^c(\calL,\calB))
\longrightarrow
\Comp(\calL,\calE)/\cong,
\qquad
[r]\longmapsto[\calB_r],
$$
is a bijection. 
Consequently, the factorization index is
$$
[\calE:\calL]^f
=
\left|
\Comp(\calL,\calE)/\cong
\right|
=
\left|
\mathcal H^2_{\mathrm{ccp}}
(\calL,\calB\mid\olOmega^c(\calL,\calB))
\right|.
$$
\end{Thm}

\begin{proof}
By \cref{Thm: transposed Poisson complements correspondence}, deformation maps correspond
bijectively to complement algebras via $r\mapsto\calB_r$. For $i=1,2$, the
graph map gives an isomorphism
$\lambda_{r_i}:\calB^{r_i}\to\calB_{r_i}$, and hence
$$
r_1\sim r_2
\quad\Longleftrightarrow\quad
\calB^{r_1}\cong\calB^{r_2}
\quad\Longleftrightarrow\quad
\calB_{r_1}\cong\calB_{r_2}.
$$
Therefore the correspondence descends to the stated quotient sets. Taking
cardinalities gives the formula for $[\calE:\calL]^f$.
\end{proof}

Thus,
\cref{Thm: transposed Poisson complements correspondence,Thm: transposed Poisson CCP classification}
complete the \textbf{CCP} for transposed Poisson algebras.

\section{Extending structures for differential commutative algebras}\label{sec: differential commutative algebras}

This section is devoted to extending structures for differential commutative algebras. Besides being a basic example fitting into the linear framework of Universal Algebra, it is also closely related to the pre-Lie Poisson algebras studied in the next section: as shown in \cite{BBGW23}, if $(C,-\cdot-,d)$ is a differential commutative algebra, then $(C,-\cdot-,-\circ-)$ is a pre-Lie Poisson algebra, where $x\circ y\coloneqq x\cdot d(y)$ for all $x,y\in C$. Extending structures for algebras endowed with an additional operator were also considered in the literature, for instance for Rota--Baxter Lie algebras \cite{PZ24}.

\begin{Def}\label{Def: comm alg with d}
  Let $(C,-\cdot-)$ be a commutative algebra and $d:C\to C$ a derivation
  of $(C,-\cdot-)$, i.e., a linear map satisfying
  $$ d(a\cdot b)=d(a)\cdot b+a\cdot d(b) $$ 
  for all $a,b\in C$. 
  The triple $(C,-\cdot-,d)$ is called a \textbf{differential commutative algebra}.
\end{Def}

For this variety, the type is
$ \calF=\{-\cdot-,d\}, $
where $-\cdot-$ is a binary operation symbol and $d$ is a unary
operation symbol. 
The defining identities are
\begin{gather*}
  (x\cdot y)\cdot z\approx x\cdot(y\cdot z),\qquad
  x\cdot y\approx y\cdot x,\\
  d(x\cdot y)\approx d(x)\cdot y+x\cdot d(y).
\end{gather*}
Linearity of $d$ is not an additional identity in $\Sigma$: it follows from the standing multilinearity assumption on the fundamental operations.

\subsection{Extending structures and their classification}\label{subsec: differential commutative extending structures and classification}

We begin with the description part of the \textbf{ESP} for differential
commutative algebras.
Let $\calC=(C,-\cdot-,d)$ be a differential commutative algebra, $E$ a
vector space containing $C$ as a subspace, and $V$ a complement of $C$ in
$E$, so that $E=C\oplus V$.
We retain from \cref{sec: framework} the notation $\Extd(E,\calC)_V$ for
the set of all differential commutative algebra structures on $E$
containing $\calC$ as a subalgebra.

For the binary operation $-\cdot-$, the extending datum is the
commutative datum $(\lt,\rt,f,\ell)$ of
\cref{sec: commutative algebras}. The unary operation $d$ contributes two
additional linear components $\mu:V\to C$ and $\nu:V\to V$.
This leads to the following notion.

\begin{Def}\label{Def: diff comm extending structure}
A \textbf{differential commutative extending datum}
of $\calC$ through $V$ is a system
$$
\Omega(\calC,V)=(\lt,\rt,f,\ell,\mu,\nu),
$$
where $(\lt,\rt,f,\ell)$ is a commutative extending
datum of $(C,-\cdot-)$ through $V$, and
  $$
\mu:V\to C,
\qquad
\nu:V\to V
  $$
are linear maps. On $C\times V$, define a bilinear multiplication
$-*-$ and a linear map $D$ by
\begin{align*}
(a,x)*(b,y)
&=
\bigl(
a\cdot b+x\rt b+y\rt a+f(x,y),
x\lt b+y\lt a+\ell(x,y)
\bigr),\\
D(a,x)&=\bigl(d(a)+\mu(x),\nu(x)\bigr)
\end{align*}
for all $a,b\in C$ and $x,y\in V$. If these operations make $C\times V$
a differential commutative algebra, then $\Omega(\calC,V)$ is called a
\textbf{differential commutative extending structure} of $\calC$ through $V$. The resulting algebra
$$
\calC\ltimes V
=\calC\ltimes_{\Omega(\calC,V)}V
\coloneqq(C\times V,-*-,D)
$$
is called the \textbf{unified product} of $\calC$ and $V$ associated to $\Omega(\calC,V)$. We denote the set of all such extending
structures by
$$
\DCE(\calC,V). 
$$
\end{Def}

By \cref{Thm: commutative equivalent definition}, the multiplication $-*-$
defines a commutative algebra structure precisely when $(\lt,\rt,f,\ell)$
is a commutative extending structure. It remains to determine when $D$ is
a derivation of this multiplication.

\begin{Thm}\label{Thm: diff comm equivalent definition}
  Let
  $\Omega(\calC,V)=(\lt,\rt,f,\ell,\mu,\nu)$ be
  a differential commutative extending datum. Then
  $(C\times V,-*-,D)$ is a differential commutative algebra if and
  only if the following conditions hold for all $a\in C$ and $x,y\in V$:
  \par
  \begin{DCenum}
    \item \label{DC0} $(\lt,\rt,f,\ell)$ is a commutative extending structure of $(C,-\cdot-)$ through $V$ (see \cref{Thm: commutative equivalent definition});
    \item \label{DC1} $d(y\rt a)+\mu(y\lt a)=y\rt d(a)+a\cdot \mu(y)+\nu(y)\rt a$;
    \item \label{DC2} $\nu(y\lt a)=y\lt d(a)+\nu(y)\lt a$;
    \item \label{DC3} $d(f(x,y))+\mu(\ell(x,y))=x\rt \mu(y)+y\rt \mu(x)+f(x,\nu(y))+f(\nu(x),y)$;
    \item \label{DC4} $\nu(\ell(x,y))=x\lt \mu(y)+y\lt \mu(x)+\ell(x,\nu(y))+\ell(\nu(x),y)$.
  \end{DCenum}
\end{Thm}

\begin{proof}
By \cref{Thm: commutative equivalent definition},
$(C\times V,-*-)$ is a commutative algebra if and only if \ref{DC0} holds.
Under this condition, $D$ is a derivation if and only if
$$
D\bigl((a,x)*(b,y)\bigr)
=
D(a,x)*(b,y)+(a,x)*D(b,y)
$$
for all $(a,x),(b,y)\in C\times V$.

By bilinearity and commutativity, it suffices to check this identity on
$((a,0),(b,0))$, $((a,0),(0,y))$, and $((0,x),(0,y))$.
The first case holds because $d$ is a derivation of $(C,-\cdot-)$.
Comparing the $C$- and $V$-components in the second case gives
\ref{DC1} and \ref{DC2}, while the third gives
\ref{DC3} and \ref{DC4}.
\end{proof}

Combining
\cref{Thm: diff comm equivalent definition}
with \cref{thm:universal-description}
yields the following correspondence.

\begin{Thm}\label{Thm: first diff comm correspondence theorem}
There is a bijection
$$
\Extd(E,\calC)_V
\longleftrightarrow
\DCE(\calC,V).
$$
\end{Thm}

\begin{proof}
By \cref{thm:universal-description}, it remains only to record the
components corresponding to the unary operation. Under the stated
correspondence,
$$
d_E(a+x)=d(a)+\mu(x)+\nu(x)
$$
for all $a\in C$ and $x\in V$. Conversely, if
$(-\cdot_E-,d_E)\in\Extd(E,\calC)_V$ and
$p:E=C\oplus V\to C$ is the canonical projection, then
$$
\mu(x)=p(d_E(x)),
\qquad
\nu(x)=d_E(x)-p(d_E(x))
$$
for all $x\in V$. Together with the four commutative components, these
constructions are inverse.
\end{proof}

Thus,
\cref{Thm: diff comm equivalent definition,Thm: first diff comm correspondence theorem}
complete the description part of the \textbf{ESP} for differential
commutative algebras.

We now turn to the classification part.
By the general results of \cref{subsec: UA-ESP-classification},
it remains only to specialize the transformation relations
\eqref{Eq: F algebra map componentwisely} to differential commutative
extending structures.

For this purpose, fix two differential commutative extending structures
of $\calC$ through $V$,
$$
\Omega(\calC,V)=(\lt,\rt,f,\ell,\mu,\nu),
\qquad
\Omega'(\calC,V)=(\lt',\rt',f',\ell',\mu',\nu'),
$$
and denote their corresponding unified products, respectively, by
$$
\calC\ltimes V,
\qquad
\calC\ltimes'V.
$$

By \cref{Lem:linear-maps-correspondence}, every linear map
$C\times V\to C\times V$ stabilizing $C$ is uniquely of the form
$$
\psi_{(r,v)}(a,x)=(a+r(x),v(x)),
\qquad \forall\,a\in C,\ x\in V,
$$
where $r:V\to C$ and $v:V\to V$ are linear maps.
For the binary operation, the transformation relations are those for
commutative extending structures given in
\cref{Lem: commutative transformation relations}.
It remains to impose compatibility with the unary operation.

\begin{Lem}\label{Lem: differential commutative transformation relations}
For linear maps $r:V\to C$ and $v:V\to V$, the map
$$
\psi_{(r,v)}:\calC\ltimes V\longrightarrow\calC\ltimes'V
$$
is a homomorphism of differential commutative algebras if and only if,
for all $a\in C$ and $x,y\in V$, the following relations hold:
\begin{DCRenum}
  \item\label{DCR1} $v(x\lt a)=v(x)\lt' a$;
  \item\label{DCR2} $r(x\lt a)=v(x)\rt' a+r(x)\cdot a-x\rt a$;
  \item\label{DCR3} $v(\ell(x,y))=\ell'(v(x),v(y))+v(x)\lt' r(y)+v(y)\lt' r(x)$;
  \item\label{DCR4} $r(\ell(x,y))=f'(v(x),v(y))+v(x)\rt' r(y)+v(y)\rt' r(x)+r(x)\cdot r(y)-f(x,y)$;
  \item\label{DCR5} $\mu(x)+r(\nu(x))=d(r(x))+\mu'(v(x))$;
  \item\label{DCR6} $v(\nu(x))=\nu'(v(x))$.
\end{DCRenum}
Moreover, $\psi_{(r,v)}$ is an isomorphism if and only if
$v:V\to V$ is a linear isomorphism, and it costabilizes $V$ if and only if
$v=\operatorname{id}_V$.
\end{Lem}

\begin{proof}
By \cref{Lem: commutative transformation relations},
$\psi_{(r,v)}$ preserves the binary operation if and only if
\ref{DCR1}--\ref{DCR4} hold. It remains to require
$$
\psi_{(r,v)}\circ D=D'\circ\psi_{(r,v)}.
$$
This identity holds automatically on $(a,0)$. Evaluating it on $(0,x)$ gives
$$
\bigl(\mu(x)+r(\nu(x)),v(\nu(x))\bigr)
=
\bigl(d(r(x))+\mu'(v(x)),\nu'(v(x))\bigr),
$$
which is equivalent to \ref{DCR5} and \ref{DCR6}.

The assertions concerning isomorphisms and costabilization follow from
\cref{Lem:linear-maps-correspondence}.
\end{proof}

The relations \ref{DCR1}--\ref{DCR6}, together with the last assertion of
\cref{Lem: differential commutative transformation relations}, give the
explicit forms of the relations $\sim$ and $\approx$.

\begin{Def}\label{Def: equivalent diff comm extending structure}
The differential commutative extending structures
$\Omega(\calC,V)$ and $\Omega'(\calC,V)$ are called
\textbf{equivalent}, denoted by
$$
\Omega(\calC,V)\sim\Omega'(\calC,V),
$$
if there exist a linear map $r:V\to C$ and a linear isomorphism
$v:V\to V$ satisfying \ref{DCR1}--\ref{DCR6}.
\end{Def}

Taking $v=\operatorname{id}_V$ gives the cohomologous relation.

\begin{Def}\label{Def: cohomologous diff comm extending structure}
The differential commutative extending structures
$\Omega(\calC,V)$ and $\Omega'(\calC,V)$ are called
\textbf{cohomologous}, denoted by
$$
\Omega(\calC,V)\approx\Omega'(\calC,V),
$$
if there exists a linear map $r:V\to C$ such that
\ref{DCR1}--\ref{DCR6} hold with $v=\operatorname{id}_V$.
In particular, two cohomologous differential commutative extending
structures have the same pair $(\lt,\nu)$.
\end{Def}

The relations $\sim$ and $\approx$ are equivalence relations on
$\DCE(\calC,V)$. We denote the corresponding quotient sets by
$$
\DCH^{2}_{\sim}(\calC,V)
\coloneqq\DCE(\calC,V)/\sim,
\qquad
\DCH^{2}_{\approx}(\calC,V)
\coloneqq\DCE(\calC,V)/\approx.
$$

By
\cref{thm:universal-classification for sim,thm:universal-classification-equiv},
these relations yield the following classification.

\begin{Thm}\label{Thm: second diff comm correspondence theorem}
There are bijections
$$
\DCH^{2}_{\sim}(\calC,V)
\longleftrightarrow
\Extd'(E,\calC)_V,
\qquad
\DCH^{2}_{\approx}(\calC,V)
\longleftrightarrow
\Extd''(E,\calC)_V.
$$
\end{Thm}

Thus $\DCH^{2}_{\sim}(\calC,V)$ and
$\DCH^{2}_{\approx}(\calC,V)$ classify the differential commutative algebra
structures on $E$ extending $\calC$, respectively, up to isomorphisms
stabilizing $C$ and up to those additionally costabilizing the fixed
complement $V$.

\subsection{Flag extending structures}\label{subsec: differential commutative flag extending structures}

We now turn to flag extending structures.

\begin{Def}\label{Def: diff comm flag extending structure}
  Let $\calC=(C,-\cdot-,d)$ be a differential commutative algebra and
  $V$ a finite-dimensional vector space, and denote by $E=C\oplus V$
  their direct sum as vector spaces. A differential commutative algebra
  structure $(-\cdot_E-,d_E)$ on $E$ extending $\calC$ is called a
  \textbf{flag extending structure} of $\calC$ to $E$ if there exists a
  flag of subspaces
  $$
  0=V_0\subset V_1\subset\cdots\subset V_d=V,
  \qquad
  \dim V_i=i\quad (0\leq i\leq d),
  \qquad
  d=\dim V,
  $$
  such that, for each $1\leq i\leq d$, the subspace
  $
  E_i\coloneqq C\oplus V_i
  $
  is closed under $-\cdot_E-$ and $d_E$.
\end{Def}

As in \cref{sec: commutative algebras,sec: transposed Poisson algebras},
we now consider the one-dimensional case. Throughout the remainder of this
subsection, let $\calC=(C,-\cdot-,d)$ be a differential commutative algebra
and let $V$ be a one-dimensional vector space. Fix a nonzero vector
$x\in V$, so that $V=\bbK x$.
Then a differential commutative extending datum of $\calC$ through $V$ is
uniquely determined by
\begin{equation*}
\begin{gathered}
x\lt a=\Lambda(a)x,
\qquad x\rt a=\Delta(a),
\qquad f(x,x)=a_0,
\qquad \ell(x,x)=k_0x,\\
\mu(x)=b_0,
\qquad \nu(x)=m_0x,
\end{gathered}
\end{equation*}
for all $a\in C$, where $\Lambda:C\to\bbK$ and $\Delta:C\to C$ are
linear maps, $a_0,b_0\in C$, and $k_0,m_0\in\bbK$.

Under this identification, \ref{DC0} is equivalent to requiring
$(\Lambda,\Delta,a_0,k_0)$ to be a commutative flag datum, while
\ref{DC1}--\ref{DC4} reduce to the four additional conditions below.

\begin{Def}\label{Def: diff comm flag datum}
A \textbf{differential commutative flag datum} of $\calC$ is a 6-tuple
$$
\Omega(\calC)=(\Lambda,\Delta,a_0,k_0,b_0,m_0),
$$
where $(\Lambda,\Delta,a_0,k_0)$ is a commutative flag datum of
$(C,-\cdot-)$ in the sense of \cref{Def: commutative flag datum},
$b_0\in C$, and $m_0\in\bbK$, satisfying, for all $a\in C$,
\begin{DCFenum}
  \item \label{DCF1} $\Lambda(d(a))=0$;
  \item \label{DCF2} $d(\Delta(a))+\Lambda(a)b_0=\Delta(d(a))+a\cdot b_0+m_0\Delta(a)$;
  \item \label{DCF3} $m_0k_0+2\Lambda(b_0)=0$;
  \item \label{DCF4} $d(a_0)+k_0b_0=2\Delta(b_0)+2m_0a_0$.
\end{DCFenum}
We denote the set of all differential commutative flag data of $\calC$ by
$$
\DCF(\calC).
$$
\end{Def}

The preceding identification, together with
\cref{Thm: first diff comm correspondence theorem}, yields the following
correspondence.

\begin{Prop}\label{Prop: diff comm flag correspondence property}
Let $E$ be a vector space containing $C$ as a subspace and suppose that
$E=C\oplus V$.
There are bijections
$$
\Extd(E,\calC)_V
\longleftrightarrow
\DCE(\calC,V)
\longleftrightarrow
\DCF(\calC).
$$
\end{Prop}

Under the right-hand bijection in
\cref{Prop: diff comm flag correspondence property}, the unified product
corresponding to a differential commutative flag datum
$\Omega(\calC)=(\Lambda,\Delta,a_0,k_0,b_0,m_0)$ is given, for
$a,b\in C$ and $k_1,k_2\in\bbK$, by
\begin{align*}
(a,k_1x)*(b,k_2x)
&=
\bigl(
a\cdot b+k_1\Delta(b)+k_2\Delta(a)+k_1k_2a_0,\,
(k_1\Lambda(b)+k_2\Lambda(a)+k_1k_2k_0)x
\bigr),\\
D(a,k_1x)
&=
\bigl(d(a)+k_1b_0,\,
k_1m_0x\bigr).
\end{align*}

This completes the description of differential commutative flag extending
structures with a one-dimensional complement. We now turn to their
classification.
Through the right-hand bijection in
\cref{Prop: diff comm flag correspondence property}, the relations
$\sim$ and $\approx$ on $\DCE(\calC,V)$ induce corresponding relations
on $\DCF(\calC)$.

\begin{Def}\label{Def: diff comm flag equivalence relations}

Let $\Omega(\calC),\Omega'(\calC)\in\DCF(\calC)$, and let
$\Omega(\calC,V),\Omega'(\calC,V)\in\DCE(\calC,V)$ be the corresponding
differential commutative extending structures. We define
$$
\Omega(\calC)\sim\Omega'(\calC)
\Longleftrightarrow
\Omega(\calC,V)\sim\Omega'(\calC,V),\qquad
\Omega(\calC)\approx\Omega'(\calC)
\Longleftrightarrow
\Omega(\calC,V)\approx\Omega'(\calC,V).
$$

\end{Def}

Since $V=\bbK x$, the linear map $r:V\to C$ and the linear isomorphism
$v:V\to V$ occurring in
\cref{Lem: differential commutative transformation relations}
are uniquely determined by
$
r(x)=c_0$ and $ v(x)=p_0x,
$
for some $c_0\in C$ and $p_0\in\bbK^\times$.
Hence \ref{DCR1}--\ref{DCR6} reduce to the following conditions.

\begin{Prop}\label{Prop: diff comm flag equivalence relations}
Given two differential commutative flag data of $\calC$
$$
\Omega(\calC)=(\Lambda,\Delta,a_0,k_0,b_0,m_0),
\qquad
\Omega'(\calC)=(\Lambda',\Delta',a'_0,k'_0,b'_0,m'_0),
$$
we have $ \Omega(\calC)\sim\Omega'(\calC) $
if and only if $\Lambda'=\Lambda$, $m'_0=m_0$, and there exists
$(p_0,c_0)\in\bbK^\times\times C$ such that, for all $a\in C$,
\begin{align*}
p_0\Delta'(a)
&=\Delta(a)+\Lambda(a)c_0-a\cdot c_0,\\
p_0^2a'_0
&=a_0+k_0c_0+c_0\cdot c_0
  -2\Delta(c_0)-2\Lambda(c_0)c_0,\\
p_0k'_0
&=k_0-2\Lambda(c_0),\\
p_0b'_0
&=b_0+m_0c_0-d(c_0).
\end{align*}
Moreover,
$
\Omega(\calC)\approx\Omega'(\calC)
$
if and only if the above conditions hold with $p_0=1$; equivalently,
$\Lambda'=\Lambda$, $m'_0=m_0$, and there exists $c_0\in C$ such that,
for all $a\in C$,
\begin{align*}
\Delta'(a)
&=\Delta(a)+\Lambda(a)c_0-a\cdot c_0,\\
a'_0
&=a_0+k_0c_0+c_0\cdot c_0
  -2\Delta(c_0)-2\Lambda(c_0)c_0,\\
k'_0
&=k_0-2\Lambda(c_0),\\
b'_0
&=b_0+m_0c_0-d(c_0).
\end{align*}
\end{Prop}

By the right-hand bijection in
\cref{Prop: diff comm flag correspondence property} and the relations
defined in \cref{Def: diff comm flag equivalence relations}, the
corresponding quotient sets of differential commutative flag data and
differential commutative extending structures are in bijection.
Applying \cref{Thm: second diff comm correspondence theorem} yields the
following classification.

\begin{Thm}\label{Thm: diff comm flag classification}

Let $E$ be a vector space containing $C$ as a subspace and suppose that
$E=C\oplus V$. Then there are bijections
\begin{align*}
\DCF(\calC)/\sim
&\longleftrightarrow
\DCH^{2}_{\sim}(\calC,V)
\longleftrightarrow
\Extd'(E,\calC)_V,\\
\DCF(\calC)/\approx
&\longleftrightarrow
\DCH^{2}_{\approx}(\calC,V)
\longleftrightarrow
\Extd''(E,\calC)_V.
\end{align*}

\end{Thm}

Thus
\cref{Prop: diff comm flag correspondence property,Thm: diff comm flag classification}
give the complete description and classification of differential commutative
extending structures with a one-dimensional complement; this is the basic
step for studying general flag extending structures.

\subsection{Factorization problem}
\label{subsec: differential commutative factorization problem}

We now turn to the Factorization Problem for differential commutative
algebras.

Let
$
\calC=(C,-\cdot-,d),
\
\calB=(B,\ell,\nu)
$
be two prescribed differential commutative algebras, and let
$E=C\oplus B$ be the direct sum of their underlying vector spaces.
Following the general notation of \cref{sec: Universal Algebra FP}, denote by
$\Fact(\calC,\calB)$ the differential commutative algebra structures on
$E$ factorizing through $\calC$ and $\calB$.

Since $\calB$ is required to remain a subalgebra with its prescribed multiplication $\ell$ and derivation $\nu$, in a general differential commutative extending datum
$
(\lt,\rt,f,\ell,\mu,\nu),
$
we must have
$
f=0$ and $ \mu=0,
$
while $\ell$ and $\nu$ are fixed as the prescribed operations of $\calB$. Hence only the two mixed maps
$$
\lt:B\times C\to B,
\qquad
\rt:B\times C\to C
$$
remain to be determined.
Under this specialization, \ref{DC0} requires
$((C,-\cdot-),(B,\ell),\lt,\rt)$ to be a matched pair of commutative
algebras, \ref{DC1} and \ref{DC2} give the two additional conditions
below, while \ref{DC3} is automatic and \ref{DC4} is the derivation
identity for $\nu$.

\begin{Def}\label{Def: diff comm matched pair}
Given two bilinear maps
$$
\lt:B\times C\to B,
\qquad
\rt:B\times C\to C,
$$
the system
$$
(\calC,\calB,\olOmega(\calC,\calB)),
\qquad
\olOmega(\calC,\calB)=(\lt,\rt),
$$
is called a \textbf{matched pair of differential commutative algebras}
if
$$
((C,-\cdot-),(B,\ell),\lt,\rt)
$$
is a matched pair of commutative algebras in the sense of
\cref{Def: commutative matched pair}, and, for all $a\in C$ and $y\in B$,
\begin{DCMenum}
  \item\label{DCM1}
  $d(y\rt a)=y\rt d(a)+\nu(y)\rt a$;
  \item\label{DCM2}
  $\nu(y\lt a)=y\lt d(a)+\nu(y)\lt a$.
\end{DCMenum}
We denote the set of all such matched pairs by
$$
\MP(\calC,\calB).
$$

For a matched pair
$\bigl(\calC,\calB,\olOmega(\calC,\calB)\bigr)$,
let
$
\Omega(\calC,B)=(\lt,\rt,0,\ell,0,\nu)
$
be the corresponding differential commutative extending structure.
The unified product of $\calC$ and $B$ associated to $\Omega(\calC,B)$ is denoted by
$$
\calC\bowtie_{\olOmega(\calC,\calB)}\calB
\coloneqq
\calC\ltimes_{\Omega(\calC,B)}B
=
(C\times B,-*-,D)
$$
and is called the \textbf{bicrossed product} associated to the matched pair 
$\bigl(\calC,\calB,\olOmega(\calC,\calB)\bigr)$.
Its operations are
\begin{align*}
(a,x)*(b,y)
&=
\bigl(
a\cdot b+x\rt b+y\rt a,\,
x\lt b+y\lt a+\ell(x,y)
\bigr),\\
D(a,x)&=\bigl(d(a),\nu(x)\bigr)
\end{align*}
for all $a,b\in C$ and $x,y\in B$.
\end{Def}

Thus,
\cref{Thm: first diff comm correspondence theorem}
restricts to the following factorization correspondence.

\begin{Thm}\label{Thm: diff comm factorization}
There is a bijection
$$
\Fact(\calC,\calB)
\longleftrightarrow
\MP(\calC,\calB).
$$

\end{Thm}

For later use in the \textbf{CCP}, we record explicitly the matched pair
determined by a fixed factorization.

\begin{Rem}\label{Rem: diff comm canonical matched pair}
Let $\calE=(E,-\cdot_E-,d_E)$ be a differential commutative algebra
factorizing through the prescribed differential commutative algebras
$\calC$ and $\calB$, with $E=C\oplus B$. By
\cref{Thm: diff comm factorization}, this factorization determines a
unique matched pair
$$
\bigl(\calC,\calB,\olOmega^c(\calC,\calB)\bigr),
\qquad
\olOmega^c(\calC,\calB)=(\lt,\rt),
$$
called the \textbf{canonical matched pair} associated to the
factorization.

More explicitly, let
$
p:E=C\oplus B\to C
$
be the canonical projection. Then, for all $a\in C$ and $x\in B$,
$$
x\rt a=p(x\cdot_E a),
\qquad
x\lt a=x\cdot_E a-p(x\cdot_E a).
$$
\end{Rem}

\subsection{Classifying complements problem}\label{subsec: differential commutative classifying complements problem}

Here the differential commutative algebra
$\calE=(E,-\cdot_E-,d_E)$ and its subalgebra
$\calC=(C,-\cdot-,d)$ are fixed, and we describe and classify the
differential commutative complements of $\calC$ in $\calE$.
Following \cref{sec: Universal Algebra CCP}, 
write $\Comp(\calC,\calE)$ for the set of all differential commutative complements of $\calC$ in $\calE$, and $[\calE:\calC]^f$ for the factorization index.

From now on, assume that
$\Comp(\calC,\calE)\neq\varnothing$, and fix a reference complement
$$
\calB=(B,\ell,\nu)\in\Comp(\calC,\calE).
$$
Then $E=C\oplus B$, and this factorization determines, by
\cref{Rem: diff comm canonical matched pair}, the canonical matched pair
$$
\bigl(\calC,\calB,\olOmega^c(\calC,\calB)\bigr),
\qquad
\olOmega^c(\calC,\calB)=(\lt,\rt).
$$

Since $E=C\oplus B$, every vector space complement of $C$ in $E$ is
uniquely of the form
$$
B_r\coloneqq\{r(x)+x\mid x\in B\}
$$
for a linear map $r:B\to C$.
By \cref{Prop: commutative graph deformation criterion}, $B_r$ is closed
under $-\cdot_E-$ if and only if
\begin{equation}\label{eq: diff comm deformation multiplication}
r\bigl(\ell(x,y)+x\lt r(y)+y\lt r(x)\bigr)
=
r(x)\cdot r(y)+x\rt r(y)+y\rt r(x)
\end{equation}
for all $x,y\in B$. Moreover,
$d_E(r(x)+x)=d(r(x))+\nu(x)$, so $B_r$ is also closed under $d_E$ if and
only if
\begin{equation}\label{eq: diff comm deformation derivation}
d(r(x))=r(\nu(x)).
\end{equation}

When these conditions hold,
$$
\calB_r
\coloneqq
\left(
B_r,\,
-\cdot_E-\big|_{B_r\times B_r},\,
d_E\big|_{B_r}
\right)
$$
is a differential commutative complement of $\calC$ in $\calE$.

\begin{Def}\label{Def: diff comm deformation map}
A linear map $r:B\to C$ is called a \textbf{deformation map} of the
canonical matched pair
$\bigl(\calC,\calB,\olOmega^c(\calC,\calB)\bigr)$
if it satisfies
\eqref{eq: diff comm deformation multiplication} and
\eqref{eq: diff comm deformation derivation}. We denote the set of all such
maps by
$$
\DDM(\calC,\calB\mid\olOmega^c(\calC,\calB)).
$$
\end{Def}

The graph description above yields the following correspondence.

\begin{Thm}\label{Thm: diff comm CCP correspondence}
The assignment
$$
\DDM(\calC,\calB\mid\olOmega^c(\calC,\calB))
\longrightarrow
\Comp(\calC,\calE),
\qquad
r\longmapsto\calB_r
$$
is a bijection.
\end{Thm}

Thus the description part of the \textbf{CCP} for differential commutative
algebras is complete. We now turn to the classification of complements up to
isomorphism of differential commutative algebras.

For
$r\in\DDM(\calC,\calB\mid\olOmega^c(\calC,\calB))$, let
$$
\lambda_r:B\longrightarrow B_r,
\qquad
\lambda_r(x)=r(x)+x,
$$
be the canonical linear isomorphism. Transport the differential commutative
algebra $\calB_r$ to the fixed vector space $B$ through $\lambda_r$, and
denote the resulting algebra by
$$
\calB^r=(B,\ell^r,\nu),
\qquad
\ell^r(x,y)=\ell(x,y)+x\lt r(y)+y\lt r(x).
$$
Then
$
\lambda_r:\calB^r\longrightarrow\calB_r
$
is an isomorphism of differential commutative algebras.

This leads to the following equivalence relation on deformation maps.

\begin{Def}\label{Def: diff comm deformation equivalence}
Two deformation maps
$r_1,r_2\in
\DDM(\calC,\calB\mid\olOmega^c(\calC,\calB))$
are called \textbf{equivalent}, denoted by
$$
r_1\sim r_2,
$$
if and only if
$
\calB^{r_1}\cong\calB^{r_2}
$
as differential commutative algebras.
Equivalently, there exists a linear isomorphism
$\sigma:B\to B$ such that, for all $x,y\in B$,
$$
\sigma(\ell^{r_1}(x,y))
=\ell^{r_2}(\sigma(x),\sigma(y)),\qquad
\sigma(\nu(x))=\nu (\sigma(x)).
$$
\end{Def}

The relation $\sim$ is an equivalence relation on
$\DDM(\calC,\calB\mid\olOmega^c(\calC,\calB))$.
We denote the corresponding quotient set by
$$
\mathcal H^2_{\mathrm{ccp}}
(\calC,\calB\mid\olOmega^c(\calC,\calB))
\coloneqq
\DDM(\calC,\calB\mid\olOmega^c(\calC,\calB))/\sim.
$$

Applying \cref{Thm:ccp-classification} gives the following classification.

\begin{Thm}\label{Thm: diff comm CCP classification}
There is a bijection
$$
\mathcal H^2_{\mathrm{ccp}}
(\calC,\calB\mid\olOmega^c(\calC,\calB))
\longrightarrow
\Comp(\calC,\calE)/\cong,
\qquad
[r]\longmapsto[\calB_r].
$$
Consequently, the factorization index is given by
$$
[\calE:\calC]^f
=
\left|
\mathcal H^2_{\mathrm{ccp}}
(\calC,\calB\mid\olOmega^c(\calC,\calB))
\right|
=
\left|
\Comp(\calC,\calE)/\cong
\right|.
$$
\end{Thm}

Thus,
\cref{Thm: diff comm CCP correspondence,Thm: diff comm CCP classification}
complete the \textbf{CCP} for differential commutative algebras.

We conclude this section by recording the compatibility with the construction
leading to pre-Lie Poisson algebras.

\begin{Rem}\label{Rem: PLP from DCA}
Let $\calC=(C,-\cdot-,d)$ be a differential commutative algebra
and let
$\Omega(\calC,V)=(\lt,\rt,f,\ell,\mu,\nu)$ be an extending structure.
By \cite{BBGW23},
$$
\calC_{\circ}=(C,-\cdot-,-\circ-),
\qquad
a\circ b\coloneqq a\cdot d(b),
$$
is a pre-Lie Poisson algebra.
The construction of \cite{BBGW23} lifts to extending structures as follows:
$$
\Omega(\calC_{\circ},V)\coloneqq(\lt,\rt,f,\ell,\prec,\succ,\brt,\blt,g,\kappa),
$$
where
\begin{equation*}
\begin{gathered}
x\succ a\coloneqq x\rt d(a),
\qquad
x\prec a\coloneqq x\lt d(a),
\qquad
a\blt x\coloneqq a\cdot\mu(x)+\nu(x)\rt a,
\qquad
a\brt x\coloneqq \nu(x)\lt a,\\
g(x,y)\coloneqq x\rt\mu(y)+f(x,\nu(y)),
\qquad
\kappa(x,y)\coloneqq x\lt\mu(y)+\ell(x,\nu(y)),
\end{gathered}
\end{equation*}
for all $a\in C$ and $x,y\in V$. 
Then $\Omega(\calC_{\circ},V)$ is a pre-Lie Poisson extending structure, and
its unified product is 
the pre-Lie Poisson algebra obtained from
$$
\calC\ltimes V=(C\times V,-*-,D)
$$
by setting $u\circ v\coloneqq u*D(v)$. 
Indeed, \cref{Thm: diff comm equivalent definition} makes $\calC\ltimes V$ a
differential commutative algebra, and decomposition of this product into its
$C$- and $V$-components yields exactly the six displayed formulas.
\end{Rem}

\bigskip

\section{Extending structures for pre-Lie Poisson algebras}\label{sec: pre-Lie Poisson alg}

This section applies the preceding framework to pre-Lie Poisson algebras,
introduced in \cite{BBGW23} as a common generalization of Novikov--Poisson,
pre-Lie commutative, and differential Novikov--Poisson algebras. We record
only the computations specific to this variety. Extending structures for
another notion of pre-Poisson algebra were studied in \cite{ZLS25}. In that
setting, the first operation is a Zinbiel product, whereas in the present
setting it is commutative and associative. The former is
related to Poisson algebras by symmetrization, while the latter is related to
transposed Poisson algebras by the commutator.

We first recall the definition of pre-Lie Poisson algebras.

\begin{Def}\defref{BBGW23}{3.8}\label{Def: PLP algebra specialization}
  Let $L$ be a vector space equipped with two bilinear operations:
  $$
  -\cdot-:L\otimes L\to L
  \qquad\text{and}\qquad
  -\circ-:L\otimes L\to L
  $$
  The triple $\calL=(L,-\cdot-,-\circ-)$ is called a
  \textbf{pre-Lie Poisson algebra} if $(L,-\cdot-)$ is a
  (nonunital) commutative associative algebra,
  $(L,-\circ-)$ is a pre-Lie algebra, i.e., for all $a,b,c\in L$
  \begin{align}
    (a\circ b)\circ c-a\circ (b\circ c)=(b\circ a)\circ c-b\circ (a\circ c),\label{eq: pre-Lie identity}
  \end{align}
  and the following compatibility conditions hold for all $a,b,c\in L$:
  \begin{align}
    (a\cdot b)\circ c &= a\cdot (b\circ c),\label{eq: PLP axiom specialization 1}\\
    (a\circ b)\cdot c-(b\circ a)\cdot c &= a\circ (b\cdot c)-b\circ (a\cdot c).\label{eq: PLP axiom specialization 2}
  \end{align}
\end{Def}

The following two constructions from \cite{BBGW23} relate pre-Lie Poisson
algebras to the differential commutative and transposed Poisson algebras
considered in 
Section~\ref{sec: differential commutative algebras} and Section~\ref{sec: transposed Poisson algebras}, respectively.

\begin{Rem}\label{Rem: PLP relations specialization}
  (a) By \cite{BBGW23}, if
  $\calC=(C,-\cdot-,d)$ is a differential commutative algebra,
  then
  $\calC_{\circ}=(C,-\cdot-,-\circ-)$ is a pre-Lie
  Poisson algebra for
  $$
  a\circ b\coloneqq a\cdot d(b),\qquad a,b\in C.
  $$
  This is the construction of \cref{Rem: PLP from DCA}.

  (b) Assume $\operatorname{char}(\bbK)=0$.
  By \propref{BBGW23}{3.10}, if $\calL=(L,-\cdot-,-\circ-)$ is a pre-Lie Poisson algebra, then $(L,-\cdot-,[-,-])$ is a transposed Poisson algebra for
  $$
  [a,b]\coloneqq a\circ b-b\circ a,\qquad a,b\in L.
  $$
  This gives the connection with \cref{sec: transposed Poisson algebras}.
\end{Rem}

\subsection{Extending structures and their classification}\label{subsec: PLP extending structures and classification}

Let $\calL=(L,-\cdot-,-\circ-)$ be a pre-Lie Poisson algebra, $E$ a vector
space containing $L$ as a subspace, and $V$ a complement of $L$ in $E$, so
that $E=L\oplus V$. With this fixed decomposition, we retain from
\cref{sec: framework} the notation $\Extd(E,\calL)_V$ for the set of all
pre-Lie Poisson algebra structures on $E$ containing $\calL$ as a subalgebra.

Since the pre-Lie product $-\circ-$ is neither commutative nor
antisymmetric, its mixed components on $V\times L$ and $L\times V$ must be
specified separately. Together with the four maps from the commutative
extending datum, this gives the following ten-component datum.

\begin{Def}\label{Def: PLP extending datum specialization}
  A \textbf{pre-Lie Poisson extending datum} of $\calL$ through $V$ is a system
  $$
  \Omega(\calL,V)=(\lt,\rt,f,\ell,\prec,\succ,\brt,\blt,g,\kappa)
  $$
  consisting of ten bilinear maps
  \begin{gather*}
    \lt : V \times L \to V, \quad \rt : V \times L \to L, \quad f : V \times V \to L, \quad \ell : V \times V \to V, \\
    \prec : V \times L \to V, \quad \succ : V \times L \to L, \quad \brt : L \times V \to V, \quad \blt : L \times V \to L,\\
    \quad  g : V \times V \to L, \quad \kappa : V \times V \to V
  \end{gather*}
  where $f$ and $\ell$ are symmetric.

  For a pre-Lie Poisson extending datum of $\calL$ through $V$ as above, define two bilinear operations on the vector space $L\times V$ by
  \begin{gather*}
    (a,x)*(b,y)\coloneqq (a\cdot b+x\rt b+y\rt a+f(x,y),\
    x\lt b+ y\lt a+ \ell (x,y)),\\
    (a,x)\bullet(b,y)\coloneqq (a\circ b+x\succ b+a\blt y+g(x,y),\
    x\prec b+ a\brt y+ \kappa (x,y)),
  \end{gather*}
  for all $a,b\in L$ and $x,y\in V$. If these operations make
  $L\times V$ into a pre-Lie Poisson algebra, then $\Omega(\calL,V)$ is called a
  \textbf{pre-Lie Poisson extending structure} of $\calL$ through $V$.
  In this case, the resulting pre-Lie Poisson algebra
  $$
  \calL\ltimes V
  =\calL\ltimes_{\Omega(\calL,V)}V
  \coloneqq(L\times V,-*-,-\bullet-)
  $$
  is called the \textbf{unified product} of $\calL$ and $V$ associated to
  $\Omega(\calL,V)$. We denote the set of all such pre-Lie Poisson extending structures by
  $$
  \mathcal{PLE}(\calL,V).
  $$
\end{Def}

It remains to determine the additional conditions imposed by the pre-Lie identity \eqref{eq: pre-Lie identity} and the compatibilities \eqref{eq: PLP axiom specialization 1} and \eqref{eq: PLP axiom specialization 2}.

\begin{Thm}\label{Thm: PLP compatibility specialization}
  Let 
  $\Omega(\calL,V)=(\lt,\rt,f,\ell,\prec,\succ,\brt,\blt,g,\kappa)$ be a pre-Lie
  Poisson extending datum of $\calL$ through $V$. Then
  the two operations $-*-$ and $-\bullet-$ defined in
  \cref{Def: PLP extending datum specialization} make $L\times V$ into a unified product if and only if the following compatibility
  conditions hold for all $a,b\in L$ and $x,y,z\in V$:
  \begin{PLPenum}[start=0]
    \item\label{PLP0} $(\lt, \rt, f, \ell)$ is a commutative extending structure of $(L,-\cdot-)$ through $V$ (see \cref{Thm: commutative equivalent definition});
    \item\label{PLP1} $(a\circ b)\blt x-a\circ(b\blt x)-a\blt(b\brt x)
    =(b\circ a)\blt x-b\circ(a\blt x)-b\blt(a\brt x)$, and
    $(a\circ b)\brt x-a\brt(b\brt x)
    =(b\circ a)\brt x-b\brt(a\brt x)$;
    \item\label{PLP2} $(a\blt x)\circ b+(a\brt x)\succ b-a\circ(x\succ b)-a\blt(x\prec b)
    =(x\succ a)\circ b+(x\prec a)\succ b-x\succ(a\circ b)$, and
    $(a\brt x)\prec b-a\brt(x\prec b)
    =(x\prec a)\prec b-x\prec(a\circ b)$;
    \item\label{PLP3} $(a\blt x)\blt y+g(a\brt x,y)-a\circ g(x,y)-a\blt\kappa(x,y)
    =(x\succ a)\blt y+g(x\prec a,y)-x\succ(a\blt y)-g(x,a\brt y)$, and
    $(a\blt x)\brt y+\kappa(a\brt x,y)-a\brt\kappa(x,y)
    =(x\succ a)\brt y+\kappa(x\prec a,y)-x\prec(a\blt y)-\kappa(x,a\brt y)$;
    \item\label{PLP4} $g(x,y)\circ a+\kappa(x,y)\succ a-x\succ(y\succ a)-g(x,y\prec a)
    =g(y,x)\circ a+\kappa(y,x)\succ a-y\succ(x\succ a)-g(y,x\prec a)$, and
    $\kappa(x,y)\prec a-x\prec(y\succ a)-\kappa(x,y\prec a)
    =\kappa(y,x)\prec a-y\prec(x\succ a)-\kappa(y,x\prec a)$;
    \item\label{PLP5} $g(x,y)\blt z+g(\kappa(x,y),z)-x\succ g(y,z)-g(x,\kappa(y,z))
    =g(y,x)\blt z+g(\kappa(y,x),z)-y\succ g(x,z)-g(y,\kappa(x,z))$, and
    $g(x,y)\brt z+\kappa(\kappa(x,y),z)-x\prec g(y,z)-\kappa(x,\kappa(y,z))
    =g(y,x)\brt z+\kappa(\kappa(y,x),z)-y\prec g(x,z)-\kappa(y,\kappa(x,z))$;
    \item\label{PLP6} $(a\cdot b)\blt x
    =a\cdot(b\blt x)+(b\brt x)\rt a$, and
    $(a\cdot b)\brt x
    =(b\brt x)\lt a$;
    \item\label{PLP7} $a\cdot(x\succ b)+(x\prec b)\rt a
    =(x\rt a)\circ b+(x\lt a)\succ b
    =x\rt(a\circ b)$, and
    $(x\prec b)\lt a
    =(x\lt a)\prec b
    =x\lt(a\circ b)$;
    \item\label{PLP8} $a\cdot g(x,y)+\kappa(x,y)\rt a
    =(x\rt a)\blt y+g(x\lt a,y)
    =x\rt(a\blt y)+f(x,a\brt y)$, and
    $\kappa(x,y)\lt a
    =(x\rt a)\brt y+\kappa(x\lt a,y)
    =x\lt(a\blt y)+\ell(x,a\brt y)$;
    \item\label{PLP9} $f(x,y)\circ a+\ell(x,y)\succ a
    =x\rt(y\succ a)+f(x,y\prec a)$, and
    $\ell(x,y)\prec a
    =x\lt(y\succ a)+\ell(x,y\prec a)$;
    \item\label{PLP10} $f(x,y)\blt z+g(\ell(x,y),z)
    =x\rt g(y,z)+f(x,\kappa(y,z))$, and
    $f(x,y)\brt z+\kappa(\ell(x,y),z)
    =x\lt g(y,z)+\ell(x,\kappa(y,z))$;
    \item\label{PLP11} $x\succ(a\cdot b)-a\circ(x\rt b)-a\blt(x\lt b)
    =(x\succ a)\cdot b+(x\prec a)\rt b-(a\blt x)\cdot b-(a\brt x)\rt b$, and
    $x\prec(a\cdot b)-a\brt(x\lt b)
    =(x\prec a)\lt b-(a\brt x)\lt b$;
    \item\label{PLP12} $x\succ(y\rt a)+g(x,y\lt a)-a\circ f(x,y)-a\blt\ell(x,y)
    =y\rt(x\succ a)+f(x\prec a,y)-y\rt(a\blt x)-f(a\brt x,y)$, and
    $x\prec(y\rt a)+\kappa(x,y\lt a)-a\brt\ell(x,y)
    =y\lt(x\succ a)+\ell(x\prec a,y)-y\lt(a\blt x)-\ell(a\brt x,y)$;
    \item\label{PLP13} $z\rt g(x,y)+f(\kappa(x,y),z)-z\rt g(y,x)-f(\kappa(y,x),z)
    =x\succ f(y,z)+g(x,\ell(y,z))-y\succ f(x,z)-g(y,\ell(x,z))$, and
    $z\lt g(x,y)+\ell(\kappa(x,y),z)-z\lt g(y,x)-\ell(\kappa(y,x),z)
    =x\prec f(y,z)+\kappa(x,\ell(y,z))-y\prec f(x,z)-\kappa(y,\ell(x,z))$.
  \end{PLPenum}

\end{Thm}

\begin{proof}
  By \ref{PLP0} and \cref{Thm: commutative equivalent definition}, $(L\times V,-*-)$ is a commutative associative algebra. Since the remaining identities are trilinear, it suffices to verify them on elements of the forms $(a,0)$ and $(0,x)$.

  For the pre-Lie identity \eqref{eq: pre-Lie identity}, the all-$L$ case holds in $\calL$. Since interchanging the first two variables interchanges the two sides, it remains to consider the representatives
  $$
  ((a,0),(b,0),(0,x)),\quad
  ((a,0),(0,x),(b,0)),\quad
  ((a,0),(0,x),(0,y)),
  $$
  $$
  ((0,x),(0,y),(a,0)),\quad
  ((0,x),(0,y),(0,z)).
  $$
  Comparing the $L$- and $V$-components gives \ref{PLP1}--\ref{PLP5}, respectively.

  For compatibility \eqref{eq: PLP axiom specialization 1}, the all-$L$ case again holds in $\calL$, and comparison of the remaining $L$- and $V$-components gives \ref{PLP6}--\ref{PLP10}. The chains of equalities in \ref{PLP7} and \ref{PLP8} follow from the commutativity of $-*-$: applying \eqref{eq: PLP axiom specialization 1} to $(u,v,w)$ and $(v,u,w)$ gives the same left-hand side, while the right-hand sides are $u*(v\bullet w)$ and $v*(u\bullet w)$, respectively.

  Finally, for compatibility \eqref{eq: PLP axiom specialization 2}, define
  $$
  D(u,v,w)
  \coloneqq
  (u\bullet v)*w-(v\bullet u)*w-u\bullet(v*w)+v\bullet(u*w).
  $$
  By definition,
  $
  D(v,u,w)=-D(u,v,w).
  $
  And \eqref{eq: PLP axiom specialization 2} is equivalent to $D(u,v,w)=0$.
  Since \eqref{eq: PLP axiom specialization 1} has already been established and $-*-$ is commutative,
  \begin{align*}
  &(u\bullet v)*w
  =w*(u\bullet v)
  =(w*u)\bullet v
  =(u*w)\bullet v
  =u*(w\bullet v)
  =(w\bullet v)*u,\\
  &D(u,v,w)+D(v,w,u)+D(w,u,v)=0.
  \end{align*}
  These two relations reduce all non-pure-$L$ substitutions to the three representatives
  $$
  ((a,0),(0,x),(b,0)),\quad
  ((a,0),(0,x),(0,y)),\quad
  ((0,x),(0,y),(0,z)).
  $$
  Comparing their $L$- and $V$-components gives \ref{PLP11}, \ref{PLP12}, and \ref{PLP13}, respectively.
\end{proof}

Combining
\cref{Thm: PLP compatibility specialization}
with \cref{thm:universal-description}
yields the following correspondence.

\begin{Thm}\label{Thm: PLP description specialization}
There is a bijection
$$
\Extd(E,\calL)_V\longleftrightarrow\mathcal{PLE}(\calL,V).
$$
\end{Thm}

\begin{proof}
By \cref{thm:universal-description,Thm: PLP compatibility specialization},
the stated bijection is the specialization of the general correspondence.
It remains only to record the concrete recovery formulas.

Let
$
p:E\to L,\ p(a+x)=a
$
be the canonical projection associated to $E=L\oplus V$.
For a pre-Lie Poisson algebra structure on $E$ extending $\calL$, 
($-\cdot_E-,-\circ_E-$),
the corresponding extending structure is recovered, 
for all $a\in L$ and $x,y\in V$, by
$$
\begin{aligned}
x\rt a&=p(x\cdot_Ea),
&\qquad x\lt a&=x\cdot_Ea-p(x\cdot_Ea),\\
f(x,y)&=p(x\cdot_Ey),
&\ell(x,y)&=x\cdot_Ey-p(x\cdot_Ey),\\
x\succ a&=p(x\circ_Ea),
&x\prec a&=x\circ_Ea-p(x\circ_Ea),\\
a\blt x&=p(a\circ_Ex),
&a\brt x&=a\circ_Ex-p(a\circ_Ex),\\
g(x,y)&=p(x\circ_Ey),
&\kappa(x,y)&=x\circ_Ey-p(x\circ_Ey).
\end{aligned}
$$

Conversely, every $\Omega(\calL,V)\in\mathcal{PLE}(\calL,V)$ yields,
through its unified product and the canonical identification
$L\times V\cong E$, a pre-Lie Poisson algebra structure on $E$ extending
$\calL$. These two constructions are inverse.
\end{proof}

Thus, \cref{Thm: PLP compatibility specialization,Thm: PLP description specialization}
complete the description part of the \textbf{ESP} for pre-Lie Poisson algebras.

We now turn to the classification part. 
By the bijection in \cref{Thm: PLP description specialization}, it remains to describe the induced relations $\sim$ and $\approx$ on $\mathcal{PLE}(\calL,V)$.

Fix two pre-Lie Poisson extending structures of $\calL$ through $V$,
$$
\Omega(\calL,V)
=
(\lt,\rt,f,\ell,\prec,\succ,\brt,\blt,g,\kappa),\quad
\Omega'(\calL,V)
=
(\lt',\rt',f',\ell',\prec',\succ',\brt',\blt',g',\kappa'),
$$
and denote their associated unified products by
$$
\calL\ltimes V,\qquad \calL\ltimes'V.
$$
By \cref{Lem:linear-maps-correspondence}, every linear map
$L\times V\to L\times V$ stabilizing $L$ is uniquely of the form
$$
\psi_{(r,v)}(a,x)=(a+r(x),v(x)),\quad \forall a\in L,\ x\in V,
$$
where $r:V\to L$ and $v:V\to V$ are linear maps. The general transformation
relations specialize as follows.

\begin{Lem}\label{Lem: PLP morphism specialization}
$\psi_{(r,v)}:\calL\ltimes V\longrightarrow \calL\ltimes'V$ is a
homomorphism of pre-Lie Poisson algebras if and only if, for all
$a\in L$ and $x,y\in V$, the following conditions hold:
\begin{PLPRenum}
\item\label{PLPR1} $v(x\lt a)=v(x)\lt' a$;
\item\label{PLPR2} $r(x\lt a)=v(x)\rt' a+r(x)\cdot a-x\rt a$;
\item\label{PLPR3}
$v(\ell(x,y))=\ell'(v(x),v(y))+v(x)\lt' r(y)+v(y)\lt' r(x)$;
\item\label{PLPR4}
$r(\ell(x,y))=f'(v(x),v(y))+v(x)\rt' r(y)+v(y)\rt' r(x)+r(x)\cdot r(y)-f(x,y)$;
\item\label{PLPR5}
$a\blt y+r(a\brt y)=a\circ r(y)+a\blt' v(y)$;
\item\label{PLPR6} $v(a\brt y)=a\brt' v(y)$;
\item\label{PLPR7}
$x\succ a+r(x\prec a)=r(x)\circ a+v(x)\succ' a$;
\item\label{PLPR8} $v(x\prec a)=v(x)\prec' a$;
\item\label{PLPR9}
$\begin{aligned}[t]
&g(x,y)+r(\kappa(x,y))=r(x)\circ r(y)+v(x)\succ' r(y)+r(x)\blt' v(y)+g'(v(x),v(y));
\end{aligned}$
\item\label{PLPR10}
$v(\kappa(x,y))=v(x)\prec' r(y)+r(x)\brt' v(y)+\kappa'(v(x),v(y))$.
\end{PLPRenum}
Moreover, $\psi_{(r,v)}$ is an isomorphism if and only if
$v:V\to V$ is a linear isomorphism, and it costabilizes $V$ if and only if
$v=\operatorname{id}_V$.
\end{Lem}

\begin{proof}
By bilinearity, it suffices to check preservation of the two operations on
pairs whose entries are of the forms $(a,0)$ and $(0,x)$. For the commutative
multiplication $-*-$, the pure $L$-pair is automatic, while the mixed pair and
the pure $V$-pair yield, by comparing the $L$- and $V$-components,
\ref{PLPR1}--\ref{PLPR4}. For the pre-Lie product $-\bullet-$, the pure
$L$-pair is again automatic. Since $-\bullet-$ is not commutative, the two
mixed pairs must be considered separately; together with the pure $V$-pair,
they yield \ref{PLPR5}--\ref{PLPR10}. 
The final assertions follow from \cref{Lem:linear-maps-correspondence}.
\end{proof}

The preceding lemma gives the explicit form of the equivalence relations on
$\mathcal{PLE}(\calL,V)$. We record them as follows.

\begin{Def}\label{Def: PLP equivalence specialization}
$ \Omega(\calL,V) $ and $ \Omega'(\calL,V)$
are called \textbf{equivalent}, denoted by
$
\Omega(\calL,V)\sim\Omega'(\calL,V),
$
if there exist a linear map $r:V\to L$ and a linear isomorphism
$v:V\to V$ satisfying \ref{PLPR1}--\ref{PLPR10}.
\end{Def}

\begin{Def}\label{Def: PLP cohomology specialization}
$ \Omega(\calL,V) $ and $ \Omega'(\calL,V)$
are called \textbf{cohomologous}, denoted by
$
\Omega(\calL,V)\approx\Omega'(\calL,V),
$
if there exists a linear map $r:V\to L$ such that
\ref{PLPR1}--\ref{PLPR10} hold with $v=\operatorname{id}_V$.
In particular, two cohomologous pre-Lie Poisson extending structures
have the same triple $(\lt,\prec,\brt)$.
\end{Def}

The relations $\sim$ and $\approx$ are equivalence relations on
$\mathcal{PLE}(\calL,V)$. We denote the corresponding quotient sets by
$$
\mathcal{PLH}^{2}_{\sim}(\calL,V)
\coloneqq\mathcal{PLE}(\calL,V)/\sim,
\qquad
\mathcal{PLH}^{2}_{\approx}(\calL,V)
\coloneqq\mathcal{PLE}(\calL,V)/\approx.
$$

By
\cref{thm:universal-classification for sim,thm:universal-classification-equiv},
these relations yield the following classification.

\begin{Thm}\label{Thm: PLP ESP classification specialization}
There are bijections
$$
\mathcal{PLH}^{2}_{\sim}(\calL,V)
\longleftrightarrow \Extd'(E,\calL)_V,
\qquad
\mathcal{PLH}^{2}_{\approx}(\calL,V)
\longleftrightarrow \Extd''(E,\calL)_V.
$$
\end{Thm}

The preceding bijections complete the classification of pre-Lie Poisson
extending structures. Since $\approx$ preserves $(\lt,\prec,\brt)$, a
fixed-action refinement analogous to that in Section~\ref{subsec: transposed Poisson extending structures and classification} is available.
This refinement and the one-dimensional flag case are omitted.

\subsection{Factorization problem}

Let $\calL=(L,-\cdot-,-\circ-)$ and
$\calB=(B,\ell,\kappa)$ be prescribed pre-Lie Poisson algebras, and let
$E=L\oplus B$ be the direct sum of their underlying vector spaces.
We use $\Fact(\calL,\calB)$ for the pre-Lie Poisson algebra structures on
$E$ factorizing through $\calL$ and $\calB$.

Requiring $\calB$ to remain a subalgebra in a pre-Lie Poisson extending
datum
$
(\lt,\rt,f,\ell,\prec,\succ,\brt,\blt,g,\kappa)
$
forces $f=g=0$, while $\ell$ and $\kappa$ are the prescribed operations of
$\calB$. Hence only the six mixed maps
\begin{gather*}
  \lt:B\times L\to B,\quad \rt:B\times L\to L,\quad
  \prec:B\times L\to B,\quad \succ:B\times L\to L,\\
  \brt:L\times B\to B,\quad \blt:L\times B\to L,
\end{gather*}
remain to be determined. 
Here $(\lt,\rt)$ describe the mixed components of the commutative product,
while $(\prec,\succ)$ and $(\brt,\blt)$ describe the two mixed directions
of the pre-Lie product.

Under this specialization, \ref{PLP0} requires
$
\bigl((L,-\cdot-),(B,\ell),\lt,\rt\bigr)
$
to be a matched pair of commutative algebras, while the remaining mixed
compatibility conditions reduce to those in the following definition. The
pure $B$-conditions are already satisfied since $\calB$ is a prescribed
pre-Lie Poisson algebra.

\begin{Def}\label{Def: PLP matched pair specialization}
Given six bilinear maps
\begin{gather*}
  \lt:B\times L\to B,\qquad \rt:B\times L\to L,\qquad \prec:B\times L\to B,\qquad
  \succ:B\times L\to L,\\
  \brt:L\times B\to B,\qquad\blt:L\times B\to L,
\end{gather*}
the system
$$
\bigl(\calL,\calB,\olOmega(\calL,\calB)\bigr),
\qquad
\olOmega(\calL,\calB)
=(\lt,\rt,\prec,\succ,\brt,\blt),
$$
is called a \textbf{matched pair of pre-Lie Poisson algebras} if
$$
\bigl((L,-\cdot-),(B,\ell),\lt,\rt\bigr)
$$
is a matched pair of commutative algebras in the sense of \cref{Def: commutative matched pair}, and the following conditions hold for all $a,b\in L$ and $x,y\in B$:
\begin{PLPMenum}
\item
$(a\circ b)\blt x-a\circ(b\blt x)-a\blt(b\brt x)
=(b\circ a)\blt x-b\circ(a\blt x)-b\blt(a\brt x)$, and
$(a\circ b)\brt x-a\brt(b\brt x)
=(b\circ a)\brt x-b\brt(a\brt x).$

\item
$(a\blt x)\circ b+(a\brt x)\succ b-a\circ(x\succ b)-a\blt(x\prec b)
=(x\succ a)\circ b+(x\prec a)\succ b-x\succ(a\circ b)$, and
$(a\brt x)\prec b-a\brt(x\prec b)
=(x\prec a)\prec b-x\prec(a\circ b).$

\item
$(a\blt x)\blt y-a\blt\kappa(x,y)
=(x\succ a)\blt y-x\succ(a\blt y)$, and
$(a\blt x)\brt y+\kappa(a\brt x,y)-a\brt\kappa(x,y)
=(x\succ a)\brt y+\kappa(x\prec a,y)-x\prec(a\blt y)-\kappa(x,a\brt y).$

\item
$\kappa(x,y)\succ a-x\succ(y\succ a)
=\kappa(y,x)\succ a-y\succ(x\succ a)$, and
$\kappa(x,y)\prec a-x\prec(y\succ a)-\kappa(x,y\prec a)
=\kappa(y,x)\prec a-y\prec(x\succ a)-\kappa(y,x\prec a).$

\item
$(a\cdot b)\blt x
=a\cdot(b\blt x)+(b\brt x)\rt a$, and
$(a\cdot b)\brt x
=(b\brt x)\lt a.$

\item
$a\cdot(x\succ b)+(x\prec b)\rt a
=(x\rt a)\circ b+(x\lt a)\succ b
=x\rt(a\circ b)$, and
$(x\prec b)\lt a
=(x\lt a)\prec b
=x\lt(a\circ b).$

\item
$\kappa(x,y)\rt a
=(x\rt a)\blt y
=x\rt(a\blt y)$, and
$\kappa(x,y)\lt a
=(x\rt a)\brt y+\kappa(x\lt a,y)
=x\lt(a\blt y)+\ell(x,a\brt y).$

\item
$\ell(x,y)\succ a
=x\rt(y\succ a)$, and
$\ell(x,y)\prec a
=x\lt(y\succ a)+\ell(x,y\prec a).$

\item
$x\succ(a\cdot b)-a\circ(x\rt b)-a\blt(x\lt b)
=(x\succ a)\cdot b+(x\prec a)\rt b-(a\blt x)\cdot b-(a\brt x)\rt b$, and
$x\prec(a\cdot b)-a\brt(x\lt b)
=(x\prec a)\lt b-(a\brt x)\lt b.$

\item
$x\succ(y\rt a)-a\blt\ell(x,y)
=y\rt(x\succ a)-y\rt(a\blt x)$, and
$x\prec(y\rt a)+\kappa(x,y\lt a)-a\brt\ell(x,y)
=y\lt(x\succ a)+\ell(x\prec a,y)-y\lt(a\blt x)-\ell(a\brt x,y).$
\end{PLPMenum}
We denote the set of all such matched pairs by
$$ \MP(\calL,\calB). $$

For a matched pair $\bigl(\calL,\calB,\olOmega(\calL,\calB)\bigr)$,
let
$
\Omega(\calL,B)
=(\lt,\rt,0,\ell,\prec,\succ,\brt,\blt,0,\kappa)
$
be the corresponding pre-Lie Poisson extending structure.
The unified product of $\calL$ and $B$ associated to $\Omega(\calL,B)$ is denoted by
$$
\calL\bowtie_{\olOmega(\calL,\calB)}\calB
\coloneqq
\calL\ltimes_{\Omega(\calL,B)}B
=(L\times B,-*-,-\bullet-)
$$
and is called the \textbf{bicrossed product} associated to the matched pair
$\bigl(\calL,\calB,\olOmega(\calL,\calB)\bigr)$.
Its multiplication and pre-Lie product are given, for all $a,b\in L$ and $x,y\in B$, by
\begin{align*}
(a,x)*(b,y)
&=\bigl(a\cdot b+x\rt b+y\rt a,\,
x\lt b+y\lt a+\ell(x,y)\bigr),\\
(a,x)\bullet(b,y)
&=\bigl(a\circ b+x\succ b+a\blt y,\,
x\prec b+a\brt y+\kappa(x,y)\bigr).
\end{align*}
\end{Def}

Thus the bijection in \cref{Thm: PLP description specialization} restricts
to the following factorization correspondence.

\begin{Thm}\label{Thm: PLP factorization specialization}
There is a bijection
$$
\Fact(\calL,\calB)\longleftrightarrow\MP(\calL,\calB).
$$
\end{Thm}

For later use in the \textbf{CCP}, we record explicitly the matched pair
determined by a fixed factorization.

\begin{Rem}\label{Rem: PLP canonical matched pair specialization}
Let $\calE=(E,-\cdot_E-,-\circ_E-)$ be a pre-Lie Poisson algebra
factorizing through the prescribed pre-Lie Poisson algebras
$\calL$ and $\calB$, with
$E=L\oplus B$. By \cref{Thm: PLP factorization specialization}, this
factorization determines a unique matched pair
$$
\bigl(\calL,\calB,\olOmega^c(\calL,\calB)\bigr),
\qquad
\olOmega^c(\calL,\calB)
=(\lt,\rt,\prec,\succ,\brt,\blt),
$$
called the \textbf{canonical matched pair} associated to the factorization.

More explicitly, let
$
p:E=L\oplus B\to L
$
be the canonical projection. Then, for all $a\in L$ and $x\in B$,
\begin{align*}
  \begin{aligned}
    x\lt a &= x\cdot_E a-p(x\cdot_E a), &
    \qquad x\rt a &= p(x\cdot_E a), \\
    x\prec a &= x\circ_E a-p(x\circ_E a), &
    \qquad x\succ a &= p(x\circ_E a), \\
    a\brt x &= a\circ_E x-p(a\circ_E x), &
    \qquad a\blt x &= p(a\circ_E x).
  \end{aligned}
\end{align*}
\end{Rem}

\subsection{Classifying complements problem}

Here the pre-Lie Poisson algebra
$\calE=(E,-\cdot_E-,-\circ_E-)$ and its subalgebra
$\calL=(L,-\cdot-,-\circ-)$ are fixed, and we describe and classify the
pre-Lie Poisson complements of $\calL$ in $\calE$. Following Section~\ref{sec: Universal Algebra CCP},
write
$
\Comp(\calL,\calE)
$
for the set of all pre-Lie Poisson complements of $\calL$ in $\calE$, and
$
[\calE:\calL]^f
$
for the factorization index.

From now on, assume that $\Comp(\calL,\calE)\neq\varnothing$, and fix a
reference complement
$$
\calB=(B,\ell,\kappa)\in\Comp(\calL,\calE).
$$
Then $E=L\oplus B$, and by
\cref{Rem: PLP canonical matched pair specialization} this factorization
determines the canonical matched pair
$$
\bigl(\calL,\calB,\olOmega^c(\calL,\calB)\bigr),
\qquad
\olOmega^c(\calL,\calB)
=(\lt,\rt,\prec,\succ,\brt,\blt).
$$

Since $E=L\oplus B$, every vector space complement of $L$ in $E$ is
uniquely of the form
$$
B_r\coloneqq\{r(x)+x\mid x\in B\}
$$
for a linear map $r:B\to L$.
By \cref{Prop: commutative graph deformation criterion}, $B_r$ is closed
under $-\cdot_E-$ if and only if
\begin{equation}\label{eq: PLP multiplication deformation}
r\bigl(\ell(x,y)+x\lt r(y)+y\lt r(x)\bigr)
=
r(x)\cdot r(y)+x\rt r(y)+y\rt r(x)
\end{equation}
for all $x,y\in B$.
For the pre-Lie product, the canonical matched pair gives
$$
(r(x)+x)\circ_E(r(y)+y)=(r(x)\circ r(y)+x\succ r(y)+r(x)\blt y)
  +(\kappa(x,y)+x\prec r(y)+r(x)\brt y)
$$
where the two terms are respectively the $L$- and $B$-components.
Hence $B_r$ is closed under $-\circ_E-$ if and only if
\begin{equation}\label{eq: PLP pre-Lie deformation}
r\bigl(\kappa(x,y)+x\prec r(y)+r(x)\brt y\bigr)
=
r(x)\circ r(y)+x\succ r(y)+r(x)\blt y.
\end{equation}

When these conditions hold,
$$
\calB_r
\coloneqq
\left(
B_r,
-\cdot_E-\big|_{B_r\times B_r},
-\circ_E-\big|_{B_r\times B_r}
\right)
$$
is a pre-Lie Poisson complement of $\calL$ in $\calE$.

\begin{Def}\label{Def: PLP deformation map specialization}
A linear map $r:B\to L$ satisfying
\eqref{eq: PLP multiplication deformation} and
\eqref{eq: PLP pre-Lie deformation} is called a \textbf{deformation map} of
the canonical matched pair $(\calL,\calB,\olOmega^c(\calL,\calB))$. The set of all such
deformation maps is denoted by
$$
\mathcal{PDM}(\calL,\calB\mid\olOmega^c(\calL,\calB)).
$$
\end{Def}

The graph description and closure criterion above give the following
bijection.

\begin{Thm}\label{Thm: PLP complement description specialization}
The assignment
$$
\mathcal{PDM}(\calL,\calB\mid\olOmega^c(\calL,\calB))
\longrightarrow\Comp(\calL,\calE),
\qquad r\longmapsto\calB_r,
$$
is a bijection.
\end{Thm}

Thus, the preceding bijection completes the description part of the
\textbf{CCP} for pre-Lie Poisson algebras.

To compare the complements up to isomorphism of pre-Lie Poisson algebras, we transport the pre-Lie
Poisson algebra structure of each $\calB_r$ to the fixed vector space $B$.
For $r\in\mathcal{PDM}(\calL,\calB\mid\olOmega^c(\calL,\calB))$, let
$$
\lambda_r:B\longrightarrow B_r,
\qquad
\lambda_r(x)=r(x)+x,
$$
be the canonical linear isomorphism, and denote by
$\calB^r=(B,\ell^r,\kappa^r)$ the pre-Lie Poisson algebra on $B$
obtained by transporting the algebra structure of $\calB_r$ through
$\lambda_r$. Explicitly,
$$
\ell^r(x,y)\coloneqq\ell(x,y)+x\lt r(y)+y\lt r(x),\qquad
\kappa^r(x,y)\coloneqq\kappa(x,y)+x\prec r(y)+r(x)\brt y.
$$
By construction,
$
\lambda_r:\calB^r\to \calB_r
$
is an isomorphism of pre-Lie Poisson algebras.

\begin{Def}\label{Def: PLP deformation equivalence specialization}
Two deformation maps
$
r_1,r_2\in
\mathcal{PDM}\bigl(\calL,\calB\mid\olOmega^c(\calL,\calB)\bigr)
$
are called \textbf{equivalent}, denoted by $$r_1\sim r_2,$$ if and only if
$\calB^{r_1}\cong\calB^{r_2}$ as pre-Lie Poisson algebras. Equivalently,
there exists a linear isomorphism $\sigma:B\to B$ such that, for all
$x,y\in B$,
$$
\sigma\bigl(\ell^{r_1}(x,y)\bigr)
=\ell^{r_2}\bigl(\sigma(x),\sigma(y)\bigr),\qquad
\sigma\bigl(\kappa^{r_1}(x,y)\bigr)
=\kappa^{r_2}\bigl(\sigma(x),\sigma(y)\bigr).
$$
\end{Def}

The relation $\sim$ is an equivalence relation on
$\mathcal{PDM}\bigl(\calL,\calB\mid\olOmega^c(\calL,\calB)\bigr)$.
We denote the corresponding quotient set by
$$
\mathcal H^2_{ccp}
\bigl(\calL,\calB\mid\olOmega^c(\calL,\calB)\bigr)
\coloneqq
\mathcal{PDM}\bigl(\calL,\calB\mid\olOmega^c(\calL,\calB)\bigr)/\sim.
$$

Applying \cref{Thm:ccp-classification} gives the following classification.

\begin{Thm}\label{Thm: PLP complement classification specialization}
There is a bijection
$$
\mathcal H^2_{ccp}
\bigl(\calL,\calB\mid\olOmega^c(\calL,\calB)\bigr)
\longrightarrow
\Comp(\calL,\calE)/\cong,
\qquad
[r]\longmapsto[\calB_r].
$$
Consequently, the factorization index is
$$
[\calE:\calL]^f
=
\left|
\mathcal H^2_{ccp}
\bigl(\calL,\calB\mid\olOmega^c(\calL,\calB)\bigr)
\right|
=
|\Comp(\calL,\calE)/\cong|.
$$
\end{Thm}

Thus,
\cref{Thm: PLP complement description specialization,Thm: PLP complement classification specialization}
complete the \textbf{CCP} for pre-Lie Poisson algebras.

We conclude this section by recording the compatibility with the transposed Poisson
construction from \cref{Rem: PLP relations specialization}(b).

\begin{Rem}\label{Rem: tP from PLP specialization}
Assume $\operatorname{char}(\bbK)=0$, and let
$\calL=(L,-\cdot-,-\circ-)$ be a pre-Lie Poisson
algebra. By
\cref{Rem: PLP relations specialization}(b), the commutator
$$
[a,b]\coloneqq a\circ b-b\circ a
$$
makes $\calL_{\mathrm{tP}}=(L,-\cdot-,[-,-])$ a transposed
Poisson algebra. Let
$$
\Omega(\calL,V)=(\lt,\rt,f,\ell,\prec,\succ,\brt,\blt,g,\kappa)
$$
be a pre-Lie Poisson extending structure of $\calL$ through $V$. Define
\begin{align*}
  \begin{aligned}
    x\rh a &= x\succ a-a\blt x, &
    \qquad x\lh a &= x\prec a-a\brt x, \\
    \theta(x,y) &= g(x,y)-g(y,x), &
    \qquad \rho(x,y) &= \kappa(x,y)-\kappa(y,x).
  \end{aligned}
\end{align*}
Then
$$
\Omega(\calL_{\mathrm{tP}},V)
\coloneqq(\lt,\rt,f,\ell,\lh,\rh,\theta,\rho)
$$
is a transposed Poisson extending structure of $\calL_{\mathrm{tP}}$
through $V$ in the sense of \cref{sec: transposed Poisson algebras}. Indeed, the
commutator of the pre-Lie product in $\calL\ltimes_{\Omega(\calL,V)}V$ is
\begin{align*}
\{(a,x),(b,y)\}
=\bigl([a,b]+x\rh b-y\rh a+\theta(x,y),\ x\lh b-y\lh a+\rho(x,y)\bigr).
\end{align*}
This is exactly the bracket of the unified product associated to $\Omega(\calL_{\mathrm{tP}},V)$.
Hence taking commutators is compatible with the extending structure theories of this section and \cref{sec: transposed Poisson algebras}
at the level of extending structures and their associated unified products.
\end{Rem}

\textbf{Acknowledgements}
The    authors   were supported by  the National Key R  $\&$ D Program of China (No. 2024YFA1013803)
   and    by Shanghai Key Laboratory of PMMP (No. 22DZ2229014).

 The first author lectured on the topic of this paper in the 2026 Workshop on Rota-Baxter Algebras and Related Topics.  We are grateful to the organizers as well as the audience for giving this opportunity and providing useful remarks.

The results of this paper have been obtained by hand, hence   the authors take full responsibility of the correctness of the results of this paper. AI tools have been used only for checking the details and for revising  of the text.

\textbf{Conflict of Interest}

None of the authors has any conflict of interest in the conceptualization or publication of this
work.

\textbf{Data availability}

Data sharing is not applicable to this article as no new data were created or analyzed in this study.

\bigskip


\end{document}